%% file: ZZpG.tex
\documentclass[a4paper, 11pt]{amsart}

\usepackage{ifthen}
\usepackage{amsfonts}
\usepackage{amssymb}
\usepackage{amsthm}
\usepackage{url}
\usepackage{amscd}
\usepackage{graphicx, xcolor}
\usepackage{comment}
\usepackage[mathscr]{eucal}

\usepackage{listings,jvlisting}
\usepackage{caption}

\newtheorem{thm}{Theorem}[section]

\newtheorem{lem}[thm]{Lemma}

\theoremstyle{definition}
\newtheorem*{ack}{Acknowledgments}

\newtheorem*{brem}{Remark}

\newtheorem*{Tuiki}{Note added on September 22, 2026}

\allowdisplaybreaks[4]

\newcommand{\ZZ}{\mathbb{Z}}
\newcommand{\Q}{\mathbb{Q}}

\newcommand{\Ker}{{\mathrm{Ker}}}

\newcommand{\Gal}{\mbox{\mbox{\rm Gal}}}

\newcommand{\Irr}{\mbox{\mbox{\rm Irr}}}
\newcommand{\Ind}{\mbox{\mbox{\rm Ind}}}

\makeatletter
\@addtoreset{equation}{section}

\makeatother

\usepackage{}	
\begin{document}

\title[]%
{On group rings of the simple group of order 168, 504 or 360 and their modules}
\author[Y.~Konomi]{Yutaka Konomi}
\address{Department of Mathematics, Meijo University, Tempaku-ku, Nagoya, 468-8502, Japan}
\email{konomi@meijo-u.ac.jp}


\begin{abstract}
Let $p$ be a prime and $\ZZ_p$ the ring of $p$-adic integers.
Let $G$ denote the simple group of order 168, 504 or 360.
In this paper, we study the structure of the $\chi$-part of a $\ZZ_p[\mathrm{Im}(\chi)][G]$-module coming from
ideal class groups, Artin $L$-functions and Iwasawa theory. 
After this manuscript was posted, it was pointed out to the author that the main theorem
holds for an arbitrary finite group and follows directly from the structure theorem
for group rings; see the note added after Theorem 1.1.
\end{abstract}

\maketitle

\input{ZZpG_Intro-MainTh.tex}
\input{ZZpG_Intro-Setsumei.tex}
\input{ZZpG_168.tex}
\input{ZZpG_168A.tex}
\input{ZZpG_168B.tex}
\input{ZZpG_168C.tex}
\input{ZZpG_504.tex}
\input{ZZpG_504A.tex}
\input{ZZpG_504B.tex}
\input{ZZpG_504C.tex}
\input{ZZpG_360.tex}
\input{ZZpG_360A.tex}
\input{ZZpG_360B.tex}
\input{ZZpG_360C.tex}
\input{ZZpG_360D.tex}

\begin{ack}
The author would like to thank Yoshichika Iizuka for his many suggestions.
\end{ack}

\input{ZZpG_app.tex}
\input{ZZpG_app168-chi2.tex}
\input{ZZpG_app168-chi4.tex}
\input{ZZpG_app168-chi5+6.tex}
\input{ZZpG_app504.tex}
\input{ZZpG_app360-chi3+2.tex}

\input{ZZpG_app360-chi6.tex}
\input{ZZpG_app360-chi7.tex}
\input{ZZpG_app360-chi4+5.tex}

\input{ZZpG_Output_2-1.tex}
\input{ZZpG_Output_2-2.tex}
\input{ZZpG_Output_3-2.tex}
\input{ZZpG_Output_4-1.tex}
\input{ZZpG_Output_4-2.tex}


\end{document}

%% file: ZZpG_Intro-MainTh.tex

\section{Introduction and the main results}
Let $p$ denote a prime number, $\ZZ_p$ the ring of $p$-adic integers and
$\mathbb{Q}_p$ the field of $p$-adic numbers.
Let $G$ be a finite group
and assume that $p$ does not divide $\# G$.
Fix an embedding $\overline{\Q} \hookrightarrow \overline{\Q_p}$
and regard any character as $p$-adic one.
Let $\Irr(G)$ be the set of all irreducible characters of $G$. 
We denote the degree of a character $\chi$ of $G$ by $\deg(\chi)$.
Putting
$$e_\chi^G = \frac{\chi(1_G)}{\# G}\sum_{g\in G} \chi(g^{-1})g,$$
we have a decomposition into orthogonal idempotents 
$1_{\ZZ_p[\zeta_{k_G}][G]} 
= \sum_{\chi\in\Irr(G)} e_\chi^G$ in $\ZZ_p[\zeta_{k_G}][G]$,
where $k_G$ is an appropriate divisor of $\# G$ and 
$\zeta_{k}$ is a $k$-th root of unity. 
Unless otherwise noted, we deal with left modules.
When $M$ is a $\ZZ_p[\zeta_{k_G}][G]$-module,
we obtain the direct decomposition 
$$M = \bigoplus_{\chi\in\Irr(G)} e_\chi^G M.$$
Let $M^{(n)}$ denote the direct sum of $n$ copies of a module $M$,
that is, $$M^{(n)}=\bigoplus_{i=1}^n M.$$

First, we state the main result.
\begin{thm}\label{MainTh1}
Let $G$ denote the simple group of order $168$, $504$ or $360$. 
Define $k_G$ as follows:
$$k_G = 
\begin{cases}
 21 & \text{if } \# G=168, \\
 63 & \text{if } \# G=504, \\
 60 & \text{if } \# G=360.
\end{cases}$$
For each $\chi \in \Irr(G)$, 
there exist an explicitly determined finite set of primes 
$S_{G,\chi} \supset \{ p : p \mid \# G \}$ 
and a suitable Brauer induction
$$\chi = \sum_{i=1}^m a_i \Ind_{H_i}^G(\psi_i)$$
such that for every $p \notin S_{G,\chi}$ and 
every $\mathbb{Z}_p[\zeta_{k_G}][G]$-module $M$, 
there exists a $\mathbb{Z}_p$-module $N$ satisfying 
the following two conditions:
\begin{enumerate}
\setlength{\itemsep}{5pt}
\item $e_\chi^G M \simeq N^{(\deg\chi)}$ as $\mathbb{Z}_p$-modules.
\item If $M$ is finite, then
$$\# N = \prod_{i=1}^m \left(\# e_{\psi_i}^{H_i} M\right)^{a_i}.$$
\end{enumerate}
Here $H_i$ is a subgroup of $G$, 
$\psi_i$ is a character of $H_i$ of degree 1,
and $a_i$ is a rational integer.
\end{thm}

\begin{Tuiki}
After this manuscript was posted on arXiv, it was pointed out to the author
that Theorem \ref{MainTh1} holds for an arbitrary finite group $G$,
for an arbitrary Brauer induction $\chi=\sum_i a_i \Ind_{H_i}^G(\psi_i)$,
and with $S_{G,\chi}=\{p : p \mid \# G\}$,
and that it follows directly from the structure theorem for the group ring
$\mathbb{Z}_p[\zeta][G]$ with $p\nmid \#G$.
In particular, the prime $29$ in Lemma \ref{168-3_Lem} (6) is not exceptional,
the hypothesis $p\notin\bigcup_{\chi} S_{G,\chi}$ in Theorem \ref{OLD}
may be replaced by $p\nmid \#G$,
and the case-by-case computations in Sections 2--4 are not needed
for the proof of Theorem \ref{MainTh1}.
For the reader's convenience, we sketch the argument.
Let $\mathcal{O}=\mathbb{Z}_p[\zeta_k]$, where $p\nmid k$ and
$\mathbb{Q}_p(\zeta_k)$ is a splitting field for $G$
(e.g. $k$ is the exponent of $G$).
Since $\# G$ is a unit in $\mathcal{O}$,
$\mathcal{O}[G]$ is a maximal order in
$\mathbb{Q}_p(\zeta_k)[G]\simeq \prod_{\chi\in\Irr(G)} M_{\deg\chi}(\mathbb{Q}_p(\zeta_k))$,
and hence $e_\chi^G\mathcal{O}[G]\simeq M_{\deg\chi}(\mathcal{O})$
for every $\chi\in\Irr(G)$ (see e.g. \cite[Theorems 17.3 and 41.1]{Reiner}).
By Morita theory, for every $\mathcal{O}[G]$-module $M$ there exists
an $\mathcal{O}$-module $N_\chi$ such that
$e_\chi^G M\simeq L_\chi\otimes_{\mathcal{O}} N_\chi\simeq N_\chi^{(\deg\chi)}$,
where $L_\chi\simeq\mathcal{O}^{\deg\chi}$ is an $\mathcal{O}[G]$-lattice affording $\chi$;
this gives (1).
Let $H$ be a subgroup of $G$ and $\psi$ a linear character of $H$
with values in $\mathcal{O}$.
Then $e_\psi^H M=\bigoplus_{\chi} e_\psi^H e_\chi^G M$ and
$$e_\psi^H e_\chi^G M\simeq (e_\psi^H L_\chi)\otimes_{\mathcal{O}} N_\chi
\simeq N_\chi^{(\langle \chi,\,\Ind_H^G\psi\rangle_G)},$$
since $e_\psi^H L_\chi$ is an $\mathcal{O}$-direct summand of $L_\chi$
of rank $\langle \mathrm{Res}^G_H\chi,\psi\rangle_H=\langle \chi,\Ind_H^G\psi\rangle_G$
by Frobenius reciprocity.
Hence, if $M$ is finite,
$$\# e_\psi^H M=\prod_{\chi\in\Irr(G)} (\# N_\chi)^{\langle \chi,\,\Ind_H^G\psi\rangle_G},$$
and for any $\chi=\sum_i a_i \Ind_{H_i}^G(\psi_i)$ with $a_i\in\mathbb{Z}$ we obtain
$$\prod_i \bigl(\# e_{\psi_i}^{H_i} M\bigr)^{a_i}
=\prod_{\chi'\in\Irr(G)} (\# N_{\chi'})^{\sum_i a_i\langle \chi',\,\Ind_{H_i}^G\psi_i\rangle_G}
=\prod_{\chi'\in\Irr(G)} (\# N_{\chi'})^{\langle \chi',\chi\rangle_G}
=\# N_\chi,$$
which is (2).
The explicit computations in Sections 2--5 are retained in this version for the record.
\end{Tuiki}

For linear characters, Theorem \ref{MainTh1} holds trivially.
In Sections 2--4, Theorem \ref{MainTh1} is proved for each $\chi$ separately
by using a computer to solve a system of $\# G$ linear equations over $\mathbb{Q}(\zeta_k)$ to find explicit elements of $\mathbb{Z}_p[\zeta_k][G]$ with the desired properties (see Remark in Section 2.1).
Since we need to regard the solutions to this linear system over $\mathbb{Q}(\zeta_k)$ as elements of $\mathbb{Z}_p[\zeta_k]$, 
the exceptional set $S_{G,\chi}$ arises in this approach;
as explained in the note above, it is an artifact of the method and
one may always take $S_{G,\chi}=\{p : p\mid \# G\}$.
We list the commands in Section 5.

We obtain the following Theorem \ref{OLD} from Theorem \ref{MainTh1}.
Let $F_+$ and $K_+$ be totally real number fields.
Suppose that $K_+$ is a finite Galois extension over $F_+$. 
Put $K=K_+(\sqrt{-r})$ and $F=F_+(\sqrt{-r})$
for a totally positive algebraic number $r\in F_+$.
Let $\chi$ be a character of $\Gal(K/F_+)$ and 
$L(s,\chi, K/F_+)$ the Artin $L$-function attached to $\chi$.
We denote the $p$-part of 
the ideal class group of $K$ and the $p$-part of the roots of unity in $K$
by $A_K$ and $U_K$, respectively.
Put $\mathcal{G}=\Gal(K/F_+)$ and 
assume that $p$ does not divide $\# \mathcal{G}$.
Suppose that $\mathrm{Im}(\chi)\subset\mathbb{Z}_p$.
Then we are able to define 
$A_K^\chi := e_\chi^\mathcal{G} A_K$ and $U_K^\chi := e_\chi^\mathcal{G} U_K$. 

\begin{thm}\label{OLD}
Let $K/F$ be a Galois extension whose Galois group $G$ is isomorphic to 
the simple group of order $168$, $504$, or $360$.
Suppose that 
$p\notin \bigcup_{\chi\in\mathrm{Irr}(G)} S_{G, \chi}$.
For simplicity, assume that $\zeta_{k_G} \in \mathbb{Z}_p$.  
For the direct sum decomposition
$$A_K = \bigoplus_{\chi\in\mathrm{Irr}(\mathcal{G})} A_K^\chi,$$
the following two conditions hold:
\begin{enumerate}
\setlength{\itemsep}{5pt}
\item For each $\chi$,
there exists an abelian $p$-group $N$ such that
$A_K^{\chi} \simeq N^{(\mathrm{deg} \chi)}$.

\item 
If $\chi_-$ is an odd character and satisfies 
the technical assumption derived from \cite[Theorem 3]{Wiles}, then 
$$
\mathrm{ord}_p\left(\# A_K^{\chi_-}\right) 
-\mathrm{ord}_p\left(\# U_K^{\chi_-}\right)
~=~
\mathrm{deg}(\overline{\chi_-})\cdot\mathrm{ord}_p \left( L(0,\overline{\chi_-}, K/F_+)\right).
$$
\end{enumerate}
\end{thm}

The method of proving Theorem \ref{OLD} (2) from Theorem \ref{MainTh1} is similar to the method in \cite[Section 7]{YK}. 

%% file: ZZpG_Intro-Setsumei.tex

We state the background of Theorem \ref{MainTh1} and Theorem \ref{OLD}. 
When $K/F$ is a finite Galois extension of number fields, 
we put 
$$\mathrm{Amb}_p(K/F)=
\{x\in A_K\mid\forall\sigma\in\mathrm{Gal}(K/F), \sigma x =x\}$$
and
$$a_p(K/F)=\#\mathrm{Amb}_p(K/F).$$
We will not assume that $p \nmid \#\Gal(K/F)$ for the time being.
In \cite[Theorem 5]{Ohta}, 
it is proved that
if $\# A_K > a_p(K/K^{A_n})$, then
$$p\textrm{-rank} \bigl( A_K/\mathrm{Amb}_p(K/F) \bigr) \ge 3,$$
where
$p$ is an odd prime, $n\ge 5$ and $\Gal(K/F)\simeq S_n$.
In \cite[Theorem 2]{YK0}, 
we essentially prove that if $\# A_K > a_p(K/F)$, 
then $$p\textrm{-rank} \bigl( A_K/\mathrm{Amb}_p(K/F) \bigr) \ge l_n +1$$
in the case of $\Gal(K/F)\simeq A_n$, 
improving on the above result,
where $n\ge 5$ and $l_n$ is the largest prime number satisfying $l_n\neq p$ and $l_n \le\sqrt{n}$.
Furthermore, in \cite[Theorem 1.2]{YK1}, 
we improve on \cite[Theorem 2]{YK0} and obtain a lower bound on 
$p\textrm{-rank} \bigl( A_K/\mathrm{Amb}_p(K/F) \bigr)$
when $\Gal(K/F)$ is simple and non-abelian.
In \cite[Corollary 1.2, Theorem 1.4]{YK},
we show that, if $\Gal(K/F)\simeq A_5$ and $p\nmid \#\Gal(K/F)$,
then we determine the structure of $A_K$ more precisely,
and furthermore, by using \cite[Theorem 3]{Wiles}, 
we show that the order of $A_K$ and
the special value of a certain Artin $L$-function at zero are connected
when $K$ is a CM field arising from an $A_5$-extension.
The aim of this paper is to show that 
even when $\Gal(K/F)$ is isomorphic to the simple group of order 168, 504 or 360, 
the structure of $A_K^{\chi}$ can be precisely known, 
and $A_K^{\chi_-}$ and $L(0,\overline{\chi_-},K/F_+)$ are connected 
in the same context as in \cite{YK}.

%% file: ZZpG_168.tex

\section{The simple group of order 168}
Let $G$ be the simple group of order 168.
It is known that 
$$G \simeq \langle(3,6,7)(4,5,8), (1,8,2)(4,5,6)\rangle$$
and the character table of $G$ is given by Table \ref{Ch168}.
\begin{table}[htbp]
\begin{center}
\caption{The character table of $G$}\label{Ch168}
\begin{tabular}{l|cccccc}
\multicolumn{1}{c|}{} & $1_G$ & $c_2$ & $c_3$ & $c_4$ & $c_5$ & $c_6$ \\ \hline 
$\chi_1^G$  & $1$ & $1$  & $1$  & $1$  & $1$ & $1$   \\
$\chi_2^G$  & $3$ & $-1$ & $1$  & $0$  & $\omega_7$ & $\overline{\omega_7}$  \\
$\chi_3^G$  & $3$ & $-1$ & $1$  & $0$  & $\overline{\omega_7}$ & $\omega_7$  \\
$\chi_4^G$  & $6$ & $2$  & $0$  & $0$  & $-1$  & $-1$ \\
$\chi_5^G$  & $7$ & $-1$ & $-1$ & $1$  & $0$ & $0$ \\
$\chi_6^G$  & $8$ & $0$  & $0$  & $-1$ & $1$ & $1$
\end{tabular}
\end{center}
\end{table}

\noindent
Here, 
$c_2=(1,8)(2,4)(3,5)(6,7)$, $c_3=(1,4,8,2)(3,6,5,7)$, $c_4=(1,8,2)(4,5,6)$, 
$c_5=(1, 5, 3, 7, 6, 8, 2)$, $c_6=(1, 7, 2, 3, 8, 5, 6)$ 
and $\omega_7 =(-1+\sqrt{-7})/2$.
For simplicity, we will denote $e_{\chi_i^G}^G$ by $e_i^G$. 

%% file: ZZpG_168A.tex

\subsection{The case of the characters of degree 3}
Since $\chi_2^G$ and $\chi_3^G$ can be handled similarly,
we focus on the character $\chi_2^G$.
Put $a=(1, 8, 2, 4, 3, 7, 5)$ and $C_7=\langle a\rangle$.
Let $\chi_i^{C_7}$ be the characters of $C_7$ 
that satisfy $\chi_i^{C_7}(a)=\zeta_7^i$ for $i\in\{1,2,4\}$. 
Put $r=(1, 7, 3, 8)(2, 6, 4, 5)$, $s=(1, 2)(3, 4)(5, 7)(6, 8)$ and
$D_4 = \langle r, s\rangle$. 
Then $D_4$ is isomorphic to the dihedral group of order 8. 
Let $\eta$ be the character of $D_4$ 
that satisfies $\eta(r)=\eta(s)= -1$.
We can get the following lemma by using a computer 
(see Section \ref{168-3_Com}).

\begin{lem}\label{168-3_Lem}
Assume that $p\not\in\{2,3,7\}$.
\begin{enumerate}
\setlength{\itemsep}{5pt}
\item $\chi_2^G = \Ind_{C_7}^G(\chi_i^{C_7})-\Ind_{D_4}^G(\eta)$
for all $i \in\{1,2,4\}$. 

\item $e_2^G(e_1^{C_7}+e_2^{C_7}+e_4^{C_7}) = e_2^G.$

\item We have
$$e_4^{C_7} = h^{-1}e_2^{C_7}h
\quad\text{and}\quad
e_1^{C_7} = h e_2^{C_7} h^{-1}$$
by putting $h=(1,7,8)(3,5,4)$.

\item There exist $q_1\in\mathbb{Z}[1/2, \zeta_7][G]$ and 
$q_2 \in\mathbb{Z}[1/7, 1/29, \zeta_7][G]$ 
such that
$$(1-e_2^G)e_1^{C_7}q_1 = e_1^{C_7} e_\eta^{D_4}$$
and
$$(1-e_2^G)e_1^{C_7} = e_1^{C_7} e_\eta^{D_4}q_2,$$
that is, 
$(1-e_2^G)e_1^{C_7} \mathbb{Z}_p[\zeta_7][G] = e_1^{C_7} e_\eta^{D_4} \mathbb{Z}_p[\zeta_7][G]$
if $p\neq 29$.

\item There exist $q_3\in\mathbb{Z}[1/2, 1/29, \zeta_7][G]$ and 
$q_4\in\mathbb{Z}[1/7, \zeta_7][G]$
such that
$$q_3 e_1^{C_7} e_\eta^{D_4} = e_\eta^{D_4}$$
and
$$e_1^{C_7} e_\eta^{D_4} = q_4 e_\eta^{D_4},$$ 
that is, 
$\mathbb{Z}_p[\zeta_7][G] e_1^{C_7} e_\eta^{D_4} = \mathbb{Z}_p[\zeta_7][G] e_\eta^{D_4}$
if $p\neq 29$.

\item $S_{G,\chi_2^G}=\{2,3,7,29\}$.
\end{enumerate}
\end{lem}

\begin{brem}
We can obtain the above $q_2$ by the following procedure: 
\begin{enumerate}
\renewcommand{\labelenumi}{(\roman{enumi})}
\setlength{\itemsep}{5pt}
\item[\text{Step} 1:] Consider the following relational equation 
$$e_1^{C_7} e_\eta^{D_4}\cdot\bigl(\sum_{g\in G} T_g g\bigr) = (1-e_2^G)e_1^{C_7}$$
with $T_g ~(g\in G)$ as  variables.

\item[\text{Step} 2:] We obtain the linear system of 
$\mathbb{Q}(\zeta_7)$-coefficients and $T_g$-variables
by comparing the coefficient of $g$ in the formula of Step 1.  

\item[\text{Step} 3:] Find a solution $\{t_g\}_{g\in G}$ to 
the linear system of Step 2 by using a computer.

\item[\text{Step} 4:] We have $e_1^{C_7} e_\eta^{D_4}\cdot q_2 = (1-e_2^G)e_1^{C_7}$ 
by putting $q_2=\sum_{g\in G} t_g g$ and 
$q_2 \in\mathbb{Z}[1/7, 1/29, \zeta_7][G]$ 
by checking the concrete value of $t_g$.
\end{enumerate}
\end{brem}
Note that Step 4 yields the elements of $S_{G,\chi_2^G}$.
The prime $29$ appears only because the particular solution $q_2$ found
by the computer has $29$ in its denominators;
by the note added after Theorem \ref{MainTh1}, $29$ is not exceptional.
Now, we will prove Theorem \ref{MainTh1} in the case of $\chi = \chi_2^G$.
Let $p$ be a prime such that $p$ does not divide $\# G=168$,
and let $M$ be a $\mathbb{Z}_p[\zeta_7][G]$-module.
By Lemma \ref{168-3_Lem} (2), the homomorphism
$$
f_{e_2^G}: e_1^{C_7}M\oplus e_2^{C_7}M\oplus e_4^{C_7}M \longrightarrow e_2^G M,~
x \longmapsto e_2^Gx
$$
is surjective.
Since $e_2^G$ is a central idempotent of $\mathbb{Z}_p[\zeta_7][G]$, 
we obtain
$$\Ker(f_{e_2^G})=\bigoplus_{i\in\{1,2,4\}} 
\left(1^G -e_2^G\right)e_i^{C_7} M,$$
that is,
$$e_2^G M \simeq \bigoplus_{i\in\{1,2,4\}} 
e_i^{C_7} M /\left(1^G -e_2^G\right)e_i^{C_7} M$$
holds.
Put $N=e_1^{C_7} M /\left(1^G -e_2^G\right)e_1^{C_7} M$.
We see that
$$N\simeq e_i^{C_7} M /\left(1^G -e_2^G\right)e_i^{C_7} M$$
for $i\in\{2,4\}$ by Lemma \ref{168-3_Lem} (3).
Thus, $e_2^G M \simeq N^{(3)}$ is proved.

If $p\neq 29$,
then we have $(1-e_2^G)e_1^{C_7} M = e_1^{C_7} e_\eta^{D_4} M$
by Lemma \ref{168-3_Lem} (4).
By Lemma \ref{168-3_Lem} (5), the inverse map of 
$$
f_{q_3}: e_1^{C_7} e_\eta^{D_4} M\longrightarrow e_\eta^{D_4} M,~
x \longmapsto q_3 x
$$
is given by
$$
f_{q_4}: e_\eta^{D_4} M\longrightarrow e_1^{C_7} e_\eta^{D_4} M,~
x \longmapsto q_4 x.
$$
Consequently, $\# N = \#(e_1^{C_7}M) /\#(e_\eta^{D_4}M)$ holds
if $\# M < \infty$.

%% file: ZZpG_168B.tex

\subsection{The case of the character of degree 6}
Put
$$A_{4,(7)} = \langle(1, 7, 6)(2, 5, 4),~(1, 8)(2, 4)(3, 5)(6, 7)\rangle$$
and
$$H = \langle(1, 4, 2)(5, 8, 7),~(1, 8, 2, 4, 3, 7, 5)\rangle.$$
Then $A_{4,(7)}$ is isomorphic to the alternating group of degree $4$. 
Also, $H$ is isomorphic to the semidirect product of a cyclic group of order $7$ and one of order $3$. 
Since $A_{4,(7)}$ has exactly seven conjugates in $G$, let $A_{4,(1)}, \ldots, A_{4,(6)}$ denote the other six conjugates.
In particular, put
$$A_{4,(1)} = \langle(2, 6, 4)(3, 5, 8),~(1, 2)(3, 5)(4, 6)(7, 8)\rangle.$$
Let $\varepsilon_i$ and $\varepsilon'$ 
be the trivial characters of $A_{4,(i)}$ and $H$, respectively.
We can get the following lemma by using a computer 
(see Section \ref{168-6_Com}). 

\begin{lem}\label{168-6_Lem}
Assume that $p\not\in\{2,3,7\}.$
\begin{enumerate}
\setlength{\itemsep}{5pt}
\item $\chi_4^G = \Ind_{A_{4,(1)}}^G(\varepsilon_1)-\Ind_H^G(\varepsilon')$.

\item There exists $r_0\in\mathbb{Z}[1/2, 1/3, 1/7][G]$ such that
$$e_4^G =\bigl( \sum_{i=1}^6 e_{\varepsilon_i}^{A_{4,(i)}} \bigr) r_0.$$

\item For any $j\in \{1,2,\ldots,6\}$, 
there exist $r_j \in e_{\varepsilon_j}^{A_{4,(j)}}\mathbb{Z}[G]$ 
and $r'_j \in \mathbb{Z}[1/2, 1/7][G]$
such that
$$ e_4^G e_{\varepsilon_j}^{A_{4,(j)}} = 
r'_j r_j e_4^G e_{\varepsilon_j}^{A_{4,(j)}}
$$
and
$$
r_{j}e_{\varepsilon_{i}}^{A_{4,(i)}}=0
$$
for all $i\in\{1,2,\ldots,6\}\setminus\{j\}$.

\item For all $i\in \{1,2,\ldots,6\}$, there exists $g_i \in G$ such that
$$e_{\varepsilon_1}^{A_{4,(1)}} = g_i\,e_{\varepsilon_i}^{A_{4,(i)}} g_i^{-1}.$$

\item There exist $q_1\in\mathbb{Z}[1/7][G]$ and $q_2\in\mathbb{Z}[1/2][G]$
such that
$$\left(1-e_4^G\right)e_{\varepsilon_1}^{A_{4,(1)}} q_1 = 
e_{\varepsilon_1}^{A_{4,(1)}}e_{\varepsilon'}^{H}
$$
and
$$\left(1-e_4^G\right)e_{\varepsilon_1}^{A_{4,(1)}} = 
e_{\varepsilon_1}^{A_{4,(1)}}e_{\varepsilon'}^{H} q_2,
$$
that is, 
$\left(1-e_4^G\right)e_{\varepsilon_1}^{A_{4,(1)}}\,\mathbb{Z}_p[G] 
=
e_{\varepsilon_1}^{A_{4,(1)}}e_{\varepsilon'}^{H}\,\mathbb{Z}_p[G].$

\item 
There exist $q_3\in\mathbb{Z}[G]$ and $q_4\in\mathbb{Z}[1/2][G]$
such that
$$q_3 e_{\varepsilon_1}^{A_{4,(1)}}e_{\varepsilon'}^{H}= e_{\varepsilon'}^{H}$$
and
$$e_{\varepsilon_1}^{A_{4,(1)}}e_{\varepsilon'}^{H}= q_4 e_{\varepsilon'}^{H},$$
that is,
$\mathbb{Z}_p[G]\,
e_{\varepsilon_1}^{A_{4,(1)}}e_{\varepsilon'}^{H}
=
\mathbb{Z}_p[G]\,e_{\varepsilon'}^{H}$. 

\item $S_{G,\chi_4^G}=\{2,3,7\}$.
\end{enumerate}
\end{lem}

Now, we will prove Theorem \ref{MainTh1} in the case of $\chi = \chi_4^G$.
Let $M$ be a $\mathbb{Z}_p[G]$-module.
The following homomorphism
$$f_{e_4^G}:
\prod_{i=1}^6 e_{\varepsilon_i}^{A_{4,(i)}} M
\longrightarrow e_4^G M,~~
(x_i)_i \longmapsto e_4^G\,\bigl(\sum_{i=1}^6 x_i\bigr)
$$
is surjective by Lemma \ref{168-6_Lem} (2).
We will show that 
$$\Ker(f_{e_4^G})=
\prod_{i=1}^6 (1-e_4^G) e_{\varepsilon_i}^{A_{4,(i)}} M.
$$
If we take 
$\bigl(e_{\varepsilon_i}^{A_{4,(i)}} \xi_i \bigr)_{1\le i\le 6}$ in
$\Ker(f_{e_4^G})$,
then 
$$e_4^G\,\bigl(\sum_{i=1}^6 e_{\varepsilon_i}^{A_{4,(i)}} \xi_i\bigr)=0$$
holds.
Note that $e_4^G$ is a central idempotent of $\mathbb{Z}_p[G]$.
We get
\begin{eqnarray*}
0 =
r_1 e_4^G \bigl(\sum_{i=1}^6 e_{\varepsilon_i}^{A_{4,(i)}} \xi_i \bigr)=
e_4^G r_1 \bigl(\sum_{i=1}^6 e_{\varepsilon_i}^{A_{4,(i)}} \xi_i \bigr)=
r_1 e_4^G e_{\varepsilon_1}^{A_{4,(1)}}\xi_1
\end{eqnarray*}
and
$$0=r'_1 r_1 e_4^G e_{\varepsilon_1}^{A_{4,(1)}}\xi_1
= e_4^G e_{\varepsilon_1}^{A_{4,(1)}}\xi_1$$
by Lemma \ref{168-6_Lem} (3).
Thus, we see
$$e_{\varepsilon_1}^{A_{4,(1)}}\xi_1=
\left(1-e_4^G\right)e_{\varepsilon_1}^{A_{4,(1)}}\xi_1
+e_4^G e_{\varepsilon_1}^{A_{4,(1)}}\xi_1 =
\left(1-e_4^G\right)e_{\varepsilon_1}^{A_{4,(1)}}\xi_1.
$$
Similarly, we have
$$e_{\varepsilon_i}^{A_{4,(i)}}\xi_i=
\left(1-e_4^G\right)e_{\varepsilon_i}^{A_{4,(i)}}\xi_i
$$
for $i\in\{2,\ldots,6\}.$
Thus, $$\bigl(e_{\varepsilon_i}^{A_{4,(i)}} \xi_i \bigr)_{1\le i\le 6} \in
\prod_{i=1}^6 (1-e_4^G) e_{\varepsilon_i}^{A_{4,(i)}} M$$ 
holds.
It is easy to see the reverse inclusion.

By the fundamental homomorphism theorem,
it follows that
$$e_4^G M \simeq
\prod_{i=1}^6 \left(
e_{\varepsilon_i}^{A_{4,(i)}} M/(1-e_4^G) e_{\varepsilon_i}^{A_{4,(i)}} M
\right).$$ 
Put
$$N=
e_{\varepsilon_1}^{A_{4,(1)}} M/(1-e_4^G) e_{\varepsilon_1}^{A_{4,(1)}} M.
$$
From Lemma \ref{168-6_Lem} (4), we obtain
$$N \simeq
e_{\varepsilon_i}^{A_{4,(i)}} M/(1-e_4^G) e_{\varepsilon_i}^{A_{4,(i)}} M
$$
for $i\in\{1,2,\ldots,6\}$.
This completes the proof of Theorem \ref{MainTh1} $(1)$ 
in the case of $\chi=\chi_4^G$ by Lemma \ref{168-6_Lem} (4).

We have 
$(1-e_4^G) e_{\varepsilon_1}^{A_{4,(1)}} M =
e_{\varepsilon_1}^{A_{4,(1)}} e_{\varepsilon'}^{H} M$
by Lemma \ref{168-6_Lem} (5).
The inverse map of
$$f_{q_3}: e_{\varepsilon_1}^{A_{4,(1)}} e_{\varepsilon'}^{H} M
\longrightarrow
e_{\varepsilon'}^{H} M, ~ x \longmapsto q_3 x 
$$
is
$$f_{q_4}: e_{\varepsilon'}^{H} M\longrightarrow
e_{\varepsilon_1}^{A_{4,(1)}} e_{\varepsilon'}^{H} M,~
y \longmapsto q_4 y$$
by Lemma \ref{168-6_Lem} (6).
If $\# M<\infty$, then we obtain
$$\# N
=\frac{\#{e_{\varepsilon_1}^{A_{4,(1)}} M}}{\# e_{\varepsilon'}^{H} M}.
$$

%% file: ZZpG_168C.tex

\subsection{The case of the character of degree 7}
Put
$$S_{4,(1)}=\langle(2, 6, 4)(3, 5, 8),~(1, 5)(2, 8)(3, 6)(4, 7)\rangle.$$
Then $S_{4,(1)}$ is isomorphic to the symmetric group of degree 4.
Since $S_{4,(1)}$ has exactly seven conjugates in $G$, 
let $S_{4,(2)}$, $\ldots$, $S_{4,(7)}$ be all of the other conjugates of $S_{4,(1)}$ in $G$. 
For $i\in\{1,\ldots,7\}$,
let $\eta_i$ be the non-trivial character of $S_{4,(i)}$ 
with $\deg(\eta_i)=1$.
With the aid of a computer, we determine $S_{G,\chi_5^G}$ and establish the following lemma
(see Section \ref{168-7_Com}). 

\begin{lem}\label{168-7_Lem}
Assume that $p\not\in\{2,3,7\}.$
\begin{enumerate}
\setlength{\itemsep}{5pt}
\item $\chi_5^G = \Ind_{S_{4,(1)}}^G\left(\eta_1 \right)$.

\item $\displaystyle e_5^G =\sum_{i=1}^7 e_{\eta_i}^{S_{4,(i)}}$.

\item $e_{\eta_i}^{S_{4,(i)}} e_{\eta_j}^{S_{4,(j)}} 
= \delta_{ij} e_{\eta_i}^{S_{4,(i)}}$, 
where $\delta_{ij}$ is the Kronecker delta.

\item For all $i\in \{1,2,\ldots,7\}$, there exists $g_i \in G$ such that
$$e_{\eta_1}^{S_{4,(1)}} = g_i\,e_{\eta_i}^{S_{4,(i)}} g_i^{-1}.$$

\item $S_{G,\chi_5^G}=\{2,3,7\}$.
\end{enumerate}
\end{lem}

It follows from Lemma \ref{168-7_Lem} that 
$$e_5^G M = \bigoplus_{i=1}^7 e_{\eta_i}^{S_{4,(i)}} M$$
and
$$
e_{\eta_i}^{S_{4,(i)}} M \simeq e_{\eta_j}^{S_{4,(j)}} M
$$
for all $i, j\in\{1, 2, \ldots,\,7\}$.
Thus, Theorem \ref{MainTh1} holds in the case of $\chi=\chi_5^G$.

\subsection{The case of the character of degree 8}
Put
$$
H_1=\langle(1, 4, 2)(5, 8, 7),~(1, 8, 2, 4, 3, 7, 5)\rangle,
$$
$$
H_2=\langle(1, 4, 7)(2, 8, 6),~(1, 7, 5, 6, 2, 4, 8)\rangle,
$$
$$
H_3=\langle(1, 5, 2)(3, 8, 4),~(1, 4, 2, 5, 6, 3, 8)\rangle,
$$
$$
H_4=\langle(1, 7, 4)(2, 6, 8),~(1, 8, 6, 3, 7, 4, 2)\rangle,
$$
$$
H_5=\langle(1, 8, 5)(2, 7, 6),~(1, 5, 3, 7, 6, 8, 2)\rangle,
$$
$$
H_6=\langle(1, 5, 4)(2, 6, 3),~(1, 4, 7, 6, 3, 5, 2)\rangle,
$$
$$
H_7=\langle(1, 7, 5)(3, 6, 8),~(1, 8, 5, 7, 4, 3, 6)\rangle,
$$
and
$$
H_8=\langle(2, 3, 4)(6, 8, 7),~(2, 6, 7, 5, 3, 4, 8)\rangle.
$$
The subgroup $H_1$ is equal to $H$ that appeared in Section 2.2
and these are all conjugate to $H_1$ in $G$. 
Let $\eta_i$ be the non-trivial character of $H_i$ 
with $\deg(\eta_i)=1$ and $\eta_i(t_i)=\zeta_3^2$,
where $t_i$ is the generator of order 3 given in the definition of $H_i$.
We can get the following lemma by using a computer
(see Section \ref{168-8_Com}).

\begin{lem}\label{168-8_Lem}
Assume that $p\neq 2,3$ and $7$.
\begin{enumerate}
\setlength{\itemsep}{5pt}
\item $\chi_6^G = \Ind_{H_i}^G\left(\eta_i \right)$.

\item $\displaystyle e_6^G =\sum_{i=1}^8 e_{\eta_i}^{H_i}$.

\item $e_{\eta_i}^{H_i} e_{\eta_j}^{H_j} 
= \delta_{ij} e_{\eta_i}^{H_i}$, 
where $\delta_{ij}$ is the Kronecker delta.

\item For all $i\in \{1,2,\ldots,8\}$, there exists $g_i \in G$ such that
$$e_{\eta_1}^{H_1} = g_i\,e_{\eta_i}^{H_i} g_i^{-1}.$$

\item $S_{G,\chi_6^G}=\{2,3,7\}$.
\end{enumerate}
\end{lem}
The proof of Theorem \ref{MainTh1} in the case of $\chi=\chi_6^G$
is similar to the argument in Section 2.3.

%% file: ZZpG_504.tex

\section{The simple group of order 504}
Let $G$ be the simple group of order 504.
Let $\mathbb{F}_8$ denote the finite field of order 8 with a primitive element $a$ satisfying $a^3=a+1$.
It is known that 
$$G \simeq 
\langle
\begin{pmatrix}
0 &1 \\
1 &1
\end{pmatrix},~
\begin{pmatrix}
a^3 &a \\
a^4 &a
\end{pmatrix}
\rangle
$$
and the character table of $G$ is given in Table \ref{Ch504}.
\vskip1mm
\begin{table}[htbp]
\begin{center}
\caption{The character table of $G$}\label{Ch504}
\begin{tabular}{l|ccccccccc}
\multicolumn{1}{c|}{} &$1_G$ &$c_2$ &$c_3$ &$c_4$ &$c_5$ &$c_6$ &$c_7$ &$c_8$ &$c_9$ \\ \hline 
$\chi_1^G$  &$1$ &$1$ &$1$ &$1$ &$1$ &$1$ &$1$ &$1$ &$1$ \\
$\chi_2^G$  &$7$ &$-1$ &$-2$ &$0$ &$0$ &$0$ &$1$ &$1$ &$1$ \\
$\chi_3^G$  &$7$ &$-1$ &$1$ &$0$ &$0$ &$0$ &$-w_9(4)$ &$-w_9(2)$ &$-w_9(1)$ \\
$\chi_4^G$  &$7$ &$-1$ &$1$ &$0$ &$0$ &$0$ &$-w_9(2)$ &$-w_9(1)$ &$-w_9(4)$ \\
$\chi_5^G$  &$7$ &$-1$ &$1$ &$0$ &$0$ &$0$ &$-w_9(1)$ &$-w_9(4)$ &$-w_9(2)$ \\
$\chi_6^G$  &$8$ &$0$ &$-1$ &$1$ &$1$ &$1$ &$-1$ &$-1$ &$-1$ \\
$\chi_7^G$  &$9$ &$1$ &$0$ &$w_7(2)$ &$w_7(1)$ &$w_7(3)$ &$0$ &$0$ &$0$ \\
$\chi_8^G$  &$9$ &$1$ &$0$ &$w_7(1)$ &$w_7(3)$ &$w_7(2)$ &$0$ &$0$ &$0$ \\
$\chi_9^G$  &$9$ &$1$ &$0$ &$w_7(3)$ &$w_7(2)$ &$w_7(1)$ &$0$ &$0$ &$0$ \\
\end{tabular}
\end{center}
\end{table}

\noindent
Here,
$c_2=\begin{pmatrix}
1 & 1 \\
0 & 1  
\end{pmatrix}$,
$c_3=\begin{pmatrix}
0 & 1 \\
1 & 1  
\end{pmatrix}$,
$c_4=\begin{pmatrix}
a^2 & 0 \\
0 & a^5  
\end{pmatrix}$,
$c_5=\begin{pmatrix}
a & 0 \\
0 & a^6  
\end{pmatrix}$,
$c_6=\begin{pmatrix}
a^3 & 0 \\
0 & a^4  
\end{pmatrix}$,
$c_7=\begin{pmatrix}
0 & 1 \\
1 & a^2  
\end{pmatrix}$,
$c_8=\begin{pmatrix}
0 & 1 \\
1 & a  
\end{pmatrix}$,
$c_9=\begin{pmatrix}
0 & 1 \\
1 & a^4  
\end{pmatrix}$
and
$w_k(m)= \zeta_k^m+\zeta_k^{-m}$.
We denote $e_{\chi_i^H}^H$ by $e_{\chi_i}^H$ for a subgroup $H$ of $G$.

%% file: ZZpG_504A.tex

\subsection{The case of the characters of degree 7}
First, we consider the character $\chi_2^G$.
Put
$$
s=\begin{pmatrix}
0 &a^5 \\
a^2 &0  
\end{pmatrix},~
t=\begin{pmatrix}
a^5 &a^2 \\
a^6 &a^5  
\end{pmatrix},~
u=\begin{pmatrix}
a &a
\\
a^5 &a  
\end{pmatrix}
~\text{and}~
h=\begin{pmatrix}
a^4 &1
\\
1 &0   
\end{pmatrix}.
$$
Let $C_2^{(3)}=\langle s,t,u \rangle$. 
Then $C_2^{(3)}\simeq (\mathbb{Z}/2\mathbb{Z})^{(3)}$. 
Let $\chi_i^{C_2^{(3)}}$ be the non-trivial characters of $C_2^{(3)}$ for $i\in\{2,3,\cdots,8\}$. 
Let $\chi_2^{C_2^{(3)}}$ be the character satisfying 
$\chi_2^{C_2^{(3)}}(s)=-1$, $\chi_2^{C_2^{(3)}}(t)=1$ and $\chi_2^{C_2^{(3)}}(u)=1$.
Let $C_9=\langle h\rangle$.
Then $C_9\simeq \mathbb{Z}/9\mathbb{Z}$.
Let $\chi_2^{C_9}$ be the character of $C_9$ with $\chi_2^{C_9}(h)=\zeta_3^2$.
With the aid of a computer, we determine $S_{G,\chi_2^G}$ and establish the following lemma
(refer to Section \ref{504-7_Com}). 

\begin{lem}\label{504-7(2)_Lem}
Assume that $p\not\in\{2,3,7\}.$
\begin{enumerate}
\setlength{\itemsep}{5pt}
\item $\chi_2^G = \Ind_{C_2^{(3)}}^G\bigl(\chi_2^{C_2^{(3)}}\bigr)
-\Ind_{C_9}^G\bigl(\chi_2^{C_9}\bigr)$.

\item $\displaystyle e_2^G =
\Bigl(\sum_{i=2}^8 e_{\chi_i}^{C_2^{(3)}}\Bigr)e_2^G$.

\item For all $i\in \{2,3,\ldots,8\}$, there exists $g_i \in G$ such that
$$e_{\chi_i}^{C_2^{(3)}} = g_i\,e_{\chi_2}^{C_2^{(3)}} g_i^{-1}.$$

\item $\bigl(1- e_2^G\bigr) e_{\chi_2}^{C_2^{(3)}}\,\mathbb{Q}(\zeta_3)[G]=
e_{\chi_2}^{C_2^{(3)}} e_{\chi_2}^{C_9}\,\mathbb{Q}(\zeta_3)[G].$

\item $\mathbb{Q}(\zeta_3)[G]\,e_{\chi_2}^{C_2^{(3)}} e_{\chi_2}^{C_9} 
= \mathbb{Q}(\zeta_3)[G]\,e_{\chi_2}^{C_9}.$
\end{enumerate}
\end{lem} 
The proof of Theorem \ref{MainTh1} in the case of $\chi=\chi_2^G$
is similar to the argument in Section 2.1.
Next, we focus on the character $\chi_3^G$
because $\chi_3^G$, $\chi_4^G$, and $\chi_5^G$ can be handled similarly.
Let $\chi_8^{C_9}$ be the character of $C_9$ with $\chi_8^{C_9}(h)=\zeta_9$.
With the aid of a computer, we determine $S_{G,\chi_3^G}$ and establish the following lemma
(refer to Section \ref{504-7_Com}).

\begin{lem}\label{504-7(3)_Lem} 
Assume that $p\not\in\{2,3,7\}.$
\begin{enumerate}
\setlength{\itemsep}{5pt}
\item $\chi_3^G = \Ind_{C_2^{(3)}}^G\bigl(\chi_2^{C_2^{(3)}}\bigr)
-\Ind_{C_9}^G\bigl(\chi_8^{C_9}\bigr)$.

\item $\displaystyle e_3^G =
\Bigl(\sum_{i=2}^8 e_{\chi_i}^{C_2^{(3)}}\Bigr)e_3^G$.

\item $\bigl(1- e_3^G\bigr) e_{\chi_2}^{C_2^{(3)}}\,\mathbb{Q}(\zeta_9)[G]=
e_{\chi_2}^{C_2^{(3)}} e_{\chi_8}^{C_9}\,\mathbb{Q}(\zeta_9)[G].$

\item $\mathbb{Q}(\zeta_9)[G]\,e_{\chi_2}^{C_2^{(3)}} e_{\chi_8}^{C_9} 
= \mathbb{Q}(\zeta_9)[G]\,e_{\chi_8}^{C_9}.$
\end{enumerate}
\end{lem}

The proof of Theorem \ref{MainTh1} in the case of $\chi=\chi_3^G$
is similar to that of the case of $\chi=\chi_2^G$.
Since similar results hold for $\chi_4^G$ and $\chi_5^G$,
we are able to prove Theorem \ref{MainTh1} in these cases in the same way.

%% file: ZZpG_504B.tex

\subsection{The case of the character of degree 8}
Put
\begin{align*}
&D_{7,(1)} = \langle 
\begin{pmatrix}
0 &a^6 \\
a &0  
\end{pmatrix},~
\begin{pmatrix}
a^2 &a \\
a^3 &a^3  
\end{pmatrix}\rangle, 
&D_{7,(2)} = \langle 
\begin{pmatrix}
0 &a^2 \\
a^5 &0  
\end{pmatrix},~
\begin{pmatrix}
a^2 &a^4 \\
1 &a^3  
\end{pmatrix}\rangle,\\
&D_{7,(3)} = \langle 
\begin{pmatrix}
1 &a^4 \\
0 &1  
\end{pmatrix},~
\begin{pmatrix}
a^6 &0 \\
a &a  
\end{pmatrix}\rangle,
&D_{7,(4)} = \langle 
\begin{pmatrix}
a^4 &a^6 \\
a^4 &a^4  
\end{pmatrix},~
\begin{pmatrix}
a^4 &a^6 \\
a^6 &1  
\end{pmatrix}\rangle,\\
&D_{7,(5)} = \langle 
\begin{pmatrix}
a &a^4 \\
a^2 &a  
\end{pmatrix},~
\begin{pmatrix}
a^5 &a^3 \\
a^4 &0  
\end{pmatrix}\rangle,
&D_{7,(6)} = \langle 
\begin{pmatrix}
a^4 &1 \\
a^3 &a^4  
\end{pmatrix},~
\begin{pmatrix}
1 &1 \\
a^5 &a^4  
\end{pmatrix}\rangle,\\
&D_{7,(7)} = \langle 
\begin{pmatrix}
1 &0 \\
1 &1  
\end{pmatrix},~
\begin{pmatrix}
0 &a^5 \\
a^2 &a^5  
\end{pmatrix}\rangle,
&D_{7,(8)} = \langle 
\begin{pmatrix}
a^6 &a^2 \\
a^2 &a^6  
\end{pmatrix},~
\begin{pmatrix}
a &a^2 \\
0 &a^6  
\end{pmatrix}\rangle.
\end{align*}

These are all conjugate to $D_{7,(1)}$ in $G$ 
and 
$D_{7,(i)}\cap D_{7,(j)} = \{1_G\}$ for all distinct $i, j \in \{1, 2, \ldots, 8\}$.
Let $\varepsilon_i$ denote the trivial character of $D_{7,(i)}$.
Put
$$
D_9 = \langle 
\begin{pmatrix}
a^5 &a^3 \\
a^5 &a^5  
\end{pmatrix},~
\begin{pmatrix}
a^4 &1 \\
1 &0  
\end{pmatrix}\rangle.
$$
Then $D_9$ is isomorphic to the dihedral group of order 18.
Let $\varepsilon'$ denote the trivial character of $D_9$.
With the aid of a computer, we determine $S_{G,\chi_6^G}$ and establish the following lemma
(refer to Section \ref{504-8_Com}).

\begin{lem}\label{504-8_Lem}
Assume that $p\not\in\{2,3,7\}.$
\begin{enumerate}
\setlength{\itemsep}{5pt}
\item $\chi_6^G = \Ind_{D_{7,(1)}}^G(\varepsilon_1)-\Ind_{D_9}^G(\varepsilon')$.

\item $\displaystyle e_6^G\in
\Bigl( \sum_{i=1}^8 e_{\varepsilon_i}^{D_{7,(i)}} \Bigr)\,\mathbb{Q}[G]$.

\item For any $j\in \{1,2,\ldots,8\}$, 
there exists 
$r_j\in e_{\varepsilon_j}^{D_{7,(j)}}\mathbb{Q}[G]\setminus\{0\}$ 
such that
$$ e_6^G e_{\varepsilon_j}^{D_{7,(j)}} \in\mathbb{Q}[G]\,
r_j e_6^G e_{\varepsilon_j}^{D_{7,(j)}}\setminus\{0\}
$$
and
$$r_j  e_{\varepsilon_i}^{D_{7,(i)}} = 0$$
for all $i\in\{1,2,\ldots,8\}\setminus\{j\}$.

\item 
For all $i\in \{1,2,\ldots,8\}$, there exists $g_i\in G$ such that
$$e_{\varepsilon_1}^{D_{7,(1)}} = g_i\,e_{\varepsilon_i}^{D_{7,(i)}} g_i^{-1}.$$

\item $\left(1-e_6^G\right)e_{\varepsilon_1}^{D_{7,(1)}}\,\mathbb{Q}[G] 
=
e_{\varepsilon_1}^{D_{7,(1)}}e_{\varepsilon'}^{D_9}\,\mathbb{Q}[G].$

\item $\mathbb{Q}[G]\,
e_{\varepsilon_1}^{D_{7,(1)}}e_{\varepsilon'}^{D_9}
=
\mathbb{Q}[G]\,e_{\varepsilon'}^{D_9}$. 
\end{enumerate}
\end{lem}
The proof of Theorem \ref{MainTh1} in the case of $\chi=\chi_6^G$
is similar to the argument in Section 2.2.

%% file: ZZpG_504C.tex

\subsection{The case of the characters of degree 9}
Since $\chi_7^G$, $\chi_8^G$, and $\chi_9^G$ can be handled similarly,
we focus on the character $\chi_7^G$.
Put
$$F_{8,(1)}=\langle 
\begin{pmatrix}
a &a \\
a^5 &a  
\end{pmatrix},~
\begin{pmatrix}
0 &a^5 \\
a^2 &0  
\end{pmatrix},~
\begin{pmatrix}
a^5 &a^2 \\
a^6 &a^5  
\end{pmatrix},~
\begin{pmatrix}
a^6 &a^6 \\
a^2 &a^4  
\end{pmatrix}
\rangle.$$
Since $F_{8,(1)}$ has exactly nine conjugates in $G$,
let $F_{8,(2)}$, $\ldots$, $F_{8,(9)}$ be all of the other conjugate subgroups of $F_{8,(1)}$ in $G$.
Furthermore, for each $i \in \{2, \ldots, 9\}$, we choose an element $h_i \in G$ such that $F_{8,(i)} = h_i F_{8,(1)} h_i^{-1}$.
The subgroup $F_{8,(1)}$ has exactly eight irreducible characters: one trivial character, six non-trivial characters of degree 1, and one character of degree 7.
For $i\in \{1,2,\ldots,9\}$,
let $\chi^{F_{8,(i)}}$ be the non-trivial character of degree 1 
with $\chi^{F_{8,(i)}}(h_i g h_i^{-1})=\zeta_7^2$,
where 
$$g=
\begin{pmatrix}
a^6 &a^6 \\
a^2 &a^4  
\end{pmatrix}.$$
We have $S_{G,\chi_7^G}=\{2,3,7\}$  and
obtain the following lemma by using a computer
(see Section \ref{504-9_Com}).

\begin{lem}\label{504-9_Lem} 
Assume that $p\not\in\{2,3,7\}.$
\begin{enumerate}
\setlength{\itemsep}{5pt}
\item $\chi_7^G = \Ind_{F_{8,(1)}}^G\left(\chi^{F_{8,(1)}} \right).$

\item $\displaystyle e_7^G =\sum_{i=1}^9 e_{\chi}^{F_{8,(i)}}$.

\item $e_{\chi}^{F_{8,(i)}} e_{\chi}^{F_{8,(j)}} 
= \delta_{ij} e_{\chi}^{F_{8,(i)}}$, 
where $\delta_{ij}$ is the Kronecker delta.

\item For all $i\in \{1,2,\ldots,9\}$, there exists $g_i \in G$ such that
$$e_{\chi}^{F_{8,(1)}} = g_i\,e_{\chi}^{F_{8,(i)}} g_i^{-1}.$$
\end{enumerate}
\end{lem} 
The proof of Theorem \ref{MainTh1} in the case of $\chi=\chi_7^G$
is similar to the argument in Section 2.3.
Since propositions similar to Lemma \ref{504-9_Lem} hold for $\chi_{8}^{G}$ and $\chi_{9}^{G}$, 
we can prove Theorem \ref{MainTh1} in these cases similarly.

%% file: ZZpG_360.tex

\section{The simple group of order 360}
Let $G$ be the alternating group of degree 6.
Then the character table of $G$ is given in Table \ref{Ch360}.
\begin{table}[htbp]
\begin{center}
\caption{The character table of $G$}\label{Ch360}
\begin{tabular}{l|ccccccc}
\multicolumn{1}{c|}{} &$1_G$ &$c_2$ &$c_3$ &$c_4$ &$c_5$ &$c_6$ &$c_7$ \\ \hline 
$\chi_1^G$  &$1$ &$1$ &$1$ &$1$ &$1$ &$1$ &$1$ \\
$\chi_2^G$  &$5$ &$-1$ &$1$ &$0$ &$0$ &$2$ &$-1$ \\
$\chi_3^G$  &$5$ &$2$  &$1$ &$0$ &$0$ &$-1$ &$-1$ \\
$\chi_4^G$  &$8$ &$-1$ &$0$ &$-\phi^{-1}$ &$\phi$ &$-1$ &$0$ \\
$\chi_5^G$  &$8$ &$-1$ &$0$ &$\phi$ &$-\phi^{-1}$ &$-1$ &$0$ \\
$\chi_6^G$  &$9$ &$0$ &$1$ &$-1$ &$-1$ &$0$ &$1$ \\
$\chi_7^G$  &$10$ &$1$ &$-2$ &$0$ &$0$ &$1$ &$0$ \\
\end{tabular}
\end{center}
\end{table}
Here, $c_2=(1,2,3)$, $c_3=(1,2)(3,4)$, $c_4=(1,2,3,4,5)$, $c_5=(1,3,4,5,2)$,
$c_6=(1,2,3)(4,5,6)$, $c_7=(1,2,3,4)(5,6)$ 
and $\phi = (1+\sqrt{5})/2$.

%% file: ZZpG_360A.tex

\subsection{The case of the characters of degree 5}
Since $\chi_2^G$ and $\chi_3^G$ can be handled similarly,
we focus on the character $\chi_3^G$.
Put $s_1=(1, 2)(3,6)$ and $t_1=(1,4,2,5)(3,6)$.
Let $D_{4,(1)}$ be the subgroup of $G$ that is generated by $s_1$ and $t_1$.
Put $h_2=(2,4)(3,5)$, $h_3=(2,3)(5,6)$, $h_4=(1,3)(2,5)$ and $h_5=(1,5)(2,3)$.
For $i\in\{2, 3, 4, 5\}$, 
let $D_{4,(i)}$ be the subgroup of $G$ that is generated by $s_i$ and $t_i$,
where $s_i=h_i s_1 h_i^{-1}$ and $t_i=h_i t_1 h_i^{-1}$.
Then each $D_{4,(i)}$ is isomorphic to the dihedral group of order 8.
Moreover, we have $$D_{4,(i)} \cap D_{4,(j)} = \{1_G\}$$ for all distinct $i,j$.
Let $\eta^{D_{4,(i)}}$ be the character of $D_{4,(i)}$ whose degree is 1
with $\eta^{D_{4,(i)}}(s_i)=1$ and $\eta^{D_{4,(i)}}(t_i)=-1$.
Put 
$$H=\langle(1,6,3), (2,4,5)\rangle \simeq (\ZZ/3\ZZ)^{(2)}.$$
Let $\chi^H$ be the character of $H$ with 
$\chi^H((1,6,3))=\chi^H((2,4,5))=\zeta_3$.
With the aid of a computer, we determine $S_{G,\chi_3^G}$ and establish the following lemma
(refer to Section \ref{360-3_Com}).

\begin{lem}\label{360-3_Lem} 
Assume that $p\not\in\{2,3,5\}.$
\begin{enumerate}
\setlength{\itemsep}{5pt}
\item $\chi_3^G = \Ind_{D_{4,(1)}}^G(\eta^{D_{4,(1)}})-\Ind_H^G(\chi^H)$.

\item $\displaystyle e_3^G\in
\Bigl( \sum_{i=1}^5 e_{\eta}^{D_{4,(i)}} \Bigr)\,\mathbb{Q}[G]$.

\item For any $j\in \{1,2,\ldots,5\}$, 
there exists $r_j \in e_{\eta}^{D_{4,(j)}}\mathbb{Q}[G]$ such that
$$ e_3^G e_{\eta}^{D_{4,(j)}} \in\mathbb{Q}[G]\,
r_j e_3^G e_{\eta}^{D_{4,(j)}}
$$
and
$$r_j e_{\eta}^{D_{4,(i)}} =0$$
for all $i\in\{1,2,\ldots,5\}\setminus\{j\}$.

\item For all $i\in \{2,3,4,5\}$, we have
$$e_{\eta}^{D_{4,(i)}} = h_i\,e_{\eta}^{D_{4,(1)}} h_i^{-1}.$$

\item $\left(1-e_3^G\right)e_{\eta}^{D_{4,(1)}}\,\mathbb{Q}(\zeta_3)[G] 
=
e_{\eta}^{D_{4,(1)}}e_{\chi}^{H}\,\mathbb{Q}(\zeta_3)[G].$

\item $\mathbb{Q}(\zeta_3)[G]\,
e_{\eta}^{D_{4,(1)}}e_{\chi}^{H}
=
\mathbb{Q}(\zeta_3)[G]\,e_{\chi}^{H}$. 
\end{enumerate}
\end{lem}
The proof of Theorem \ref{MainTh1} in the case of $\chi=\chi_3^G$
is similar to the argument in Section 2.2.
Since a similar proposition holds for $\chi_{2}^{G}$,
we can prove Theorem \ref{MainTh1} in this case similarly.

%% file: ZZpG_360B.tex

\subsection{The case of the character of degree 9}
Put
\begin{align*}
&S_{4,(1)}=\langle (3,4)(5,6), (1,6,4)(2,3,5) \rangle,\quad
S_{4,(2)}=\langle (1,2)(3,4), (1,3,6)(2,5,4) \rangle, \\
&S_{4,(3)}=\langle (1,2)(4,5), (1,6,5)(2,4,3) \rangle,\quad
S_{4,(4)}=\langle (1,2)(3,6), (1,4,3)(2,6,5) \rangle, \\ 
&S_{4,(5)}=\langle (1,2)(5,6), (1,5,4)(2,3,6) \rangle,\quad
S_{4,(6)}=\langle (2,5)(3,6), (1,5,6)(2,4,3) \rangle, \\
&S_{4,(7)}=\langle (2,6)(3,4), (1,2,3)(4,6,5) \rangle,\quad
S_{4,(8)}=\langle (1,3)(4,5), (1,4,2)(3,6,5) \rangle, \\
&S_{4,(9)}=\langle (2,4)(5,6), (1,2,5)(3,6,4) \rangle
\end{align*}
and
$$A_{5a}=\langle (1,5)(4,6), (1,3,2)(4,5,6) \rangle.$$
Let $\varepsilon_i$ and $\varepsilon'$ be the trivial characters of 
$S_{4,(i)}$ and $A_{5a}$, respectively.
With the aid of a computer, we determine $S_{G,\chi_6^G}$ and establish the following lemma
(refer to Section \ref{360-6_Com}).

\begin{lem}\label{360-6_Lem} 
Assume that $p\not\in\{2,3,5\}$.
\begin{enumerate}
\setlength{\itemsep}{5pt}
\item $\chi_6^G 
= \Ind_{S_{4,(1)}}^G(\varepsilon_1)-\Ind_{A_{5a}}^G(\varepsilon')$.

\item $\displaystyle e_6^G\in
\Bigl( \sum_{i=1}^9 e_{\varepsilon_i}^{S_{4,(i)}} \Bigr)\,\mathbb{Q}[G]$.

\item For any $j\in \{1,2,\ldots,9\}$, 
there exists $r_j\in e_{\varepsilon_j}^{S_{4,(j)}}\mathbb{Q}[G]$ such that
$$ e_6^G e_{\varepsilon_j}^{S_{4,(j)}} \in\mathbb{Q}[G]\,
r_j e_6^G e_{\varepsilon_j}^{S_{4,(j)}}
$$
and
$$r_j  e_{\varepsilon_i}^{S_{4,(i)}} = 0$$
for all $i\in\{1,2,\ldots,9\}\setminus\{j\}$.

\item 
For all $i\in \{1,2,\ldots,9\}$, there exists $g_i\in G$ such that
$$e_{\varepsilon_1}^{S_{4,(1)}} = g_i\,e_{\varepsilon_i}^{S_{4,(i)}} g_i^{-1}.$$

\item $\left(1-e_6^G\right)e_{\varepsilon_1}^{S_{4,(1)}}\,\mathbb{Q}[G] 
=
e_{\varepsilon_1}^{S_{4,(1)}}e_{\varepsilon'}^{A_{5a}}\,\mathbb{Q}[G].$

\item $\mathbb{Q}[G]\,
e_{\varepsilon_1}^{S_{4,(1)}}e_{\varepsilon'}^{A_{5a}}
=
\mathbb{Q}[G]\,e_{\varepsilon'}^{A_{5a}}$. 
\end{enumerate}
\end{lem}
The proof of Theorem \ref{MainTh1} in the case of $\chi=\chi_6^G$
is similar to the argument in Section 2.2.

%% file: ZZpG_360C.tex

\subsection{The case of the character of degree 10}
Put 
$$H_1=\langle (3,5)(4,6), (1,2)(3,4,5,6), (1,4,6)\rangle.$$
Since $H_1$ has exactly ten conjugates in $G$, 
let $H_2,\ldots, H_{10}$ be the other nine conjugates.
Let $\eta^{H_1}$ be the character of $H_1$ of degree 1 satisfying $\eta^{H_1}((1,2)(3,4,5,6))=\sqrt{-1}$.
Choose $a_i\in G$ such that $H_i = a_i H_1 a_i^{-1}$.
Let $\eta^{H_i}$ be the character of $H_i$ that satisfies 
$\eta^{H_i}(a_i x a_i^{-1})=\eta^{H_1}(x)$ for any $x\in H_1$.
We have $S_{G,\chi_7^G}=\{2,3,5\}$ and 
obtain the following lemma by using a computer
(see Section \ref{360-7_Com}).

\begin{lem}\label{360-7_Lem} 
Assume that $p\not\in \{2,3,5\}$.
\begin{enumerate}
\setlength{\itemsep}{5pt}
\item $\chi_7^G = \Ind_{H_1}^G\left( \eta^{H_1} \right).$

\item $\displaystyle e_7^G =\sum_{i=1}^{10} e_{\eta}^{H_i}$.

\item $e_{\eta}^{H_i} e_{\eta}^{H_j} 
= \delta_{ij} e_{\eta}^{H_i}$, 
where $\delta_{ij}$ is the Kronecker delta.

\item For all $i\in \{1,2,\ldots,10\}$, there exists $g_i \in G$ such that
$$e_{\eta}^{H_1} = g_i\,e_{\eta}^{H_i} g_i^{-1}.$$
\end{enumerate}
\end{lem} 
The proof of Theorem \ref{MainTh1} in the case of $\chi=\chi_7^G$
is similar to the argument in Section 2.3.

%% file: ZZpG_360D.tex

\subsection{The case of the characters of degree 8}\label{Jyuubun}
Since $\chi_4^G$ and $\chi_5^G$ can be handled similarly,
we focus on $\chi_4^G$.
Let $H =\langle (1,2,3), (4,5,6)\rangle$. 
Then $H \simeq (\mathbb{Z}/3\mathbb{Z})^{(2)}$.
For $m,n\in\{0,1,2\}$, let 
$\psi_{((mn)_3+1)_{10}}(=\psi_{3m+n+1})$ 
be the character of $H$ with
$\psi_{((mn)_3+1)_{10}}((1,2,3))=\zeta_3^m$ and 
$\psi_{((mn)_3+1)_{10}}((4,5,6))=\zeta_3^n$,
(e.g., $\psi_{((01)_3+1)_{10}}=\psi_2$ and $\psi_{((11)_3+1)_{10}}=\psi_5$).  
Put $C_5=\langle (1,2,3,4,5) \rangle$ and 
let $\sigma$ be the character of $C_5$ with $\sigma((1,2,3,4,5))=\zeta_5$.
Put $d_i=(1^G-e_4^G)e_{\psi_i}^H$ for $i\in\{2,\ldots,9\}$.
With the aid of a computer, we determine $S_{G,\chi_4^G}$ and establish the following lemma (refer to Section \ref{360_Others}).

\begin{lem}\label{360-4_Lem}
Assume that $p\not\in\{2,3,5\}.$ 
\begin{enumerate}
\setlength{\itemsep}{5pt}
\item $\chi_4^G=\Ind_H^G(\psi_2)+\Ind_H^G(\psi_5)-\Ind_{C_5}^G(\sigma)$.
\item $\displaystyle e_4^G=\bigl(\sum_{i=2}^9 e_{\psi_i}^H\bigr) e_4^G$.
\item For $i\in\{3,4,7\}$ and $j\in\{6,8,9\}$, 
there exist $g_i\in G$ and $g'_j\in G$ such that 
\[e_{\psi_2}^H = g_i\, e_{\psi_i}^H\, g_i^{-1}
\quad\text{and}\quad
e_{\psi_5}^H = g'_j\, e_{\psi_j}^H\, {g'_j}^{-1}.\]   
However, $e_{\psi_5}^H \neq g\, e_{\psi_2}^H\, g^{-1}$ for all $g\in G$.

\item 
$\mathbb{Q}(\zeta_{15})[G]\, e_{\psi_5}^H e_\sigma^{C_5} e_{\psi_2}^H e_4^G
= \mathbb{Q}(\zeta_{15})[G]\, e_{\psi_2}^H e_4^G.$

\item
$e_{\psi_5}^H e_\sigma^{C_5} e_{\psi_2}^H e_4^G \,\mathbb{Q}(\zeta_{15})[G]
= e_{\psi_5}^H e_4^G \,\mathbb{Q}(\zeta_{15})[G].$

\item 
$\mathbb{Q}(\zeta_{15})[G]\, (e_{\psi_5}^H+d_2) 
= \mathbb{Q}(\zeta_{15})[G]\, e_\sigma^{C_5}(e_{\psi_5}^H+d_2).$

\item 
$e_\sigma^{C_5} \,\mathbb{Q}(\zeta_{15})[G]
= e_\sigma^{C_5}(e_{\psi_5}^H+d_2) \,\mathbb{Q}(\zeta_{15})[G].$
\end{enumerate}
\end{lem}

Note that if $p \notin S_{G,\chi_4^G}$, 
the statement of Lemma \ref{360-4_Lem} holds even if we replace all occurrences of $\mathbb{Q}(\zeta_{15})[G]$ with $\mathbb{Z}_p[\zeta_{15}][G]$.
Let $M$ be a $\ZZ_p[\zeta_{15}][G]$-module. 
The following theorem follows from Lemma \ref{360-4_Lem}.

\begin{thm}\label{Tsugihagi}
Assume that $p\not\in S_{G,\chi_4^G}$. 
\begin{enumerate}
\setlength{\itemsep}{5pt}

\item For $i\in\{3,4,7\}$ and $j\in\{6,8,9\}$, 
\[e_{\psi_2}^H M \simeq e_{\psi_i}^H M
\quad\text{and}\quad
e_{\psi_5}^H M \simeq e_{\psi_j}^HM.\] 

\item $e_{\psi_2}^H M /d_2 M \simeq e_{\psi_5}^H M /d_5 M.$

\item $e_4^G M \simeq \left(e_{\psi_2}^H M /d_2 M\right)^{(8)}$.

\item $e_{\psi_5}^H M \oplus d_2 M ~\simeq~ e_\sigma^{C_5} M$.

\item If $M$ is finite, then
Theorem \ref{MainTh1} holds for $\chi=\chi_4^G$. 
\end{enumerate} 
\end{thm}

\begin{proof}
Note that $e_4^G$ is a central idempotent, and $d_i$ is an idempotent.

(1)~ Let $g_i$ be as given in Lemma \ref{360-4_Lem} (3).
The following homomorphism
\[
f_{g_i}:e_{\psi_2}^H M \longrightarrow e_{\psi_i}^H M,\quad
x \longmapsto g_i^{-1} x
\] 
is an isomorphism for $i\in\{3,4,7\}$.
Similarly, 
we get $e_{\psi_5}^H M \simeq e_{\psi_j}^H M$
for $j\in\{6,8,9\}$.

(2)~ The homomorphism
\[
f_{e_4^G}:e_{\psi_2}^H M \longrightarrow e_4^G e_{\psi_2}^H M,\quad
x \longmapsto e_4^G x 
\] 
is surjective and $\Ker(f_{e_4^G})= d_2 M$.
It follows that 
\[
 e_{\psi_2}^H M /d_2 M \simeq e_4^G e_{\psi_2}^H M
\]  
by the fundamental homomorphism theorem.
We also get $e_{\psi_5}^H M /d_5 M \simeq e_4^G e_{\psi_5}^H M$. 

The following homomorphism
\[
f_{e_{\psi_5}^H e_\sigma^{C_5}}
:e_{\psi_2}^H e_4^G M \longrightarrow e_{\psi_5}^H e_4^G M,\quad
x \longmapsto e_{\psi_5}^H e_\sigma^{C_5}  x 
\] 
is injective by Lemma \ref{360-4_Lem} (4) and surjective by Lemma \ref{360-4_Lem} (5).
Therefore,
\[
e_{\psi_2}^H M /d_2 M \simeq e_4^G e_{\psi_2}^H M \simeq
e_4^G e_{\psi_5}^H M \simeq  e_{\psi_5}^H M /d_5 M.
\]

(3) The following homomorphism
$$f_{e_4^G}: \bigoplus_{i=2}^9 e_{\psi_i}^H M \longrightarrow e_4^G M,~~
x\longmapsto e_4^G x$$
is surjective by Lemma \ref{360-4_Lem} (2),
and
$\Ker(f_{e_4^G})=\bigoplus_{i=2}^9 d_i M.$
Thus,
\[
e_4^G M \simeq \bigoplus_{i=2}^9 e_{\psi_i}^H M/d_i M
\]
holds by the fundamental homomorphism theorem.
For all $i\in\{3,\dots,9\}$, 
it follows that 
\[
e_{\psi_2}^H M/d_2 M \simeq e_{\psi_i}^H M/d_i M
\]
from Theorem \ref{Tsugihagi} (1) and (2).
Therefore we obtain
\(
e_4^G M \simeq \left(e_{\psi_2}^H M /d_2 M\right)^{(8)}.
\)

(4)~ Note that $e_{\psi_5}^H d_2 = d_2 e_{\psi_5}^H =0$.
The following homomorphism
\[
f_{e_\sigma^{C_5}}
:e_{\psi_5}^H M \oplus d_2 M \longrightarrow e_\sigma^{C_5} M,\quad
x=(e_{\psi_5}^H + d_2)x \longmapsto e_\sigma^{C_5}(e_{\psi_5}^H + d_2)x 
\] 
is injective by Lemma \ref{360-4_Lem} (6) and surjective by Lemma \ref{360-4_Lem} (7).

(5)~If $\# M <\infty$ then we obtain
$$\# \bigl(e_{\psi_2}^H M/ d_2 M \bigr) =
\frac{\#e_{\psi_2}^H M \cdot \# e_{\psi_5}^H M}{\# e_\sigma^{C_5}M}$$
by Theorem \ref{Tsugihagi}(4). 
That is, Theorem \ref{MainTh1} holds for $\chi=\chi_4^G$ by Theorem \ref{Tsugihagi}. 
\end{proof}

%% file: ZZpG_app.tex

\section{Appendix}

We will explain how to use Magma to check the formulas.
When $a=(1, 2)(3, 4)$ and $b=(1, 2, 3)$, 
it is normally calculated as $ab=(2, 4, 3)$ in $A_4$.
Note that this order of multiplication must be reversed when using Magma. 
To check this, enter the following on Magma.
\begin{lstlisting}[language=C, frame=single, basicstyle=\ttfamily\tiny, escapechar=\@]
> A4:=AlternatingGroup(4);
> a:=A4!(1,2)(3,4);
> b:=A4!(1,2,3);
> b * a;
(2, 4, 3)
\end{lstlisting}
In this paper, we have used Magma V2.29-7 primarily.
The text file (e.g., Lemma21-5\_q4.txt) output by the Magma command 'PrintFile' is provided in Section \ref{Output}.

%% file: ZZpG_app168-chi2.tex

\subsection{The program for Lemma \ref{168-3_Lem}}\label{168-3_Com}

Save the following program of Magma as \verb@2-1_Def.txt@.
\begin{lstlisting}[language=C, frame=single, basicstyle=\ttfamily\tiny, escapechar=\@]
G:=PSL(2, 7);
ChG:=CharacterTable(G);

PZ<x>:=PolynomialAlgebra(Integers());
F<z>:=CyclotomicField(7);
R:=GroupAlgebra(F, G);
w:=z+z^2+z^4;

c5:=G!(1,5,3,7,6,8,2);
for i in [1..#ChG] do if ( ChG[i](c5) eq w ) then Num_chi2:=i; end if; end for;

a:=G!(1,8,2,4,3,7,5);
C7:=sub<G|a>;
ChC7:=CharacterTable(C7);

for i in [1..#ChC7] do if ( ChC7[i](a) eq z ) then Num1:=i; end if; end for;
chi1:=ChC7[Num1];

for i in [1..#ChC7] do if ( ChC7[i](a) eq z^2 ) then Num2:=i; end if; end for;
chi2:=ChC7[Num2];

for i in [1..#ChC7] do if ( ChC7[i](a) eq z^4 ) then Num4:=i; end if; end for;
chi4:=ChC7[Num4];

r:=G!(1,7,3,8)(2,6,4,5);
s:=G!(1,2)(3,4)(5,7)(6,8);
D4:=sub<G|r,s>;
ChD4:=CharacterTable(D4);

for i in [1..#ChD4] do if ((ChD4[i](r) eq -1) and (ChD4[i](s) eq -1)) then Num_eta:=i; 
end if; end for;

eta:=CharacterTable(D4)[Num_eta];

e2:=R![(ChG[Num_chi2](Id(G))/#G)*R!ChG[Num_chi2](Inverse(g)):g in G];

eC7_1:=R!0;
for g in C7 do eC7_1 := eC7_1 +R!(1/#C7)*chi1(g^(-1))*R!g; end for;
eC7_2:=R!0;
for g in C7 do eC7_2 := eC7_2 +R!(1/#C7)*chi2(g^(-1))*R!g; end for;
eC7_4:=R!0;
for g in C7 do eC7_4 := eC7_4 +R!(1/#C7)*chi4(g^(-1))*R!g; end for;

eD4:=R!0;
for g in D4 do eD4 := eD4 +R!(1/#D4)*eta(g^(-1))*R!g; end for;

PF<[T]>:=PolynomialRing(F,168);
PFG:=GroupAlgebra(PF, G);
E:=[Eltseq(Basis(VectorSpace(Rationals(), 168))[i]): i in [1..168]];
t:=PFG![T[i]:i in [1..168]];
\end{lstlisting}

One can confirm that the lemma holds by typing the following commands in the interactive mode of Magma V2.29-7.
\begin{lstlisting}[language=C, frame=single, basicstyle=\ttfamily\tiny, escapechar=\@]
> load "2-1_Def.txt";
Loading "2-1_Def.txt"

/* check */
> G eq sub<SymmetricGroup(8)|(3,6,7)(4,5,8), (1,8,2)(4,5,6)>;
true
> chi1(a) eq z;
true;
> chi2(a) eq z^2;
true;
> chi4(a) eq z^4;
true;
> (eta(r) eq -1) and (eta(s) eq -1);
true

/* Lemma 2.1 (1) */
> ChG[Num_chi2] eq Induction(chi1,G)-Induction(eta,G);
true
> ChG[Num_chi2] eq Induction(chi2,G)-Induction(eta,G);
true
> ChG[Num_chi2] eq Induction(chi4,G)-Induction(eta,G);
true

/* Lemma 2.1 (2) */
> (eC7_1+eC7_2+eC7_4)*e2 eq e2;
true

/* Lemma 2.1 (3) */
> h:=G!(1,7,8)(3,5,4);
> eC7_4 eq h*eC7_2*h^(-1);
true
> eC7_1 eq h^(-1)*eC7_2*h;     
true

/* Lemma 2.1 (4), q1 */ 
> b:=Matrix(F,1,168,[Eltseq(eD4*eC7_1)[i]:i in [1..168]]);
> c:=PFG![PF!Eltseq(eC7_1*(R!1-e2))[i]:i in [1..168]];
> time aij:=[Evaluate(Eltseq(t*c)[i], E[j]): i,j in [1..168]];
Time: 79539.070
> A:=Matrix(F,168,168,aij);
> q1,x:=Solution(Transpose(A),b);
> R!Eltseq(q1)*eC7_1*(R!1-e2) eq eD4*eC7_1;
true
> PrintFile("Lemma21-4_q1.txt", R!Eltseq(q1)); // Export q1 to Lemma21-4_q1.txt.

/* Lemma 2.1 (4), q2 */
> bb:=Matrix(F,1,168,[Eltseq(eC7_1*(R!1-e2))[i]:i in [1..168]]);
> cc:=PFG![PF!Eltseq(eD4*eC7_1)[i]:i in [1..168]];
> time aaij:=[Evaluate(Eltseq(t*cc)[i], E[j]): i,j in [1..168]];
Time: 10826.480
> AA:=Matrix(F,168,168,aaij);
> q2,xx:=Solution(Transpose(AA),bb);
> eC7_1*(R!1-e2) eq R!Eltseq(q2)*eD4*eC7_1;
true
> PrintFile("Lemma21-4_q2.txt", R!Eltseq(q2)); // Export q2 to Lemma21-4_q2.txt.

/* Lemma 2.1 (5), q3 */
> b3:=Matrix(F,1,168,[Eltseq(eD4)[i]:i in [1..168]]);
> c3:=PFG![PF!Eltseq(eD4*eC7_1)[i]:i in [1..168]];
> time a3ij:=[Evaluate(Eltseq(c3*t)[i], E[j]): i,j in [1..168]];
Time: 11193.520
> A3:=Matrix(F,168,168,a3ij);
> q3,x3:=Solution(Transpose(A3),b3);
> eD4*eC7_1*R!Eltseq(q3) eq eD4;
true
> PrintFile("Lemma21-5_q3.txt", R!Eltseq(q3)); // Export q3 to Lemma21-5_q3.txt.

/* Lemma 2.1 (5), q4 */
> b4:=Matrix(F,1,168,[Eltseq(eD4*eC7_1)[i]:i in [1..168]]);
> c4:=PFG![PF!Eltseq(eD4)[i]:i in [1..168]];
> time a4ij:=[Evaluate(Eltseq(c4*t)[i], E[j]): i,j in [1..168]];
Time: 397.400
> A4:=Matrix(F,168,168,a4ij);
> q4,x4:=Solution(Transpose(A4),b4);
> eD4*eC7_1 eq eD4*R!Eltseq(q4);             
true
> PrintFile("Lemma21-5_q4.txt", R!Eltseq(q4)); // Export q4 to Lemma21-5_q4.txt.                                                                   
\end{lstlisting}

%% file: ZZpG_app168-chi4.tex

\subsection{The program for Lemma \ref{168-6_Lem}}\label{168-6_Com}

Save the following program of Magma as \verb@2-2_Def.txt@.
\begin{lstlisting}[language=C, frame=single, basicstyle=\ttfamily\tiny, escapechar=\@]
G:=PSL(2, 7);
ChG:=CharacterTable(G);

A4_1:=sub<G|(2,6,4)(3,5,8), (1,2)(3,5)(4,6)(7,8)>;
ChA4_1:=CharacterTable(A4_1);
A4_2:=sub<G|(1,2,8)(4,6,5), (1,8)(2,3)(4,6)(5,7)>;
A4_3:=sub<G|(1,3,4)(5,6,8), (1,7)(2,8)(3,4)(5,6)>;
A4_4:=sub<G|(1,4,5)(2,3,6), (1,8)(2,7)(3,6)(4,5)>;
A4_5:=sub<G|(1,2,5)(3,4,8), (1,5)(2,7)(3,8)(4,6)>;
A4_6:=sub<G|(1,6,5)(4,8,7), (1,5)(2,8)(3,6)(4,7)>;
A4_7:=sub<G|(1,7,6)(2,5,4), (1,8)(2,4)(3,5)(6,7)>;
ConjA4a:=[A4_1, A4_2, A4_3, A4_4, A4_5, A4_6];
H:=sub<G|(1,4,2)(5,8,7), (1,8,2,4,3,7,5)>;
ChH:=CharacterTable(H);

R:=GroupAlgebra(Rationals(),G);
e4:=R![(ChG[4](Id(G))/#G)*R!ChG[4](Inverse(g)):g in G];

eA4a:=[R!0:i in [1..6]];
for i in [1..6] do
for g in ConjA4a[i] do eA4a[i]:=eA4a[i]+R!(1/#ConjA4a[i])*(R!g); end for;
end for;

eH:=R!0;
for g in H do eH:=eH+R!(1/#H)*(R!g); end for;

PQ<[T]>:=PolynomialRing(Rationals(),168);
PQG:=GroupAlgebra(PQ, G);
E:=[Eltseq(Basis(VectorSpace(Rationals(), 168))[i]): i in [1..168]];
t:=PQG![T[i]:i in [1..168]];
\end{lstlisting}

One can confirm that the lemma holds by typing the following commands in the interactive mode of Magma V2.29-7.
\begin{lstlisting}[language=C, frame=single, basicstyle=\ttfamily\tiny, escapechar=\@]
> load "2-2_Def.txt";
Loading "2-2_Def.txt"

/* Lemma 2.2 (1) */
> ChG[4] eq Induction(ChA4_1,G)[1] - Induction(ChH,G)[1];
true

/* Lemma 2.2 (2) */
> time e4 in R*(&+eA4a);  
true
Time: 354.100
> b0:=Matrix(Rationals(),1,168,[Eltseq(e4)[i]:i in [1..168]]);
> c0:=PQG![PQ!Eltseq(&+eA4a)[i]:i in [1..168]];
> time a0ij:=[Evaluate(Eltseq(t*c0)[i], E[j]): i,j in [1..168]];
Time: 200.100
> A0:=Matrix(Rationals(),168,168,a0ij);
> q0,x0:=Solution(Transpose(A0),b0);
> r0:=R!Eltseq(q0);
> r0*(eA4a[1]+eA4a[2]+eA4a[3]+eA4a[4]+eA4a[5]+eA4a[6]) eq e4;
true
> for i in [1..168] do if not(2^5*3*7*q0[1,i] in Integers()) then print(i);
end if; end for;
> PrintFile("Lemma22-2_r0.txt", r0);   // Export r0 to Lemma22-2_r0.txt. 

/*== Lemma 2.2 (3) ==*/
> time V1:=Basis(RightAnnihilator(R*eA4a[2]+R*eA4a[3]+R*eA4a[4]+R*eA4a[5]+R*eA4a[6]) 
meet R*eA4a[1]);
Time: 11.340
> r1:=V1[2];
> eA4a[1]*e4 in (eA4a[1]*e4*r1)*R;
true
> for i in [2..6] do if not(eA4a[i]*r1 eq 0) then print(i); end if; end for;
> b1:=Matrix(Rationals(),1,168,[Eltseq(eA4a[1]*e4)[i]:i in [1..168]]);
> c1:=PQG![PQ!Eltseq(eA4a[1]*e4*r1)[i]:i in [1..168]];
> time a1ij:=[Evaluate(Eltseq(c1*t)[i], E[j]): i,j in [1..168]];
Time: 881.040
> A1:=Matrix(Rationals(),168,168,a1ij);
> q1,x1:=Solution(Transpose(A1),b1);
> rr1:=R!Eltseq(q1);
> eA4a[1]*e4*r1*rr1 eq eA4a[1]*e4;
true
> PrintFile("Lemma22-3_r1.txt", r1);     // Export r1 to Lemma22-3_r1.txt. 
> PrintFile("Lemma22-3_rr1.txt", rr1);   // Export rr1 to Lemma22-3_rr1.txt.

> V2:=Basis(RightAnnihilator(R*eA4a[1]+R*eA4a[3]+R*eA4a[4]+R*eA4a[5]+R*eA4a[6])
 meet R*eA4a[2]);
> r2:=V2[2];
> eA4a[2]*e4 in (eA4a[2]*e4*r2)*R;
true
> for i in [1,3,4,5,6] do if not(eA4a[i]*r2 eq 0) then print(i); end if; end for;
> b2:=Matrix(Rationals(),1,168,[Eltseq(eA4a[2]*e4)[i]:i in [1..168]]);
> c2:=PQG![PQ!Eltseq(eA4a[2]*e4*r2)[i]:i in [1..168]];
> time a2ij:=[Evaluate(Eltseq(c2*t)[i], E[j]): i,j in [1..168]];
Time: 1269.630
> A2:=Matrix(Rationals(),168,168,a2ij);
> q2,x2:=Solution(Transpose(A2),b2);
> rr2:=R!Eltseq(q2);
> eA4a[2]*e4*r2*rr2 eq eA4a[2]*e4;
true
> PrintFile("Lemma22-3_r2.txt", r2);     // Export r2 to Lemma22-3_r2.txt. 
> PrintFile("Lemma22-3_rr2.txt", rr2);   // Export rr2 to Lemma22-3_rr2.txt.

> V3:=Basis(RightAnnihilator(R*eA4a[1]+R*eA4a[2]+R*eA4a[4]+R*eA4a[5]+R*eA4a[6])
 meet R*eA4a[3]);
> r3:=V3[2];
> eA4a[3]*e4 in (eA4a[3]*e4*r3)*R;
true
> for i in [1,2,4,5,6] do if not(eA4a[i]*r3 eq 0) then print(i); end if; end for;
> b3:=Matrix(Rationals(),1,168,[Eltseq(eA4a[3]*e4)[i]:i in [1..168]]);
> c3:=PQG![PQ!Eltseq(eA4a[3]*e4*r3)[i]:i in [1..168]];
> time a3ij:=[Evaluate(Eltseq(c3*t)[i], E[j]): i,j in [1..168]];
Time: 1258.000
> A3:=Matrix(Rationals(),168,168,a3ij);
> q3,x3:=Solution(Transpose(A3),b3);
> rr3:=R!Eltseq(q3);
> eA4a[3]*e4*r3*rr3 eq eA4a[3]*e4;
true
> PrintFile("Lemma22-3_r3.txt", r3);     // Export r3 to Lemma22-3_r3.txt. 
> PrintFile("Lemma22-3_rr3.txt", rr3);   // Export rr3 to Lemma22-3_rr3.txt.

> V4:=Basis(RightAnnihilator(R*eA4a[1]+R*eA4a[2]+R*eA4a[3]+R*eA4a[5]+R*eA4a[6])
 meet R*eA4a[4]);
> r4:=V4[2];
> eA4a[4]*e4 in (eA4a[4]*e4*r4)*R;
true
> for i in [1,2,3,5,6] do if not(eA4a[i]*r4 eq 0) then print(i); end if; end for;
> b4:=Matrix(Rationals(),1,168,[Eltseq(eA4a[4]*e4)[i]:i in [1..168]]);
> c4:=PQG![PQ!Eltseq(eA4a[4]*e4*r4)[i]:i in [1..168]];
> time a4ij:=[Evaluate(Eltseq(c4*t)[i], E[j]): i,j in [1..168]];
Time: 1254.070
> A4:=Matrix(Rationals(),168,168,a4ij);
> q4,x4:=Solution(Transpose(A4),b4);
> rr4:=R!Eltseq(q4);
> eA4a[4]*e4*r4*rr4 eq eA4a[4]*e4;
true
> PrintFile("Lemma22-3_r4.txt", r4);     // Export r4 to Lemma22-3_r4.txt. 
> PrintFile("Lemma22-3_rr4.txt", rr4);   // Export rr4 to Lemma22-3_rr4.txt.

> V5:=Basis(RightAnnihilator(R*eA4a[1]+R*eA4a[2]+R*eA4a[3]+R*eA4a[4]+R*eA4a[6])
 meet R*eA4a[5]);
> r5:=V5[2];
> eA4a[5]*e4 in (eA4a[5]*e4*r5)*R;
true
> for i in [1,2,3,4,6] do if not(eA4a[i]*r5 eq 0) then print(i); end if; end for;
> b5:=Matrix(Rationals(),1,168,[Eltseq(eA4a[5]*e4)[i]:i in [1..168]]);
> c5:=PQG![PQ!Eltseq(eA4a[5]*e4*r5)[i]:i in [1..168]];
> time a5ij:=[Evaluate(Eltseq(c5*t)[i], E[j]): i,j in [1..168]];
Time: 1263.940
> A5:=Matrix(Rationals(),168,168,a5ij);
> q5,x5:=Solution(Transpose(A5),b5);
> rr5:=R!Eltseq(q5);
> eA4a[5]*e4*r5*rr5 eq eA4a[5]*e4;
true
> PrintFile("Lemma22-3_r5.txt", r5);     // Export r5 to Lemma22-3_r5.txt. 
> PrintFile("Lemma22-3_rr5.txt", rr5);   // Export rr5 to Lemma22-3_rr5.txt.

> V6:=Basis(RightAnnihilator(R*eA4a[1]+R*eA4a[2]+R*eA4a[3]+R*eA4a[4]+R*eA4a[5])
 meet R*eA4a[6]);
> r6:=V6[1];
> eA4a[6]*e4 in (eA4a[6]*e4*r6)*R;
true
> for i in [1,2,3,4,5] do if not(eA4a[i]*r6 eq 0) then print(i); end if; end for;
> b6:=Matrix(Rationals(),1,168,[Eltseq(eA4a[6]*e4)[i]:i in [1..168]]);
> c6:=PQG![PQ!Eltseq(eA4a[6]*e4*r6)[i]:i in [1..168]];
> time a6ij:=[Evaluate(Eltseq(c6*t)[i], E[j]): i,j in [1..168]];
Time: 1342.430
> AA6:=Matrix(Rationals(),168,168,a6ij);
> q6,x6:=Solution(Transpose(AA6),b6);
> rr6:=R!Eltseq(q6);
> eA4a[6]*e4*r6*rr6 eq eA4a[6]*e4;
true
> PrintFile("Lemma22-3_r6.txt", r6);     // Export r6 to Lemma22-3_r6.txt. 
> PrintFile("Lemma22-3_rr6.txt", rr6);   // Export rr6 to Lemma22-3_rr6.txt.

/* Lemma 2.2 (4) */
> for g in G do if (eA4a[1] eq g^(-1)*eA4a[2]*g) and (g^2 eq G!1) then g; end if; 
end for;
(1, 2)(3, 4)(5, 7)(6, 8)
(1, 5)(2, 7)(3, 8)(4, 6)

> for g in G do if (eA4a[1] eq g^(-1)*eA4a[3]*g) and (g^2 eq G!1) then g; end if; 
end for;
(1, 5)(2, 6)(3, 7)(4, 8)
(1, 4)(2, 3)(5, 8)(6, 7)

> for g in G do if (eA4a[1] eq g^(-1)*eA4a[4]*g) and (g^2 eq G!1) then g; end if; 
end for;
(1, 3)(2, 6)(4, 7)(5, 8)
(1, 4)(2, 5)(3, 7)(6, 8)

> for g in G do if (eA4a[1] eq g^(-1)*eA4a[5]*g) and (g^2 eq G!1) then g; end if; 
end for;
(1, 2)(3, 8)(4, 7)(5, 6)
(1, 8)(2, 3)(4, 6)(5, 7)

> for g in G do if (eA4a[1] eq g^(-1)*eA4a[6]*g) and (g^2 eq G!1) then g; end if; 
end for;
(1, 8)(2, 4)(3, 5)(6, 7)
(1, 6)(2, 5)(3, 4)(7, 8)

/* Lemma 2.2 (5) */
> R*eA4a[1]*(R!1-e4) eq R*eH*eA4a[1];      
true

/* Lemma 2.2 (5), q1 */
> b1:=Matrix(Rationals(),1,168,[Eltseq(eH*eA4a[1])[i]:i in [1..168]]);
> c1:=PQG![PQ!Eltseq(eA4a[1]*(R!1-e4))[i]:i in [1..168]];
> time a1ij:=[Evaluate(Eltseq(t*c1)[i], E[j]): i,j in [1..168]];
Time: 1617.420
> A1:=Matrix(Rationals(),168,168,a1ij);
> q1,x:=Solution(Transpose(A1),b1);
> R!Eltseq(q1)*eA4a[1]*(R!1-e4) eq eH*eA4a[1];
true
> PrintFile("Lemma22-5_q1.txt", R!Eltseq(q1));     // Export q1 to Lemma22-5_q1.txt. 

/* Lemma 2.2 (5), q2 */
> b2:=Matrix(Rationals(),1,168,[Eltseq(eA4a[1]*(R!1-e4))[i]:i in [1..168]]);
> c2:=PQG![PQ!Eltseq(eH*eA4a[1])[i]:i in [1..168]];
> time a2ij:=[Evaluate(Eltseq(t*c2)[i], E[j]): i,j in [1..168]];
Time: 473.740
> A2:=Matrix(Rationals(),168,168,a2ij);
> q2,x:=Solution(Transpose(A2),b2);
> eA4a[1]*(R!1-e4) eq R!Eltseq(q2)*eH*eA4a[1];
true
> PrintFile("Lemma22-5_q2.txt", R!Eltseq(q2));     // Export q2 to Lemma22-5_q2.txt.

/* Lemma 2.2 (6) */
> eH*eA4a[1]*R eq eH*R;
true
              
/* Lemma 2.2 (6), q3 */
> b3:=Matrix(Rationals(),1,168,[Eltseq(eH)[i]:i in [1..168]]);
> c3:=PQG![PQ!Eltseq(eH*eA4a[1])[i]:i in [1..168]];
> time a3ij:=[Evaluate(Eltseq(c3*t)[i], E[j]): i,j in [1..168]];
Time: 409.690
> A3:=Matrix(Rationals(),168,168,a3ij);
> q3,x:=Solution(Transpose(A3),b3);
> eH*eA4a[1]*R!Eltseq(q3) eq eH;
true
> PrintFile("Lemma22-6_q3.txt", R!Eltseq(q3));     // Export q3 to Lemma22-6_q3.txt.

/* Lemma 2.2 (6), q4 */
> b4:=Matrix(Rationals(),1,168,[Eltseq(eH*eA4a[1])[i]:i in [1..168]]);
> c4:=PQG![PQ!Eltseq(eH)[i]:i in [1..168]];
> time a4ij:=[Evaluate(Eltseq(c4*t)[i], E[j]): i,j in [1..168]];
Time: 75.250
> AA4:=Matrix(Rationals(),168,168,a4ij);
> q4,x:=Solution(Transpose(AA4),b4);
> eH*eA4a[1] eq eH*R!Eltseq(q4);
true
> PrintFile("Lemma22-6_q4.txt", R!Eltseq(q4));     // Export q4 to Lemma22-6_q4.txt.
\end{lstlisting}

%% file: ZZpG_app168-chi5+6.tex

\subsection{The program for Lemma \ref{168-7_Lem}}\label{168-7_Com}
Save the following program of Magma as \verb@2-3_Def.txt@.
\begin{lstlisting}[language=C, frame=single, basicstyle=\ttfamily\tiny, escapechar=\@]
G:=PSL(2, 7);
ChG:=CharacterTable(G);
S4_1:=sub<G|(2,6,4)(3,5,8), (1,5)(2,8)(3,6)(4,7)>;
ConjS4a:=[X: X in {S4_1^g:g in G}];

R:=GroupAlgebra(Rationals(),G);
e5:=R![(ChG[5](Id(G))/#G)*R!ChG[5](Inverse(g)):g in G];
eS4a:=[R!0:i in [1..7]];
for i in [1..7] do
for g in ConjS4a[i] do 
eS4a[i]:=eS4a[i]+R!(1/#ConjS4a[i])*CharacterTable(ConjS4a[i])[2](g^(-1))*(R!g); 
end for;
end for;
\end{lstlisting}

One can confirm that the lemma holds by typing the following commands in the interactive mode of Magma V2.29-7.
\begin{lstlisting}[language=C, frame=single, basicstyle=\ttfamily\tiny, escapechar=\@]
> load "2-3_Def.txt";
Loading "2-3_Def.txt"

/*== Lemma 2.3 (1) ==*/
> ChG[5] eq Induction(CharacterTable(S4_1)[2],G);
true

/*== Lemma 2.3 (2) ==*/
> e5 eq &+eS4a; 
true

/*== Lemma 2.3 (3) ==*/
> for i, j in [1..7] do 
if (i lt j) then if not(eS4a[i]*eS4a[j] eq R!0) then print("Error!"); 
end if; end if; end for;

/*== Lemma 2.3 (4) ==*/
> g2:=G!(1, 6)(2, 7)(3, 5)(4, 8);
> ConjS4a[2]^g2 eq ConjS4a[1];
true
> eS4a[2] eq g2^(-1)*eS4a[1]*g2;
true

> g3:=G!(1, 2)(3, 4)(5, 7)(6, 8);
> ConjS4a[3]^g3 eq ConjS4a[1];                                          
true 
> eS4a[3] eq g3^(-1)*eS4a[1]*g3;
true

> g4:=G!(1, 4)(2, 3)(5, 8)(6, 7);                                       
> ConjS4a[4]^g4 eq ConjS4a[1];                                          
true
> eS4a[4] eq g4^(-1)*eS4a[1]*g4;                                        
true

> g5:=G!(1, 4)(2, 5)(3, 7)(6, 8);                                       
> ConjS4a[5]^g5 eq ConjS4a[1];                                          
true
> eS4a[5] eq g5^(-1)*eS4a[1]*g5;                                        
true

> g6:=G!(1, 8)(2, 3)(4, 6)(5, 7);
> ConjS4a[6]^g6 eq ConjS4a[1];                                          
true
> eS4a[6] eq g6^(-1)*eS4a[1]*g6;                                        
true

> g7:=G!(1, 8)(2, 4)(3, 5)(6, 7);
> ConjS4a[7]^g7 eq ConjS4a[1];   
true
> eS4a[7] eq g7^(-1)*eS4a[1]*g7;                                        
true
\end{lstlisting}

\subsection{The program for Lemma \ref{168-8_Lem}}\label{168-8_Com}
Save the following program of Magma as \verb@2-4_Def.txt@.
\begin{lstlisting}[language=C, frame=single, basicstyle=\ttfamily\tiny, escapechar=\@]
G:=PSL(2, 7);
ChG:=CharacterTable(G);

F<zeta3>:=CyclotomicField(3);
R:=GroupAlgebra(F,G);
e6:=R![(ChG[6](Id(G))/#G)*R!ChG[6](Inverse(g)):g in G];

t:=[G!(1,4,2)(5,8,7), G!(1,4,7)(2,8,6), G!(1,5,2)(3,8,4), G!(1,7,4)(2,6,8), 
G!(1,8,5)(2,7,6), G!(1,5,4)(2,6,3), G!(1,7,5)(3,6,8), G!(2,3,4)(6,8,7)];

s:=[G!(1,8,2,4,3,7,5), G!(1,7,5,6,2,4,8), G!(1,4,2,5,6,3,8), G!(1,8,6,3,7,4,2), 
G!(1,5,3,7,6,8,2), G!(1,4,7,6,3,5,2), G!(1,8,5,7,4,3,6), G!(2,6,7,5,3,4,8)];

H:=[];
for j in [1..8] do Append(~H, sub<G|t[j],s[j]>); end for;

ChH1:=CharacterTable(H[1]);
for i in [1..#ChH1] do if ( ChH1[i](t[1]) eq (zeta3)^2 ) then Num:=i; end if; end for;

eH:=[R!0:i in [1..8]];
for i in [1..8] do for g in H[i] do 
eH[i]:=eH[i]+R!(1/#H[i])*CharacterTable(H[i])[Num](g^(-1))*(R!g); 
end for; end for;
\end{lstlisting}

One can confirm that the lemma holds by typing the following commands in the interactive mode of Magma V2.29-7.
\begin{lstlisting}[language=C, frame=single, basicstyle=\ttfamily\tiny, escapechar=\@]
> load "2-4_Def.txt";
Loading "2-4_Def.txt"

/*== check ==*/
> for j in [1..8] do if not(CharacterTable(H[j])[Num](s[j]) eq 1) then print("Error!"); 
end if; end for;
> for j in [1..8] do if not(CharacterTable(H[j])[Num](t[j]) eq zeta3^2) 
then print("Error!"); end if; end for;

/*== Lemma 2.4 (1) ==*/
> for j in [1..8] do if not(ChG[6] eq Induction(CharacterTable(H[1])[Num], G)) 
then print("Error!"); end if; end for;

/*== Lemma 2.4 (2) ==*/
> e6 eq &+eH;
true

/*== Lemma 2.4 (3) ==*/
> for i, j in [1..8] do if (not(i eq j) and not(eH[i]*eH[j] eq R!0)) then 
print("Error!"); end if; end for;

/*== Lemma 2.4 (4) ==*/
> g2:=G!(1, 8)(2, 7)(3, 6)(4, 5);
> H[2] eq H[1]^g2;               
true
> eH[2] eq g2^(-1)*eH[1]*g2;
true

> g3:=G!(1, 8)(2, 4)(3, 5)(6, 7);
> eH[3] eq g3^(-1)*eH[1]*g3;                                     
true

> g4:=G!(1, 2)(3, 8)(4, 7)(5, 6);                                
> eH[4] eq g4^(-1)*eH[1]*g4;                                     
true

> g5:=G!(1, 2)(3, 5)(4, 6)(7, 8);                                
> eH[5] eq g5^(-1)*eH[1]*g5;                                     
true

> g6:=G!(1, 2)(3, 4)(5, 7)(6, 8);                                
> eH[6] eq g6^(-1)*eH[1]*g6;                                     
true

> g7:=G!(1, 3)(2, 6)(4, 7)(5, 8);                                
> eH[7] eq g7^(-1)*eH[1]*g7;                                     
true

> g8:=G!(1, 6)(2, 5)(3, 4)(7, 8);                                
> eH[8] eq g8^(-1)*eH[1]*g8;                                     
true
\end{lstlisting}

%% file: ZZpG_app504.tex

\subsection{The program for Lemmas \ref{504-7(2)_Lem} and \ref{504-7(3)_Lem}}
\label{504-7_Com}
Save the following program of Magma as \verb@3-1_Def.txt@.
\begin{lstlisting}[language=C, frame=single, basicstyle=\ttfamily\tiny, escapechar=\@]
Q3<z3>:=CyclotomicField(3);
Q9<z9>:=CyclotomicField(9);
res, phi := IsSubfield(Q3, Q9);
psi := phi^(-1);

PZ<x>:=PolynomialAlgebra(Integers());
F<w>:=NumberField(x^3-3*x-1:Global);
res2, phi2 := IsSubfield(F, Q9);
psi2 := phi2^(-1);

GF8<a>:=GF(8);
G:=Group("A1(8)");
ChG:=CharacterTable(G);

c_7:=[[0,1],[1,a^2]];
c_8:=[[0,1],[1,a]];
c_9:=[[0,1],[1,a^4]];
s:=[[0,a^5],[a^2,0]];
t:=[[a^5,a^2],[a^6,a^5]];
u:=[[a,a],[a^5,a]];
H:=sub<G|s,t,u>;
ChH:=CharacterTable(H);

h:=[[a^4,1],[1,0]];
C9:=sub<G|h>;
ChC9:=CharacterTable(C9);

QG:=GroupAlgebra(Rationals(), G);
Q3G:=GroupAlgebra(Q3, G);
Q9G:=GroupAlgebra(Q9, G);
FG:=GroupAlgebra(F, G);

e2:=QG![(ChG[2](Id(G))/#G)*QG!ChG[2](g^(-1)):g in G];
e2Q3:=Q3G![(ChG[2](Id(G))/#G)*Q3G!ChG[2](g^(-1)):g in G];

for i in [1..9] do
if (ChG[i](c_7) eq -(z9)^4-(z9)^5) then N_3:=i; break;
end if; end for;

e3:=FG![(ChG[N_3](Id(G))/#G)*FG!psi2(ChG[N_3](g^(-1))):g in G];
e3Q9:=Q9G![(ChG[N_3](Id(G))/#G)*Q9G!(ChG[N_3](g^(-1))):g in G];

eH:=[QG!0:i in [1..8]];
for i in [1..8] do
for g in H do eH[i]:=eH[i]+QG!(1/#H)*(ChH[i](g^(-1)))*(QG!g); end for;
end for;

eHQ3:=[Q3G!0:i in [1..8]];
for i in [1..8] do
for g in H do eHQ3[i]:=eHQ3[i]+Q3G!(1/#H)*(ChH[i](g^(-1)))*(Q3G!g); end for;
end for;

eH_F:=[FG!0:i in [1..8]];
for i in [1..8] do
for g in H do eH_F[i]:=eH_F[i]+FG!(1/#H)*(ChH[i](g^(-1)))*(FG!g); end for;
end for;

eHQ9:=[Q9G!0:i in [1..8]];
for i in [1..8] do
for g in H do eHQ9[i]:=eHQ9[i]+Q9G!(1/#H)*(ChH[i](g^(-1)))*(Q9G!g); end for;
end for;

eC9_2:=Q3G!0;
for i in [1..9] do
if (ChC9[i](h) eq (z3)^2) then Num2:=i; break;
end if; end for;
for g in C9 do eC9_2:=eC9_2+Q3G!(1/#C9)*psi(ChC9[Num2](g^(-1)))*(Q3G!g); end for;

eC9_8:=Q9G!0;
for i in [1..9] do
if (ChC9[i](h) eq z9) then Num1:=i; break;
end if; end for;
for g in C9 do eC9_8:=eC9_8+Q9G!(1/#C9)*ChC9[Num1](g^(-1))*(Q9G!g); end for;

/* ConjugacyClasses(G) */
C:=[];
for i in [1..9] do C[i]:=G!ConjugacyClasses(G)[i][3]; end for;

/* $\chi_2^{C_2^{(3}}$ */
for i in [1..8] do
if ((ChH[i](s) eq -1)and(ChH[i](t) eq 1)and(ChH[i](u) eq 1))
then Num:=i; break;
end if; end for;
\end{lstlisting}

One can confirm that the lemma holds 
by typing the following commands in the interactive mode of Magma V2.29-7.
\begin{lstlisting}[language=C, frame=single, basicstyle=\ttfamily\tiny, escapechar=\@]
> load "3-1_Def.txt";
Loading "3-1_Def.txt"

/* check */
> ChH[Num](s);
-1
> ChH[Num](t);              
1
> ChH[Num](u);
1

> ChC9[Num2](h);
-zeta(3)_3 - 1

/* Lemma 3.1 (1) */
> ChG[2] eq Induction(ChH[Num], G)-Induction(ChC9[Num2],G);
true

/* Lemma 3.1 (2) */
> e2 eq (&+eH - eH[1])*e2;
true

/* Lemma 3.1 (3) */
> gg:=G![[a,a^3],[0,a^6]];

> for i in [1..6] do 
if eH[3] eq gg^(i)*eH[2]*gg^(-i) then print(gg^i); break; 
end if; end for;
[  a a^3]
[  0 a^6]

> for i in [1..6] do 
if eH[4] eq gg^(i)*eH[2]*gg^(-i) then print(gg^i); break; 
end if; end for;
[a^3 a^4]
[  0 a^4]

> for i in [1..6] do 
> if eH[5] eq gg^(i)*eH[2]*gg^(-i) then print(gg^i); break; 
> end if; end for;
[a^5   a]
[  0 a^2]

> for i in [1..6] do 
if eH[6] eq gg^(i)*eH[2]*gg^(-i) then print(gg^i); break; 
end if; end for;
[a^4 a^4]
[  0 a^3]

> for i in [1..6] do 
if eH[7] eq gg^(i)*eH[2]*gg^(-i) then print(gg^i); break; 
end if; end for;
[a^6 a^3]
[  0   a]

> for i in [1..6] do 
if eH[8] eq gg^(i)*eH[2]*gg^(-i) then print(gg^i); break; 
end if; end for;
[a^2   a]
[  0 a^5]

/* Lemma 3.1 (4) */
> time (1-e2Q3)*eHQ3[Num]*Q3G eq eHQ3[Num]*eC9_2*Q3G;
true
Time: 2242.290

/* Lemma 3.1 (5) */
> time Q3G*eHQ3[Num]*eC9_2 eq Q3G*eC9_2;
true
Time: 54.800
\end{lstlisting}

One can confirm that the lemma holds 
by typing the following commands in the interactive mode of Magma V2.29-7.
\begin{lstlisting}[language=C, frame=single, basicstyle=\ttfamily\tiny, escapechar=\@]
> load "3-1_Def.txt";
Loading "3-1_Def.txt"

> ChC9[Num1](h);
zeta(9)_9

/* Lemma 3.2 (1) */
> ChG[N_3] eq Induction(ChH[Num], G)-Induction(ChC9[Num1],G);
true

/* Lemma 3.2 (2) */
> e3 eq (&+eH_F - eH_F[1])*e3;
true

/* Lemma 3.2 (3) */
> time (1-e3Q9)*eHQ9[Num]*Q9G eq eHQ9[Num]*eC9_8*Q9G;
true
Time: 13989.870

/* Lemma 3.2 (4) */
> time Q9G*eHQ9[Num]*eC9_8 eq Q9G*eC9_8;
true
Time: 266.240
\end{lstlisting}

\subsection{The program for Lemma \ref{504-8_Lem}}\label{504-8_Com}
Save the following program of Magma as \verb@3-2_Def.txt@.
\begin{lstlisting}[language=C, frame=single, basicstyle=\ttfamily\tiny, escapechar=\@]
GF8<a>:=GF(8);
G:=Group("A1(8)");
ChG:=CharacterTable(G);

QG:=GroupAlgebra(Rationals(), G);
e6:=QG![(ChG[6](Id(G))/#G)*QG!ChG[6](g^(-1)):g in G];

D7:=[];
D7[1]:=sub<G|[[0,a^6],[a,0]], [[a^2,a],[a^3,a^3]]>;
D7[2]:=sub<G|[[0,a^2],[a^5,0]], [[a^2,a^4],[1,a^3]]>;
D7[3]:=sub<G|[[1,a^4],[0,1]], [[a^6,0],[a,a]]>;
D7[4]:=sub<G|[[a^4,a^6],[a^4,a^4]], [[a^4,a^6],[a^6,1]]>;
D7[5]:=sub<G|[[a,a^4],[a^2,a]], [[a^5,a^3],[a^4,0]]>;
D7[6]:=sub<G|[[a^4,1],[a^3,a^4]], [[1,1],[a^5,a^4]]>;
D7[7]:=sub<G|[[1,0],[1,1]], [[0,a^5],[a^2,a^5]]>;
D7[8]:=sub<G|[[a^6,a^2],[a^2,a^6]], [[a,a^2],[0,a^6]]>;
ChD7:=CharacterTable(D7[1]);

eD7:=[QG!0:i in [1..8]];
for i in [1..8] do
for g in D7[i] do eD7[i]:=eD7[i]+QG!(1/#D7[i])*(CharacterTable(D7[i])[1](g^(-1)))*(QG!g); 
end for; end for;

D9:=sub<G|[[a^5,a^3],[a^5,a^5]], [[a^4,1],[1,0]]>;
ChD9:=CharacterTable(D9);
eD9:=QG!0;
for g in D9 do eD9:=eD9+QG!(1/#D9)*(ChD9[1](g^(-1)))*(QG!g); end for;

PQ<[T]>:=PolynomialRing(Rationals(),504);
PQG:=GroupAlgebra(PQ, G);
E:=[Eltseq(Basis(VectorSpace(Rationals(), 504))[i]): i in [1..504]];
t:=PQG![T[i]:i in [1..504]];
\end{lstlisting}

One can confirm that the lemma holds by typing the following commands in the interactive mode of Magma V2.29-7.
\begin{lstlisting}[language=C, frame=single, basicstyle=\ttfamily\tiny, escapechar=\@]
> load "3-2_Def.txt";
Loading "3-2_Def.txt"

/* Lemma 3.3 (1) */
> ChG[6] eq Induction(ChD7[1],G)-Induction(ChD9[1],G);
true

/* Lemma 3.3 (2) */
> b0:=Matrix(Rationals(),1,504,[Eltseq(e6)[i]:i in [1..504]]);
> c0:=PQG![PQ!Eltseq(&+eD7)[i]:i in [1..504]];
> time a0ij:=[Evaluate(Eltseq(c0*t)[i], E[j]): i,j in [1..504]];
Time: 11911.320
> A0:=Matrix(Rationals(),504,504,a0ij);
> q0,x0:=Solution(Transpose(A0),b0);
> r0:=QG!Eltseq(q0);
> (&+eD7)*r0 eq e6;
true
> PrintFile("Lemma33-2_r0.txt", r0);   // Export r0 to Lemma33-2_r0.txt. 

/* Lemma 3.3 (3) */
> time R_1:=Basis(LeftAnnihilator(eD7[2]*QG+eD7[3]*QG+eD7[4]*QG+eD7[5]*QG+eD7[6]*QG
+eD7[7]*QG+eD7[8]*QG) meet eD7[1]*QG);
Time: 1790.310
> r1:=R_1[1];
> time e6*eD7[1] in QG*(r1*e6*eD7[1]);
true
Time: 57.480
> for i in [1..8] do if not(r1*eD7[i] eq 0) then print(i); end if; end for;
1

/* the value of rr1, Start */
> b1:=Matrix(Rationals(),1,504,[Eltseq(e6*eD7[1])[i]:i in [1..504]]);
> c1:=PQG![PQ!Eltseq(r1*e6*eD7[1])[i]:i in [1..504]];
> time a1ij:=[Evaluate(Eltseq(t*c1)[i], E[j]): i,j in [1..504]];
Time: 324840.570
> A1:=Matrix(Rationals(),504,504,a1ij);
> q1,x1:=Solution(Transpose(A1),b1);
> rr1:=QG!Eltseq(q1);
> rr1*r1*e6*eD7[1] eq e6*eD7[1];
true
> PrintFile("Lemma33-3_r1.txt", r1);     // Export r1 to Lemma33-3_r1.txt. 
> PrintFile("Lemma33-3_rr1.txt", rr1);   // Export rr1 to Lemma33-3_rr1.txt.
/* the value of rr1, End */

> time r2:=Basis(LeftAnnihilator(eD7[1]*QG+eD7[3]*QG+eD7[4]*QG+eD7[5]*QG+eD7[6]*QG
+eD7[7]*QG+eD7[8]*QG) meet eD7[2]*QG);
Time: 1694.870
> time e6*eD7[2] in QG*(r2[1]*e6*eD7[2]);
true
Time: 42.430
> for i in [1..8] do if not(r2[1]*eD7[i] eq 0) then print i; end if; end for;
2
> PrintFile("Lemma33_r2[1].txt", r2[1]);

> time r3:=Basis(LeftAnnihilator(eD7[1]*QG+eD7[2]*QG+eD7[4]*QG+eD7[5]*QG+eD7[6]*QG
+eD7[7]*QG+eD7[8]*QG) meet eD7[3]*QG);
Time: 1235.150
> time e6*eD7[3] in QG*(r3[1]*e6*eD7[3]);
true
Time: 48.250
> for i in [1..8] do if not(r3[1]*eD7[i] eq 0) then print i; end if; end for;
3
> PrintFile("Lemma33_r3[1].txt", r3[1]);

> time r4:=Basis(LeftAnnihilator(eD7[1]*QG+eD7[2]*QG+eD7[3]*QG+eD7[5]*QG+eD7[6]*QG
+eD7[7]*QG+eD7[8]*QG) meet eD7[4]*QG);
Time: 1792.220
> time e6*eD7[4] in QG*(r4[1]*e6*eD7[4]);
true
Time: 56.010
> for i in [1..8] do if not(r4[1]*eD7[i] eq 0) then print i; end if; end for;
4
> PrintFile("Lemma33_r4[1].txt", r4[1]);

> time r5:=Basis(LeftAnnihilator(eD7[1]*QG+eD7[2]*QG+eD7[3]*QG+eD7[4]*QG+eD7[6]*QG
+eD7[7]*QG+eD7[8]*QG) meet eD7[5]*QG);
Time: 1794.440
> time e6*eD7[5] in QG*(r5[1]*e6*eD7[5]);
true
Time: 54.910
> for i in [1..8] do if not(r5[1]*eD7[i] eq 0) then print i; end if; end for;
5
> PrintFile("Lemma33_r5[1].txt", r5[1]);

> time r6:=Basis(LeftAnnihilator(eD7[1]*QG+eD7[2]*QG+eD7[3]*QG+eD7[4]*QG+eD7[5]*QG
+eD7[7]*QG+eD7[8]*QG) meet eD7[6]*QG);
Time: 1217.640
> time e6*eD7[6] in QG*(r6[1]*e6*eD7[6]);
true
Time: 54.980
> for i in [1..8] do if not(r6[1]*eD7[i] eq 0) then print i; end if; end for;
6
> PrintFile("Lemma33_r6[1].txt", r6[1]);

> time r7:=Basis(LeftAnnihilator(eD7[1]*QG+eD7[2]*QG+eD7[3]*QG+eD7[4]*QG+eD7[5]*QG
+eD7[6]*QG+eD7[8]*QG) meet eD7[7]*QG);
Time: 1497.560
> time e6*eD7[7] in QG*(r7[1]*e6*eD7[7]);
true
Time: 45.120
> for i in [1..8] do if not(r7[1]*eD7[i] eq 0) then print i; end if; end for;
7
> PrintFile("Lemma33_r7[1].txt", r7[1]);

> time r8:=Basis(LeftAnnihilator(eD7[1]*QG+eD7[2]*QG+eD7[3]*QG+eD7[4]*QG+eD7[5]*QG
+eD7[6]*QG+eD7[7]*QG) meet eD7[8]*QG);
Time: 2955.720
> time e6*eD7[8] in QG*(r8[1]*e6*eD7[8]);
true
Time: 61.890
>for i in [1..8] do if not(r8[1]*eD7[i] eq 0) then print i; end if; end for;
8
> PrintFile("Lemma33_r8[1].txt", r8[1]);

/* Lemma 3.3 (4) */
> for g in G do if (eD7[2] eq g*eD7[1]*g^(-1)) and (g^2 eq G!1) then g; end if; end for;
[  0 a^4]
[a^3   0]

> for g in G do if (eD7[3] eq g*eD7[1]*g^(-1)) and (g^2 eq G!1) then g; end if; end for;                            
[a^2 a^3]
[a^2 a^2]

> for g in G do if (eD7[4] eq g*eD7[1]*g^(-1)) and (g^2 eq G!1) then g; end if; end for;
[  a   1]
[a^6   a]

> for g in G do if (eD7[5] eq g*eD7[1]*g^(-1)) and (g^2 eq G!1) then g; end if; end for;
[a^4 a^2]
[  a a^4]

> for g in G do if (eD7[6] eq g*eD7[1]*g^(-1)) and (g^2 eq G!1) then g; end if; end for;
[a^3 a^5]
[a^4 a^3]

> for g in G do if (eD7[7] eq g*eD7[1]*g^(-1)) and (g^2 eq G!1) then g; end if; end for;
[a^5   a]
[  1 a^5]

> for g in G do if (eD7[8] eq g*eD7[1]*g^(-1)) and (g^2 eq G!1) then g; end if; end for;
[a^6 a^6]
[a^5 a^6]

/* Lemma 3.3 (5) */
> time (1-e6)*eD7[1]*QG eq eD7[1]*eD9*QG;
true
Time: 863.120

/* Lemma 3.3 (6) */
> time QG*eD7[1]*eD9 eq QG*eD9;
true
Time: 3.440
\end{lstlisting}

\subsection{The program for Lemma \ref{504-9_Lem}}\label{504-9_Com}
Save the following program of Magma as \verb@3-3_Def.txt@.
\begin{lstlisting}[language=C, frame=single, basicstyle=\ttfamily\tiny, escapechar=\@]
GF8<a>:=GF(8);
G:=Group("A1(8)");
ChG:=CharacterTable(G);
Q7<z>:=CyclotomicField(7);
Q7G:=GroupAlgebra(Q7, G);
F8:=sub<G|[[a^6,a^6],[a^2,a^4]],[[0,a^5],[a^2,0]],[[a^5,a^2],[a^6,a^5]],[[a,a],[a^5,a]]>;
F:=[X: X in {F8^g:g in G}];
ChF8:=CharacterTable(F8);

gg:=G![[a^6,a^6],[a^2,a^4]];
c_4:=G![[a^2,0],[0,a^5]];
c_5:=G![[a,0],[0,a^6]];
c_6:=G![[a^3,0],[0,a^4]];

for i in [1..9] do
if (ChG[i](c_4) eq z^2 + z^5) then N_7:=i; break;
end if; end for;

e7:=Q7G![(ChG[N_7](Id(G))/#G)*Q7G!ChG[N_7](g^(-1)):g in G];

for i in [1..8] do
if (ChF8[i](gg) eq z^2) then Num:=i; break;
end if; end for;

eF_2:=[Q7G!0:i in [1..9]];
for i in [1..9] do for g in F[i] do 
eF_2[i]:=eF_2[i]+Q7G!(1/#F[i])*(CharacterTable(F[i])[Num](g^(-1)))*(Q7G!g); 
end for; end for;
\end{lstlisting}

One can confirm that the lemma holds by typing the following commands in the interactive mode of Magma V2.29-7.
\begin{lstlisting}[language=C, frame=single, basicstyle=\ttfamily\tiny, escapechar=\@]
> load "3-3_Def.txt";
Loading "3-3_Def.txt"

/* Lemma 3.4 (1) */
> ChG[N_7] eq Induction(ChF8[Num], G);
true

/* Lemma 3.4 (2) */
> e7 eq &+(eF_2);
true

/* Lemma 3.4 (3) */
> for i, j in [1..9] do if ( (i lt j) and not(eF_2[i]*eF_2[j] eq Q7G!0) ) 
then print([i,j]); end if; end for;
> for i in [1..9] do if not(eF_2[i]*eF_2[i] eq eF_2[i])  then print(i); end if; end for;

/* Lemma 3.4 (4) */
> XX2:=[];
> for g in G do if (F[2] eq F8^g) and (eF_2[2] eq g*eF_2[1]*g^(-1)) then Append(~XX2, g); 
end if; end for;
> XX2[1];
[  1 a^3]
[  0   1]

> XX3:=[];
> for g in G do if (F[3] eq F8^g) and (eF_2[3] eq g*eF_2[1]*g^(-1)) then Append(~XX3, g); 
end if; end for;
> XX3[5];
[  1 a^4]
[  0   1]

> XX4:=[];
> for g in G do if (F[4] eq F8^g) and (eF_2[4] eq g*eF_2[1]*g^(-1)) then Append(~XX4, g); 
end if; end for;
> XX4[1];                          
[  1   a]
[  0   1]

> XX5:=[];
> for g in G do if (F[5] eq F8^g) and (eF_2[5] eq g*eF_2[1]*g^(-1)) then Append(~XX5, g); 
end if; end for;
> XX5[1];    
[  1   1]
[  0   1]

> XX6:=[];
> for g in G do if (F[6] eq F8^g) and (eF_2[6] eq g*eF_2[1]*g^(-1)) then Append(~XX6, g); 
end if; end for;
> XX6[1];
[  1 a^2]
[  0   1]

> XX7:=[];
> for g in G do if (F[7] eq F8^g) and (eF_2[7] eq g*eF_2[1]*g^(-1)) then Append(~XX7, g); 
end if; end for;
> XX7[1];
[  1 a^6]
[  0   1]

> XX8:=[];
> for g in G do if (F[8] eq F8^g) and (eF_2[8] eq g*eF_2[1]*g^(-1)) then Append(~XX8, g); 
end if; end for;
> XX8[4];
[  1 a^5]
[  0   1]

> XX9:=[];
> for g in G do if (F[9] eq F8^g) and (eF_2[9] eq g*eF_2[1]*g^(-1)) then Append(~XX9, g); 
end if; end for;
> XX9[2];
[  1   0]
[a^2   1]
\end{lstlisting}

%% file: ZZpG_app360-chi3+2.tex

\subsection{The program for Lemma \ref{360-3_Lem}}\label{360-3_Com}
Save the following program of Magma as \verb@4-1_Def.txt@.
\begin{lstlisting}[language=C, frame=single, basicstyle=\ttfamily\tiny, escapechar=\@]
G:=AlternatingGroup(6);
ChG:=CharacterTable(G);
Q3<z3>:=CyclotomicField(3);

c2:=G!(1,2,3);
s1:=G!(1,2)(3,6);
t1:=G!(1,4,2,5)(3,6);
h:=[G!1, G!(2,4)(3,5), G!(2,3)(5,6), G!(1,3)(2,5), G!(1,5)(2,3)];

s:=[]; t:=[];
for j in [1..5] do Append(~s, h[j]^(-1)*s1*h[j]); end for;
for j in [1..5] do Append(~t, h[j]^(-1)*t1*h[j]); end for;
D4:=[];
for j in [1..5] do Append(~D4, sub<G|s[j],t[j]>); end for;

H:=sub<G|(1,6,3),(2,4,5)>;
ChH:=CharacterTable(H);

for i in [1..7] do 
if (ChG[i](c2) eq 2) then N_3:=i; 
end if; end for;
QG:=GroupAlgebra(Rationals(),G);
e3:=QG![(ChG[N_3](Id(G))/#G)*QG!(ChG[N_3](g^(-1))):g in G];
Q3G:=GroupAlgebra(Q3,G);
e3Q3:=Q3G![(ChG[N_3](Id(G))/#G)*Q3G!(ChG[N_3](g^(-1))):g in G];

ChD4_1:=CharacterTable(D4[1]);
for j in [1..5] do 
if (ChD4_1[j](s1) eq 1) and (ChD4_1[j](t1) eq -1) then NumD:=j; 
end if; end for;
eD4:=[QG!0:i in [1..5]];
for i in [1..5] do
for g in D4[i] do 
eD4[i]:=eD4[i]+QG!(1/#D4[i])*CharacterTable(D4[i])[NumD](g^(-1))*(QG!g); 
end for; end for;

eD4Q3:=[Q3G!0:i in [1..5]];
for i in [1..5] do
for g in D4[i] do 
eD4Q3[i]:=eD4Q3[i]+Q3G!(1/#D4[i])*CharacterTable(D4[i])[NumD](g^(-1))*(Q3G!g); 
end for; end for;

for j in [1..9] do 
if (ChH[j](H!(1,6,3)) eq z3) and (ChH[j](H!(2,4,5)) eq z3) then NumH:=j; 
end if; end for;
eH:=Q3G!0;
for g in H do eH:=eH+Q3G!(1/#H)*ChH[NumH](g^(-1))*(Q3G!g); end for;

PQ<[T]>:=PolynomialRing(Rationals(),360);
PQG:=GroupAlgebra(PQ, G);
E:=[Eltseq(Basis(VectorSpace(Rationals(), 360))[i]): i in [1..360]];
tt:=PQG![T[i]:i in [1..360]];
\end{lstlisting}

One can confirm that the lemma holds by typing the following commands in the interactive mode of Magma V2.29-7.
\begin{lstlisting}[language=C, frame=single, basicstyle=\ttfamily\tiny, escapechar=\@]
> load "4-1_Def.txt";
Loading "4-1_Def.txt"

> for i,j in [1..5] do if not(i eq j) and not(#(D4[i] meet D4[j]) eq 1) 
then print([i,j]); end if; end for;

/* Lemma 4.1 (1) */
> ChG[N_3] eq Induction(CharacterTable(D4[1])[NumD],G)-Induction(ChH[NumH],G);
true

/* Lemma 4.1 (2) */
> b0:=Matrix(Rationals(),1,360,[Eltseq(e3)[i]:i in [1..360]]);
> c0:=PQG![PQ!Eltseq(&+eD4)[i]:i in [1..360]];
> time a0ij:=[Evaluate(Eltseq(tt*c0)[i], E[j]): i,j in [1..360]];
Time: 1505.640
> A0:=Matrix(Rationals(),360,360,a0ij);
> q0,x0:=Solution(Transpose(A0),b0);
> r0:=QG!Eltseq(q0);
> r0*(&+eD4) eq e3;
true
> PrintFile("Lemma41-2_r0.txt", r0);   // Export r0 to Lemma41-2_r0.txt.

/* Lemma 4.1 (3) */
> time r1:=Basis(RightAnnihilator(QG*eD4[2]+QG*eD4[3]+QG*eD4[4]+QG*eD4[5]) meet QG*eD4[1]);
Time: 1125.310
> time eD4[1]*e3 in (eD4[1]*e3*r1[1])*QG;
true
Time: 9.930
> PrintFile("Lemma41-3_r1[1].txt", r1[1]);  // Export r1[1] to Lemma41-3_r1[1].txt.

> time r2:=Basis(RightAnnihilator(QG*eD4[1]+QG*eD4[3]+QG*eD4[4]+QG*eD4[5]) meet QG*eD4[2]);
Time: 1201.680
> time eD4[2]*e3 in (eD4[2]*e3*r2[1])*QG;
true
Time: 10.790
> PrintFile("Lemma41-3_r2[1].txt", r2[1]);  // Export r2[1] to Lemma41-3_r2[1].txt. 

> time r3:=Basis(RightAnnihilator(QG*eD4[1]+QG*eD4[2]+QG*eD4[4]+QG*eD4[5]) meet QG*eD4[3]);
Time: 1369.020
> time eD4[3]*e3 in (eD4[3]*e3*r3[1])*QG;
true
Time: 11.440
> PrintFile("Lemma41-3_r3[1].txt", r3[1]);  // Export r3[1] to Lemma41-3_r3[1].txt.

> time r4:=Basis(RightAnnihilator(QG*eD4[1]+QG*eD4[2]+QG*eD4[3]+QG*eD4[5]) meet QG*eD4[4]);
Time: 1042.140
> time eD4[4]*e3 in (eD4[4]*e3*r4[1])*QG;
true
Time: 12.730
> PrintFile("Lemma41-3_r4[1].txt", r4[1]);  // Export r4[1] to Lemma41-3_r4[1].txt.

> time r5:=Basis(RightAnnihilator(QG*eD4[1]+QG*eD4[2]+QG*eD4[3]+QG*eD4[4]) meet QG*eD4[5]);
Time: 1129.750
> time eD4[5]*e3 in (eD4[5]*e3*r5[5])*QG;
true
Time: 12.070
> PrintFile("Lemma41-3_r5[5].txt", r5[5]);  // Export r5[5] to Lemma41-3_r5[5].txt.

/* Lemma 4.1 (4) */
> for i in [1..5] do 
if not(eD4[i] eq h[i]^(-1)*eD4[1]*h[i]) then print(i); end if; 
end for;

/* Lemma 4.1 (5) */
> time Q3G*eD4Q3[1]*(1-e3Q3) eq Q3G*eH*eD4Q3[1];
true
Time: 52.180

/* Lemma 4.1 (6) */
> time eH*eD4Q3[1]*Q3G eq eH*Q3G;         
true
Time: 74.030
\end{lstlisting}

%% file: ZZpG_app360-chi6.tex

\subsection{The program for Lemma \ref{360-6_Lem}}\label{360-6_Com}
Save the following program of Magma as \verb@4-2_Def.txt@.
\begin{lstlisting}[language=C, frame=single, basicstyle=\ttfamily\tiny, escapechar=\@]
A6:=AlternatingGroup(6);
ChA6:=CharacterTable(A6);

QG:=GroupAlgebra(Rationals(), A6);
e6:=QG![(ChA6[6](Id(A6))/#A6)*QG!ChA6[6](g^(-1)):g in A6];

S4_1:=sub<A6|(3,4)(5,6), (1,6,4)(2,3,5)>;
S4_2:=sub<A6|(1,2)(3,4), (1,3,6)(2,5,4)>;
S4_3:=sub<A6|(1,2)(4,5), (1,6,5)(2,4,3)>;
S4_4:=sub<A6|(1,2)(3,6), (1,4,3)(2,6,5)>;
S4_5:=sub<A6|(1,2)(5,6), (1,5,4)(2,3,6)>;
S4_6:=sub<A6|(2,5)(3,6), (1,5,6)(2,4,3)>;
S4_7:=sub<A6|(2,6)(3,4), (1,2,3)(4,6,5)>;
S4_8:=sub<A6|(1,3)(4,5), (1,4,2)(3,6,5)>;
S4_9:=sub<A6|(2,4)(5,6), (1,2,5)(3,6,4)>;
ChS4_1:=CharacterTable(S4_1);
S4a:=[S4_1,S4_2,S4_3,S4_4,S4_5,S4_6,S4_7,S4_8,S4_9];

eS4:=[QG!0:i in [1..9]];
for i in [1..9] do for g in S4a[i] do 
eS4[i]:=eS4[i]+QG!(1/#S4a[i])*(CharacterTable(S4a[i])[1](g^(-1)))*(QG!g); 
end for; end for;

A5a:=sub<A6|(1,5)(4,6), (1,3,2)(4,5,6)>;
ChA5a:=CharacterTable(A5a);
eA5_1:=QG!0;
for g in A5a do 
eA5_1:=eA5_1+QG!(1/#A5a)*(ChA5a[1](g^(-1)))*(QG!g); 
end for;

PQ<[T]>:=PolynomialRing(Rationals(),360);
PQG:=GroupAlgebra(PQ, A6);
E:=[Eltseq(Basis(VectorSpace(Rationals(), 360))[i]): i in [1..360]];
tt:=PQG![T[i]:i in [1..360]];
\end{lstlisting}

One can confirm that the lemma holds by typing the following commands in the interactive mode of Magma V2.29-7.
\begin{lstlisting}[language=C, frame=single, basicstyle=\ttfamily\tiny, escapechar=\@]
> load "4-2_Def.txt";
Loading "4-2_Def.txt"

/* Lemma 4.2 (1) */
> ChA6[6] eq Induction(ChS4_1[1],A6)-Induction(ChA5a[1],A6);
true

/* Lemma 4.2 (2) */
> b0:=Matrix(Rationals(),1,360,[Eltseq(e6)[i]:i in [1..360]]);
> c0:=PQG![PQ!Eltseq(&+eS4)[i]:i in [1..360]];
> time a0ij:=[Evaluate(Eltseq(tt*c0)[i], E[j]): i,j in [1..360]];
Time: 7702.590
> AA0:=Matrix(Rationals(),360,360,a0ij);
> q0,x0:=Solution(Transpose(AA0),b0);
> r0:=QG!Eltseq(q0);
> r0*(&+eS4) eq e6;
true
> PrintFile("Lemma42-2_r0.txt", r0);   // Export r0 to Lemma42-2_r0.txt.

/* Lemma 4.2 (3) */
> time B1:=Basis(RightAnnihilator(QG*eS4[2]+QG*eS4[3]+QG*eS4[4]+QG*eS4[5]
+QG*eS4[6]+QG*eS4[7]+QG*eS4[8]+QG*eS4[9]) meet QG*eS4[1]);
Time: 77.800
> #B1;
1
> r1:=B1[1];
> for i in [1..9] do if not(eS4[i]*r1 eq 0) then print(i); end if; end for;
1
> time eS4[1]*e6 in (eS4[1]*e6*r1)*QG;    
true
Time: 20.820
> b1:=Matrix(Rationals(),1,360,[Eltseq(eS4[1]*e6)[i]:i in [1..360]]);
> c1:=PQG![PQ!Eltseq(eS4[1]*e6*r1)[i]:i in [1..360]];
> time a1ij:=[Evaluate(Eltseq(c1*tt)[i], E[j]): i,j in [1..360]];
Time: 45584.580
> AA1:=Matrix(Rationals(),360,360,a1ij);
> q1,x1:=Solution(Transpose(AA1),b1);
> rr1:=QG!Eltseq(q1);
> eS4[1]*e6*r1*rr1 eq eS4[1]*e6;
true
> PrintFile("Lemma42-3_r1.txt", r1);     // Export r1 to Lemma42-3_r1.txt.
> PrintFile("Lemma42-3_rr1.txt", rr1);   // Export rr1 to Lemma42-3_rr1.txt.

> time B2:=Basis(RightAnnihilator(QG*eS4[1]+QG*eS4[3]+QG*eS4[4]+QG*eS4[5]
+QG*eS4[6]+QG*eS4[7]+QG*eS4[8]+QG*eS4[9]) meet QG*eS4[2]);
Time: 93.550
> #B2;
1
> r2:=B2[1];
> for i in [1..9] do if not(eS4[i]*r2 eq 0) then print(i); 
end if; end for;
2
> time eS4[2]*e6 in (eS4[2]*e6*r2)*QG;    
true
Time: 23.270
> PrintFile("Lemma42-3_r2.txt", r2);     // Export r2 to Lemma42-3_r2.txt.

> time B3:=Basis(RightAnnihilator(QG*eS4[1]+QG*eS4[2]+QG*eS4[4]+QG*eS4[5]
+QG*eS4[6]+QG*eS4[7]+QG*eS4[8]+QG*eS4[9]) meet QG*eS4[3]);
Time: 182.240
> #B3;
1
> r3:=B3[1];
> for i in [1..9] do if not(eS4[i]*r3 eq 0) then print(i); 
end if; end for;
3
> time eS4[3]*e6 in (eS4[3]*e6*r3)*QG;    
true
Time: 23.490
> PrintFile("Lemma42-3_r3.txt", r3);     // Export r3 to Lemma42-3_r3.txt.

> time B4:=Basis(RightAnnihilator(QG*eS4[1]+QG*eS4[2]+QG*eS4[3]+QG*eS4[5]
+QG*eS4[6]+QG*eS4[7]+QG*eS4[8]+QG*eS4[9]) meet QG*eS4[4]);
Time: 58.390
> #B4;
1
> r4:=B4[1];
> for i in [1..9] do if not(eS4[i]*r4 eq 0) then print(i); 
end if; end for;
4
> time eS4[4]*e6 in (eS4[4]*e6*r4)*QG;    
true
Time: 22.190
> PrintFile("Lemma42-3_r4.txt", r4);     // Export r4 to Lemma42-3_r4.txt.

> time B5:=Basis(RightAnnihilator(QG*eS4[1]+QG*eS4[2]+QG*eS4[3]+QG*eS4[4]
+QG*eS4[6]+QG*eS4[7]+QG*eS4[8]+QG*eS4[9]) meet QG*eS4[5]);
Time: 79.140
> #B5;
2
> r5:=B5[1];
> for i in [1..9] do if not(eS4[i]*r5 eq 0) then print(i); 
end if; end for;
5
> time eS4[5]*e6 in (eS4[5]*e6*r5)*QG;    
true
Time: 24.210
> PrintFile("Lemma42-3_r5.txt", r5);     // Export r5 to Lemma42-3_r5.txt.

> time B6:=Basis(RightAnnihilator(QG*eS4[1]+QG*eS4[2]+QG*eS4[3]+QG*eS4[4]
+QG*eS4[5]+QG*eS4[7]+QG*eS4[8]+QG*eS4[9]) meet QG*eS4[6]);
Time: 96.690
> #B6;
1
> r6:=B6[1];
> for i in [1..9] do if not(eS4[i]*r6 eq 0) then print(i); 
end if; end for;
6
> time eS4[6]*e6 in (eS4[6]*e6*r6)*QG;    
true
Time: 23.130
> PrintFile("Lemma42-3_r6.txt", r6);     // Export r6 to Lemma42-3_r6.txt.

> time B7:=Basis(RightAnnihilator(QG*eS4[1]+QG*eS4[2]+QG*eS4[3]+QG*eS4[4]
+QG*eS4[5]+QG*eS4[6]+QG*eS4[8]+QG*eS4[9]) meet QG*eS4[7]);
Time: 167.040
> #B7;
1
> r7:=B7[1];
> for i in [1..9] do if not(eS4[i]*r7 eq 0) then print(i); 
end if; end for;
7
> time eS4[7]*e6 in (eS4[7]*e6*r7)*QG;    
true
Time: 24.180
> PrintFile("Lemma42-3_r7.txt", r7);     // Export r7 to Lemma42-3_r7.txt.

> time B8:=Basis(RightAnnihilator(QG*eS4[1]+QG*eS4[2]+QG*eS4[3]+QG*eS4[4]
+QG*eS4[5]+QG*eS4[6]+QG*eS4[7]+QG*eS4[9]) meet QG*eS4[8]);
Time: 97.730
> #B8;
1
> r8:=B8[1];
> for i in [1..9] do if not(eS4[i]*r8 eq 0) then print(i); 
end if; end for;
8
> time eS4[8]*e6 in (eS4[8]*e6*r8)*QG;    
true
Time: 23.350
> PrintFile("Lemma42-3_r8.txt", r8);     // Export r8 to Lemma42-3_r8.txt.

> time B9:=Basis(RightAnnihilator(QG*eS4[1]+QG*eS4[2]+QG*eS4[3]+QG*eS4[4]
+QG*eS4[5]+QG*eS4[6]+QG*eS4[7]+QG*eS4[8]) meet QG*eS4[9]);
Time: 78.660
> #B9;
1
> r9:=B9[1];
> for i in [1..9] do if not(eS4[i]*r9 eq 0) then print(i); 
end if; end for;
9
> time eS4[9]*e6 in (eS4[9]*e6*r9)*QG;    
true
Time: 23.740
> PrintFile("Lemma42-3_r9.txt", r9);     // Export r8 to Lemma42-3_r9.txt.

/* Lemma 4.2 (4) */
> for g in A6 do if (eS4[2] eq g^(-1)*eS4[1]*g) and (g^2 eq A6!1) then g; 
end if; end for;
(2, 4)(3, 5)
(1, 3)(4, 6)
(1, 5)(2, 6)

> for g in A6 do if (eS4[3] eq g^(-1)*eS4[1]*g) and (g^2 eq A6!1) then g; 
end if; end for;
(2, 4)(3, 6)
(1, 5)(3, 6)

> for g in A6 do if (eS4[4] eq g^(-1)*eS4[1]*g) and (g^2 eq A6!1) then g; 
end if; end for;
(1, 3)(4, 5)
(2, 6)(4, 5)

> for g in A6 do if (eS4[5] eq g^(-1)*eS4[1]*g) and (g^2 eq A6!1) then g; 
end if; end for;
(1, 3)(2, 4)
(1, 5)(4, 6)
(2, 6)(3, 5)

> for g in A6 do if (eS4[6] eq g^(-1)*eS4[1]*g) and (g^2 eq A6!1) then g; 
end if; end for;
(2, 4)(5, 6)
(1, 6)(3, 4)
(1, 5)(2, 3)

> for g in A6 do if (eS4[7] eq g^(-1)*eS4[1]*g) and (g^2 eq A6!1) then g; 
end if; end for;
(1, 4)(3, 6)
(2, 5)(3, 6)

> for g in A6 do if (eS4[8] eq g^(-1)*eS4[1]*g) and (g^2 eq A6!1) then g; 
end if; end for;
(1, 4)(2, 3)
(2, 5)(4, 6)
(1, 6)(3, 5)

> for g in A6 do if (eS4[9] eq g^(-1)*eS4[1]*g) and (g^2 eq A6!1) then g; 
end if; end for;
(2, 3)(4, 5)
(1, 6)(4, 5)

/* Lemma 4.2 (5) */
> time QG*eS4[1]*(QG!1-e6) eq QG*eA5_1*eS4[1];
true
Time: 1.700

/* Lemma 4.2 (6) */
> time eA5_1*eS4[1]*QG eq eA5_1*QG;
true
Time: 2.490
\end{lstlisting}

%% file: ZZpG_app360-chi7.tex

\subsection{The program for Lemma \ref{360-7_Lem}}\label{360-7_Com}
Save the following program of Magma as \verb@4-3_Def.txt@.
\begin{lstlisting}[language=C, frame=single, basicstyle=\ttfamily\tiny, escapechar=\@]
A6:=AlternatingGroup(6);
ChA6:=CharacterTable(A6);
Q4<z4>:=CyclotomicField(4);

H1:=sub<A6|(3,5)(4,6),(1,2)(3,4,5,6),(1,4,6)>;
H:=[X: X in {H1^g:g in A6}];
ChH1:=CharacterTable(H1);

Q4A6:=GroupAlgebra(Q4, A6);
e7:=Q4A6![(ChA6[7](Id(A6))/#A6)*Q4A6!ChA6[7](g^(-1)):g in A6];

for i in[1..6] do
if (ChH1[i](H1!(1,2)(3,4,5,6)) eq z4)
then Num:=i;
end if; end for;

eH:=[Q4A6!0:i in [1..10]];
for i in [1..10] do for g in H[i] do 
eH[i]:=eH[i]+Q4A6!(1/#H[i])*(CharacterTable(H[i])[Num](g^(-1)))*(Q4A6!g); 
end for; end for;
\end{lstlisting}

One can confirm that the lemma holds by typing the following commands in the interactive mode of Magma V2.29-7.
\begin{lstlisting}[language=C, frame=single, basicstyle=\ttfamily\tiny, escapechar=\@]
> load "4-3_Def.txt";
Loading "4-3_Def.txt"

/* Lemma 4.3 (1) */
> ChA6[7] eq Induction(ChH1[Num],A6);
true

/* Lemma 4.3 (2) */
> e7 eq &+eH;
true

/* Lemma 4.3 (3) */
> for i, j in [1..10] do if (i lt j) and not(eH[i]*eH[j] eq Q4A6!0) then print([i,j]); 
end if; end for;
> for i in [1..10] do if not(eH[i]*eH[i] eq eH[i]) then print(i); end if; end for;

/* Lemma 4.3 (4) */
> gg2:=[];
> for g in A6 do 
if (eH[2] eq g^(-1)*eH[1]*g) and (g^2 eq A6!1) then Append(~gg2, g); 
end if; end for;
> gg2[4];
(1, 2)(3, 4)

> gg3:=[];
> for g in A6 do 
if (eH[3] eq g^(-1)*eH[1]*g) and (g^2 eq A6!1) then Append(~gg3, g); 
end if; end for;
> gg3[1];
(1, 4)(3, 6)

> gg4:=[];
> for g in A6 do 
if (eH[4] eq g^(-1)*eH[1]*g) and (g^2 eq A6!1) then Append(~gg4, g); 
end if; end for;                            
> gg4[1];
(2, 3)(4, 5)

> gg5:=[];
> for g in A6 do 
if (eH[5] eq g^(-1)*eH[1]*g) and (g^2 eq A6!1) then Append(~gg5, g); 
end if; end for;
> gg5[1];
(1, 3)(4, 5)

> gg6:=[];
> for g in A6 do 
if (eH[6] eq g^(-1)*eH[1]*g) and (g^2 eq A6!1) then Append(~gg6, g); 
end if; end for;
> gg6[1];
(1, 3)(2, 5)

> gg7:=[];
> for g in A6 do 
if (eH[7] eq g^(-1)*eH[1]*g) and (g^2 eq A6!1) then Append(~gg7, g); 
end if; end for;
> gg7[1];
(2, 4)(3, 5)

> gg8:=[];
> for g in A6 do 
if (eH[8] eq g^(-1)*eH[1]*g) and (g^2 eq A6!1) then Append(~gg8, g); 
end if; end for;
> gg8[1];
(1, 5)(4, 6)

> gg9:=[];
> for g in A6 do 
if (eH[9] eq g^(-1)*eH[1]*g) and (g^2 eq A6!1) then Append(~gg9, g); 
end if; end for;
> gg9[1];
(2, 5)(3, 4)

> gg10:=[];
> for g in A6 do 
if (eH[10] eq g^(-1)*eH[1]*g) and (g^2 eq A6!1) then Append(~gg10, g); 
end if; end for;
> gg10[1];
(3, 4)(5, 6)
\end{lstlisting}

%% file: ZZpG_app360-chi4+5.tex

\subsection{The program for Lemma \ref{360-4_Lem}}
\label{360_Others}
Save the following program of Magma as \verb@4-4_Def.txt@.
\begin{lstlisting}[language=C, frame=single, basicstyle=\ttfamily\tiny, escapechar=\@]
A6:=AlternatingGroup(6);
ChA6:=CharacterTable(A6);
Q15<z15>:=CyclotomicField(15);
z3:=(z15)^5;
z5:=(z15)^3;
pHI:= -(z5^2+z5^3);
R:=GroupAlgebra(Q15,A6);
c4:=A6!(1,2,3,4,5);
c5:=A6!(1,3,4,5,2);

for i in [1..7] do if (ChA6[i](c5) eq pHI) then NN_4:=i; 
end if; end for;
for i in [1..7] do if (ChA6[i](c4) eq pHI) then NN_5:=i; 
end if; end for;
e4:=R![(ChA6[NN_4](Id(A6))/#A6)*R!(ChA6[NN_4](g^(-1))):g in A6];
e5:=R![(ChA6[NN_5](Id(A6))/#A6)*R!(ChA6[NN_5](g^(-1))):g in A6];

H:=sub<A6|(1,2,3),(4,5,6)>;
ChH:=CharacterTable(H);

eH:=[R!0:j in [1..#H]];
for i in [1..#H] do 
for g in H do
eH[i]:=eH[i]+R!(1/9)*(ChH[i](g^(-1)))*(R!g); 
end for;
end for;

Num:=[0:j in [1..#H]];
/* 1 = (00)_3 +1 */
for i in [1..9] do if (ChH[i](A6!(1,2,3)) eq z3^0) and (ChH[i](A6!(4,5,6)) eq z3^0) 
then Num[1] := i; end if; end for;

/* 2 = (01)_3 +1 */
for i in [1..9] do if (ChH[i](A6!(1,2,3)) eq z3^0) and (ChH[i](A6!(4,5,6)) eq z3^1) 
then Num[2] := i; end if; end for;

/* 3 = (02)_3 +1 */
for i in [1..9] do if (ChH[i](A6!(1,2,3)) eq z3^0) and (ChH[i](A6!(4,5,6)) eq z3^2) 
then Num[3] := i; end if; end for;
 
/* 4 = (10)_3 +1 */
for i in [1..9] do if (ChH[i](A6!(1,2,3)) eq z3^1) and (ChH[i](A6!(4,5,6)) eq z3^0) 
then Num[4] := i; end if; end for;

/* 5 = (11)_3 +1 */
for i in [1..9] do if (ChH[i](A6!(1,2,3)) eq z3^1) and (ChH[i](A6!(4,5,6)) eq z3^1) 
then Num[5] := i; end if; end for;

/* 6 = (12)_3 +1 */
for i in [1..9] do if (ChH[i](A6!(1,2,3)) eq z3^1) and (ChH[i](A6!(4,5,6)) eq z3^2) 
then Num[6] := i; end if; end for;

/* 7 = (20)_3 +1 */
for i in [1..9] do if (ChH[i](A6!(1,2,3)) eq z3^2) and (ChH[i](A6!(4,5,6)) eq z3^0) 
then Num[7] := i; end if; end for;

/* 8 = (21)_3 +1 */
for i in [1..9] do if (ChH[i](A6!(1,2,3)) eq z3^2) and (ChH[i](A6!(4,5,6)) eq z3^1) 
then Num[8] := i; end if; end for;

/* 9 = (22)_3 +1 */
for i in [1..9] do if (ChH[i](A6!(1,2,3)) eq z3^2) and (ChH[i](A6!(4,5,6)) eq z3^2) 
then Num[9] := i; end if; end for;

C5:=sub<A6|(1,2,3,4,5)>;
ChC5:=CharacterTable(C5);

eC5:=[R!0:j in [1..#C5]];
for i in [1..#C5] do 
for g in C5 do
eC5[i]:=eC5[i]+R!(1/5)*(ChC5[i](g^(-1)))*(R!g); 
end for;
end for;

for i in [1..5] do if (ChC5[i](c4) eq z5^0) then N_0 :=i; end if; end for;
for i in [1..5] do if (ChC5[i](c4) eq z5^1) then N_1 :=i; end if; end for;
for i in [1..5] do if (ChC5[i](c4) eq z5^2) then N_2 :=i; end if; end for;
for i in [1..5] do if (ChC5[i](c4) eq z5^3) then N_3 :=i; end if; end for;
for i in [1..5] do if (ChC5[i](c4) eq z5^4) then N_4 :=i; end if; end for;

d2:=eH[Num[2]]*(R!1-e4);
\end{lstlisting}

One can confirm that the proposition holds by typing the following commands in the interactive mode of Magma V2.29-7.
\begin{lstlisting}[language=C, frame=single, basicstyle=\ttfamily\tiny, escapechar=\@]
> load "4-4_Def.txt";
Loading "4-4_Def.txt"

/* Prop 4.4 (1) */
> ChA6[NN_4] eq Induction(ChH[Num[2]], A6) + Induction(ChH[Num[5]], A6) 
- Induction(ChC5[N_1], A6);
true
> ChA6[NN_4] eq Induction(ChH[Num[2]], A6) + Induction(ChH[Num[5]], A6) 
- Induction(ChC5[N_4], A6);
true

/* Prop 4.4 (2) */
> e4 eq (&+eH - eH[Num[1]])*e4;
true
> e5 eq (&+eH - eH[Num[1]])*e5;
true

/* Prop 4.4 (3) */
> gg3:=[];
> for g in A6 do if (eH[Num[3]] eq g^(-1)*eH[Num[2]]*g) then Append(~gg3, g); 
end if; end for;
> gg3[1];
(2, 3)(4, 5)

> gg4:=[];
> for g in A6 do if (eH[Num[4]] eq g^(-1)*eH[Num[2]]*g) then Append(~gg4, g); 
end if; end for;
> gg4[1];
(1, 4)(2, 6, 3, 5)

> gg7:=[];
> for g in A6 do if (eH[Num[7]] eq g^(-1)*eH[Num[2]]*g) then Append(~gg7, g); 
end if; end for;
> gg7[1];
(1, 4, 2, 5)(3, 6)

> gg6:=[];
> for g in A6 do if (eH[Num[6]] eq g^(-1)*eH[Num[5]]*g) then Append(~gg6, g); 
end if; end for;
> gg6[1];
(1, 4)(2, 6, 3, 5)

> gg8:=[];
> for g in A6 do if (eH[Num[8]] eq g^(-1)*eH[Num[5]]*g) then Append(~gg8, g); 
end if; end for;
> gg8[1];
(1, 4, 2, 5)(3, 6)

> gg9:=[];
> for g in A6 do if (eH[Num[9]] eq g^(-1)*eH[Num[5]]*g) then Append(~gg9, g); 
end if; end for;
> gg9[1];
(2, 3)(4, 5)

> for g in A6 do if (eH[Num[5]] eq g^(-1)*eH[Num[2]]*g) then g; end if; end for;

/* Prop 4.4 (4) */
> time eH[Num[2]]*e4*R eq eH[Num[2]]*eC5[N_1]*eH[Num[5]]*e4*R;
true 
Time: 352.890

/* Prop 4.4 (5) */
> time R*e4*eH[Num[5]] eq R*e4*eH[Num[2]]*eC5[N_1]*eH[Num[5]];
true
Time: 17335.500

/* Prop 4.4 (6) */
> time (eH[Num[5]] + d2)*R eq (eH[Num[5]] + d2)*eC5[N_1]*R;
true
Time: 27331.710

/* Prop 4.4 (7) */
> time R*eC5[N_1] eq R*(eH[Num[5]] + d2)*eC5[N_1];
true
Time: 6345.960
\end{lstlisting}

%% file: ZZpG_Output_2-1.tex

\subsection{Output Files}\label{Output}

The text files output by the programs so far are provided below.

\begin{lstlisting}[language=C, frame=single, basicstyle=\ttfamily\tiny, escapechar=\@, caption=, title={Lemma21-4\_q1.txt}]
1/8*(-z^5 - z^4 - z^3 - z^2 - z - 1)*(1, 6, 8, 3)(2, 5, 7, 4) + 1/8*(1, 3)(2, 
    4)(5, 6)(7, 8) - 1/8*z^5*(1, 4, 6, 3)(2, 8, 5, 7) + 1/8*z^2*(1, 7, 4, 8, 5, 
    3, 2) + 1/8*z^5*(1, 8, 3, 6, 5, 2, 4) - 1/8*z^2*(2, 7, 3, 8, 6, 5, 4) - 
    1/8*(1, 4)(2, 3)(5, 8)(6, 7) + 1/8*(z^5 + z^4 + z^3 + z^2 + z + 1)*(1, 5, 
    4)(2, 6, 3)
\end{lstlisting}

\begin{lstlisting}[language=C, frame=single, basicstyle=\ttfamily\tiny, escapechar=\@, caption=, title={Lemma21-4\_q2.txt}]
-z^4*(1, 3, 2, 5)(4, 7, 6, 8) + (-z^5 - z^4 - z^3 - z^2 - z - 1)*(1, 8, 2, 4, 3,
    7, 5) + 1/203*(z^5 + 50*z^4 - 14*z^3 + 40*z^2 - 40*z - 16)*(1, 2, 3, 6)(4, 
    8, 7, 5) + 1/203*(z^5 + 50*z^4 - 14*z^3 + 40*z^2 - 40*z - 16)*(2, 4, 3)(6, 
    7, 8) - z*(1, 8, 4, 3)(2, 5, 6, 7) + z^3*(1, 6, 3, 2)(4, 5, 7, 8) + z^2*(1, 
    7, 4, 8, 5, 3, 2) + 1/203*(-16*z^5 - 17*z^4 - 66*z^3 - 2*z^2 - 56*z + 
    24)*(1, 3, 2)(5, 6, 7) + 1/203*(-49*z^5 + 15*z^4 - 39*z^3 + 41*z^2 + 17*z + 
    1)*(2, 6, 4)(3, 5, 8) + 1/203*(80*z^5 + 56*z^4 + 40*z^3 + 39*z^2 - 10*z + 
    54)*(1, 2, 5)(3, 4, 8) + 1/203*(54*z^5 - 26*z^4 - 2*z^3 + 14*z^2 + 15*z + 
    64)*(1, 6)(2, 4)(3, 7)(5, 8) + z^5*(1, 7, 6)(2, 5, 4) + 1/203*(64*z^5 + 
    10*z^4 + 90*z^3 + 66*z^2 + 50*z + 49)*(1, 3, 6)(2, 8, 4) + 1/203*(-54*z^5 + 
    26*z^4 + 2*z^3 - 14*z^2 - 15*z - 64)*(2, 7, 3, 8, 6, 5, 4) + 1/203*(-49*z^5 
    + 15*z^4 - 39*z^3 + 41*z^2 + 17*z + 1)*(1, 8, 4)(2, 7, 3) + 1/203*(-24*z^5 -
    40*z^4 - 41*z^3 - 90*z^2 - 26*z - 80)*(1, 2, 7, 3)(4, 8, 5, 6) - (1, 4)(2, 
    3)(5, 8)(6, 7) + 1/203*(-24*z^5 - 40*z^4 - 41*z^3 - 90*z^2 - 26*z - 80)*(1, 
    5, 4)(2, 6, 3) + 1/203*(16*z^5 + 17*z^4 + 66*z^3 + 2*z^2 + 56*z - 24)*(1, 6,
    7, 5)(2, 3, 8, 4) + 1/203*(-80*z^5 - 56*z^4 - 40*z^3 - 39*z^2 + 10*z - 
    54)*(1, 7, 8, 5)(2, 6, 3, 4) + 1/203*(-64*z^5 - 10*z^4 - 90*z^3 - 66*z^2 - 
    50*z - 49)*(1, 3, 5)(2, 7, 4)
\end{lstlisting}

\begin{lstlisting}[language=C, frame=single, basicstyle=\ttfamily\tiny, escapechar=\@, caption=, title={Lemma21-5\_q3.txt}]
1/232*(-173*z^5 + 50*z^4 + 102*z^3 - 395*z^2 - 69*z - 509)*(1, 8, 7, 4)(2, 3, 5,
    6) + 1/232*(-189*z^5 + 207*z^4 - 457*z^3 + 38*z^2 - 415*z - 108)*(1, 4)(2, 
    6)(3, 8)(5, 7) + 1/116*(69*z^5 + 57*z^4 + 223*z^3 + 63*z^2 + 53*z + 172)*(1,
    5, 3, 4)(2, 7, 8, 6) + 1/232*(214*z^5 - 175*z^4 + 107*z^3 - 53*z^2 + 24*z + 
    317)*(1, 6, 8, 3)(2, 5, 7, 4) + 1/8*(-15*z^5 - 10*z^4 - 13*z^3 - 17*z^2 - 
    8*z - 14)*(1, 7, 5, 3)(2, 8, 6, 4) + 1/232*(-59*z^5 - 137*z^4 - 479*z^3 - 
    98*z^2 - 163*z + 306)*(1, 3)(2, 4)(5, 6)(7, 8) + 1/232*(-381*z^5 + 90*z^4 + 
    27*z^3 - 73*z^2 - 14*z - 342)*(1, 6, 3, 2)(4, 5, 7, 8) + 1/232*(322*z^5 + 
    179*z^4 - 100*z^3 + 381*z^2 - 149*z + 242)*(1, 7, 4, 8, 5, 3, 2) + 
1/232*(59*z^5 - 269*z^4 + 73*z^3 - 308*z^2 + 163*z + 100)*(1, 3, 2)(5, 6, 7) + 
1/232*(713*z^5 - 78*z^4 + 371*z^3 + 303*z^2 + 306*z + 772)*(1, 4)(2, 5)(3, 7)(6,
    8) + 1/116*(-106*z^5 - 196*z^4 - 53*z^3 - 209*z^2 + 64*z - 102)*(1, 6)(2, 
    4)(3, 7)(5, 8) + 1/232*(-75*z^5 + 49*z^4 - 342*z^3 + 306*z^2 - 16*z + 
    127)*(1, 7, 6)(2, 5, 4) + 1/232*(511*z^5 - 28*z^4 + 357*z^3 + 343*z^2 + 
    266*z + 350)*(1, 3, 6)(2, 8, 4) + 1/232*(-425*z^5 - 138*z^4 - 140*z^3 - 
    557*z^2 + 151*z - 711)*(2, 7, 3, 8, 6, 5, 4) + 1/232*(658*z^5 + 275*z^4 + 
    706*z^3 + 481*z^2 + 447*z + 492)*(1, 8, 4)(2, 7, 3) + 1/232*(20*z^5 + 
    391*z^4 - 77*z^3 + 597*z^2 + 12*z + 289)*(1, 2, 7, 3)(4, 8, 5, 6) + 
1/116*(-166*z^5 - 6*z^4 + 4*z^3 - 115*z^2 + 57*z - 418)*(1, 4)(2, 3)(5, 8)(6, 7)
+ 1/232*(-46*z^5 + 339*z^4 + 35*z^3 + 393*z^2 - 190*z + 127)*(1, 5, 4)(2, 6, 3) 
+ 1/232*(79*z^5 + 180*z^4 + 170*z^3 + 231*z^2 - 57*z + 41)*(1, 6, 7, 5)(2, 3, 8,
    4) + 1/116*(123*z^5 - 56*z^4 + 163*z^3 - 39*z^2 + 10*z - 54)*(1, 7, 8, 5)(2,
    6, 3, 4) + 1/116*(-166*z^5 - 6*z^4 - 199*z^3 - 115*z^2 - 146*z - 215)*(1, 3,
    5)(2, 7, 4)
\end{lstlisting}

\begin{lstlisting}[language=C, frame=single, basicstyle=\ttfamily\tiny, escapechar=\@, caption=, title={Lemma21-5\_q4.txt}]
1/7*Id(G) + 1/7*z^3*(1, 8, 2)(4, 5, 6) - 1/7*z^4*(1, 5, 3, 7, 6, 8, 2) + 
1/7*(-z^5 - z^4 - z^3 - z^2 - z - 1)*(3, 6, 7)(4, 5, 8) + 1/7*z^2*(1, 2, 8, 
    4)(3, 7, 5, 6) - 1/7*z^5*(1, 3, 7, 8, 6, 4, 5) - 1/7*z*(1, 3)(2, 6)(4, 7)(5,
    8)
\end{lstlisting}

%% file: ZZpG_Output_2-2.tex

\begin{lstlisting}[language=C, frame=single, basicstyle=\ttfamily\tiny, escapechar=\@, caption=, title={Lemma22-2\_r0.txt}]
-69/14*(1, 3, 7, 2)(4, 6, 5, 8) - 27/14*(1, 5, 7, 3, 4, 2, 8) + 27/14*(1, 2)(3, 
    8)(4, 7)(5, 6) + 577/224*(1, 5, 4, 6, 8, 7, 3) - 745/336*(1, 6, 4, 3, 2, 8, 
    5) - 349/96*(1, 7, 3, 2, 4, 6, 5) + 1695/224*(1, 3, 2, 5)(4, 7, 6, 8) + 
1991/168*(2, 7, 8)(3, 5, 6) - 4261/672*(1, 8)(2, 7)(3, 6)(4, 5) - 1367/336*(1, 
    2, 7)(3, 4, 6) - 73/21*(1, 4, 8)(2, 3, 7) + 67/12*(1, 5, 8)(2, 6, 7) + 
23/4*(1, 6, 5, 4)(2, 3, 7, 8) + 817/672*(1, 7, 4)(2, 6, 8) - 1513/672*(1, 3, 6, 
    4)(2, 7, 5, 8) + 113/168*(2, 3, 8)(4, 5, 7) + 2851/672*(1, 8)(2, 3)(4, 6)(5,
    7) + 2321/672*(1, 2, 3)(5, 7, 6) - 4043/672*(1, 4, 7, 8)(2, 6, 5, 3) - 
139/84*(1, 5, 6, 8)(2, 7, 4, 3) + 2887/336*(1, 6, 4, 7)(2, 5, 3, 8) - 29/84*(1, 
    7)(2, 8)(3, 4)(5, 6) - 419/224*(1, 3, 5, 7)(2, 4, 6, 8) + 3967/672*(2, 7, 
    5)(4, 6, 8) + 2/3*(1, 8, 5)(2, 7, 6) + 53/336*(1, 2, 7, 4, 5, 8, 6) + 
361/84*(1, 4, 5)(2, 3, 6) + 607/672*(1, 5)(2, 6)(3, 7)(4, 8) - 4141/672*(1, 6, 
    8)(2, 3, 5) - 139/168*(1, 7, 3, 8)(2, 6, 4, 5) + 3005/336*(1, 3, 4, 8)(2, 7,
    6, 5) + 349/96*(2, 3, 6, 4, 7, 8, 5) - 6473/672*(1, 8, 6, 5)(2, 3, 4, 7) + 
1249/168*(1, 2, 3, 5, 8, 4, 7) - 173/96*(1, 4, 3, 5)(2, 6, 8, 7) + 6817/672*(1, 
    5)(2, 7)(3, 8)(4, 6) + 3355/672*(1, 6)(2, 5)(3, 4)(7, 8) + 613/168*(1, 7, 4,
    6)(2, 8, 3, 5) - 1601/112*(1, 3, 8, 6)(2, 4, 7, 5) - 2197/336*(2, 6, 5)(3, 
    7, 4) + 455/32*(1, 8, 2, 4, 3, 7, 5) - 290/21*(1, 2, 5, 8)(3, 7, 6, 4) - 
3271/224*(1, 4, 6, 7, 8, 2, 5) - 1199/672*(1, 5)(2, 8)(3, 6)(4, 7) + 
4843/672*(1, 6, 7)(2, 4, 5) - 3737/336*(1, 7)(2, 5)(3, 6)(4, 8) + 391/112*(1, 3,
    7)(2, 8, 5) - 403/672*(2, 3, 4)(6, 8, 7) - 971/224*(1, 8, 7, 4)(2, 3, 5, 6) 
- 1637/336*(1, 2, 3, 6)(4, 8, 7, 5) + 169/672*(1, 4)(2, 6)(3, 8)(5, 7) - 
709/84*(1, 5, 3, 4)(2, 7, 8, 6) - 377/336*(1, 6, 8, 3)(2, 5, 7, 4) + 531/112*(1,
    7, 5, 3)(2, 8, 6, 4) - 1963/336*(1, 3)(2, 4)(5, 6)(7, 8) + 1457/672*(2, 4, 
    3)(6, 7, 8) + 1243/84*(1, 8, 4, 3)(2, 5, 6, 7) - 89/112*(1, 2, 6, 7)(3, 8, 
    5, 4) - 1439/224*(1, 4, 6, 3)(2, 8, 5, 7) - 4331/672*(1, 5, 3)(2, 4, 7) + 
431/112*(1, 6, 3, 2)(4, 5, 7, 8) + 4331/672*(1, 7, 4, 8, 5, 3, 2) - 829/84*(1, 
    3, 2)(5, 6, 7) - 325/168*(2, 6, 4)(3, 5, 8) + 215/24*(1, 8, 3, 6, 5, 2, 4) -
3/7*(1, 2, 5)(3, 4, 8) + 1751/672*(1, 4)(2, 5)(3, 7)(6, 8) - 7771/672*(1, 5, 2, 
    8, 7, 6, 4) + 473/224*(1, 6)(2, 4)(3, 7)(5, 8) + 3431/672*(1, 7, 6)(2, 5, 4)
+ 3295/336*(1, 3, 6)(2, 8, 4) + 839/672*(2, 7, 3, 8, 6, 5, 4) - 691/112*(1, 8, 
    4)(2, 7, 3) - 1817/336*(1, 2, 7, 3)(4, 8, 5, 6) + 841/96*(1, 4)(2, 3)(5, 
    8)(6, 7) - 2551/672*(1, 5, 4)(2, 6, 3) - 211/32*(1, 6, 7, 5)(2, 3, 8, 4) + 
7513/672*(1, 7, 8, 5)(2, 6, 3, 4) - 2405/336*(1, 3, 5)(2, 7, 4)
\end{lstlisting}

\begin{lstlisting}[language=C, frame=single, basicstyle=\ttfamily\tiny, escapechar=\@, caption=, title={Lemma22-3\_r1.txt}]
(1, 3, 4)(5, 6, 8) + (1, 8, 5, 2)(3, 4, 6, 7) + (1, 5, 4, 2)(3, 6, 7, 8) + (1, 
    7, 5)(3, 6, 8) + (1, 8, 7, 5, 2, 6, 3) - (1, 5, 2, 3)(4, 8, 6, 7) - (1, 
    7)(2, 3)(4, 5)(6, 8) - (1, 3)(2, 8)(4, 5)(6, 7) + (1, 5, 4, 6, 8, 7, 3) + 
    (1, 7, 3, 2, 4, 6, 5) - (1, 3, 2, 5)(4, 7, 6, 8) - (1, 8)(2, 7)(3, 6)(4, 5) 
- (1, 5, 6, 8)(2, 7, 4, 3) - (1, 7)(2, 8)(3, 4)(5, 6) - (1, 3, 4, 8)(2, 7, 6, 5)
- (1, 8, 6, 5)(2, 3, 4, 7) - (1, 5)(2, 8)(3, 6)(4, 7) - (1, 7)(2, 5)(3, 6)(4, 8)
+ (1, 3)(2, 4)(5, 6)(7, 8) - (1, 8, 4, 3)(2, 5, 6, 7) + (1, 3, 2)(5, 6, 7) + (1,
    8, 3, 6, 5, 2, 4) + (1, 5, 4)(2, 6, 3) + (1, 7, 8, 5)(2, 6, 3, 4)
\end{lstlisting}

\begin{lstlisting}[language=C, frame=single, basicstyle=\ttfamily\tiny, escapechar=\@, caption=, title={Lemma22-3\_rr1.txt}]
1/14*(1, 3, 6)(2, 8, 4) + 1/14*(2, 7, 3, 8, 6, 5, 4) + 1/14*(1, 8, 4)(2, 7, 3) +
1/14*(1, 4)(2, 3)(5, 8)(6, 7) + 1/14*(1, 5, 4)(2, 6, 3) + 1/14*(1, 3, 5)(2, 7, 
    4)
\end{lstlisting}

\begin{lstlisting}[language=C, frame=single, basicstyle=\ttfamily\tiny, escapechar=\@, caption=, title={Lemma22-3\_r2.txt}]
(1, 7, 6, 5, 3, 8, 4) + (1, 5, 2)(3, 8, 4) + (1, 4, 2)(5, 8, 7) + (1, 6, 5)(4, 
    8, 7) + (1, 5, 2, 3)(4, 8, 6, 7) - (1, 4, 2, 7)(3, 6, 8, 5) - (1, 6)(2, 
    7)(3, 5)(4, 8) - (1, 7, 2, 4)(3, 5, 8, 6) - (1, 5, 8, 7)(2, 4, 3, 6) + (1, 
    4, 3)(5, 8, 6) + (1, 6, 4, 3, 2, 8, 5) - (1, 5)(2, 6)(3, 7)(4, 8) - (1, 4, 
    3, 5)(2, 6, 8, 7) - (1, 6)(2, 5)(3, 4)(7, 8) - (1, 7)(2, 5)(3, 6)(4, 8) - 
    (1, 5, 3, 4)(2, 7, 8, 6) + (1, 7, 5, 3)(2, 8, 6, 4) + (1, 7, 4, 8, 5, 3, 2) 
- (1, 4)(2, 5)(3, 7)(6, 8) + (1, 5, 2, 8, 7, 6, 4) - (1, 6)(2, 4)(3, 7)(5, 8) + 
(1, 4)(2, 3)(5, 8)(6, 7) + (1, 6, 7, 5)(2, 3, 8, 4) - (1, 7, 8, 5)(2, 6, 3, 4)
\end{lstlisting}

\begin{lstlisting}[language=C, frame=single, basicstyle=\ttfamily\tiny, escapechar=\@, caption=, title={Lemma22-3\_rr2.txt}]
-1/14*(1, 7, 8, 5)(2, 6, 3, 4)
\end{lstlisting}

\begin{lstlisting}[language=C, frame=single, basicstyle=\ttfamily\tiny, escapechar=\@, caption=, title={Lemma22-3\_r3.txt}]
(1, 5, 3, 7, 6, 8, 2) - (1, 5, 4, 2)(3, 6, 7, 8) + (1, 6, 2, 7, 5, 4, 3) - (1, 
    2, 4, 5)(3, 8, 7, 6) + (1, 2)(3, 4)(5, 7)(6, 8) - (1, 5, 7, 6)(2, 4, 8, 3) +
(1, 2, 4)(5, 7, 8) - (1, 6)(2, 7)(3, 5)(4, 8) - (1, 2)(3, 5)(4, 6)(7, 8) + (1, 
    6, 8, 5, 4, 7, 2) - (1, 8)(2, 4)(3, 5)(6, 7) - (1, 8)(2, 7)(3, 6)(4, 5) + 
    (1, 6, 5, 4)(2, 3, 7, 8) - (1, 8)(2, 3)(4, 6)(5, 7) + (1, 2, 3)(5, 7, 6) + 
    (1, 8, 5)(2, 7, 6) + (1, 8, 6, 5)(2, 3, 4, 7) - (1, 5)(2, 7)(3, 8)(4, 6) + 
    (1, 8, 2, 4, 3, 7, 5) - (1, 2, 3, 6)(4, 8, 7, 5) + (1, 5, 3, 4)(2, 7, 8, 6) 
+ (1, 5, 3)(2, 4, 7) - (1, 6, 3, 2)(4, 5, 7, 8) - (1, 6, 7, 5)(2, 3, 8, 4)
\end{lstlisting}

\begin{lstlisting}[language=C, frame=single, basicstyle=\ttfamily\tiny, escapechar=\@, caption=, title={Lemma22-3\_rr3.txt}]
-1/14*(1, 6, 7, 5)(2, 3, 8, 4)
\end{lstlisting}

\begin{lstlisting}[language=C, frame=single, basicstyle=\ttfamily\tiny, escapechar=\@, caption=, title={Lemma22-3\_r4.txt}]
(1, 2, 8, 6, 7, 3, 5) + (1, 2, 4, 5)(3, 8, 7, 6) - (1, 2)(3, 4)(5, 7)(6, 8) + 
    (1, 6, 2)(3, 8, 7) + (1, 7, 6, 2)(3, 4, 5, 8) + (1, 3, 5, 4, 8, 6, 2) - (1, 
    2)(3, 5)(4, 6)(7, 8) + (1, 6, 7, 2, 8, 3, 4) + (1, 7, 2, 4)(3, 5, 8, 6) + 
    (1, 3, 8, 7, 2, 5, 4) + (1, 6, 3)(2, 4, 8) + (1, 7, 3)(2, 5, 8) + (1, 3)(2, 
    8)(4, 5)(6, 7) - (1, 2)(3, 8)(4, 7)(5, 6) - (1, 6, 4, 7)(2, 5, 3, 8) - (1, 
    7)(2, 8)(3, 4)(5, 6) - (1, 3, 5, 7)(2, 4, 6, 8) - (1, 6)(2, 5)(3, 4)(7, 8) -
(1, 7, 4, 6)(2, 8, 3, 5) - (1, 3, 8, 6)(2, 4, 7, 5) - (1, 6, 8, 3)(2, 5, 7, 4) -
(1, 7, 5, 3)(2, 8, 6, 4) - (1, 3)(2, 4)(5, 6)(7, 8) + (1, 2, 5)(3, 4, 8)
\end{lstlisting}

\begin{lstlisting}[language=C, frame=single, basicstyle=\ttfamily\tiny, escapechar=\@, caption=, title={Lemma22-3\_rr4.txt}]
-1/14*(1, 5, 4)(2, 6, 3)
\end{lstlisting}

\begin{lstlisting}[language=C, frame=single, basicstyle=\ttfamily\tiny, escapechar=\@, caption=, title={Lemma22-3\_r5.txt}]
(1, 2, 8)(4, 6, 5) - (1, 8, 2, 6)(3, 5, 4, 7) + (1, 5, 8, 2, 3, 4, 6) - (1, 
    3)(2, 6)(4, 7)(5, 8) - (1, 6, 2, 8)(3, 7, 4, 5) + (1, 7, 5, 6, 2, 4, 8) + 
    (2, 4, 6)(3, 8, 5) - (1, 4, 5, 6)(2, 8, 7, 3) + (1, 7, 2, 3, 8, 5, 6) + (2, 
    8, 7)(3, 6, 5) + (1, 5, 8, 7)(2, 4, 3, 6) + (2, 6, 7, 5, 3, 4, 8) - (1, 
    8)(2, 4)(3, 5)(6, 7) - (1, 3)(2, 8)(4, 5)(6, 7) + (1, 5, 8)(2, 6, 7) - (1, 
    6, 5, 4)(2, 3, 7, 8) - (1, 4, 7, 8)(2, 6, 5, 3) + (1, 7)(2, 8)(3, 4)(5, 6) -
(1, 8, 7, 4)(2, 3, 5, 6) + (1, 2, 3, 6)(4, 8, 7, 5) - (1, 3)(2, 4)(5, 6)(7, 8) +
(1, 2, 6, 7)(3, 8, 5, 4) - (1, 6)(2, 4)(3, 7)(5, 8) - (1, 4)(2, 3)(5, 8)(6, 7)
\end{lstlisting}

\begin{lstlisting}[language=C, frame=single, basicstyle=\ttfamily\tiny, escapechar=\@, caption=, title={Lemma22-3\_rr5.txt}]
-1/14*(1, 4)(2, 3)(5, 8)(6, 7)
\end{lstlisting}

\begin{lstlisting}[language=C, frame=single, basicstyle=\ttfamily\tiny, escapechar=\@, caption=, title={Lemma22-3\_r6.txt}]
Id(G) - (1, 6, 3, 4, 7, 5, 8) - (3, 7, 6)(4, 8, 5) + (1, 6, 5)(4, 8, 7) - (1, 6,
    5, 8, 3, 2, 7) + (1, 5, 6)(4, 7, 8) - (1, 5, 7, 6)(2, 4, 8, 3) + (1, 6)(2, 
    7)(3, 5)(4, 8) + (2, 8, 7)(3, 6, 5) - (1, 5, 7)(3, 8, 6) - (1, 5, 7, 3, 4, 
    2, 8) + (1, 6, 3)(2, 4, 8) - (2, 8, 3)(4, 7, 5) + (2, 7, 8)(3, 5, 6) - (1, 
    3, 4, 8)(2, 7, 6, 5) - (1, 3, 8, 6)(2, 4, 7, 5) + (1, 5)(2, 8)(3, 6)(4, 7) -
(1, 3, 7)(2, 8, 5) + (1, 3)(2, 4)(5, 6)(7, 8) + (1, 5, 3)(2, 4, 7) - (1, 6)(2, 
    4)(3, 7)(5, 8) + (1, 3, 6)(2, 8, 4) - (2, 7, 3, 8, 6, 5, 4) + (1, 3, 5)(2, 
    7, 4)
\end{lstlisting}

\begin{lstlisting}[language=C, frame=single, basicstyle=\ttfamily\tiny, escapechar=\@, caption=, title={Lemma22-3\_rr6.txt}]
1/14*(1, 3, 5)(2, 7, 4)
\end{lstlisting}

\begin{lstlisting}[language=C, frame=single, basicstyle=\ttfamily\tiny, escapechar=\@, caption=, title={Lemma22-5\_q1.txt}]
3/7*(1, 7, 4, 8, 5, 3, 2) + 1/7*(1, 6)(2, 4)(3, 7)(5, 8) + 1/7*(1, 8, 4)(2, 7, 
    3) + 1/7*(1, 2, 7, 3)(4, 8, 5, 6) - 1/7*(1, 4)(2, 3)(5, 8)(6, 7) + 2/7*(1, 
    6, 7, 5)(2, 3, 8, 4) - 1/7*(1, 7, 8, 5)(2, 6, 3, 4) + 1/7*(1, 3, 5)(2, 7, 4)
\end{lstlisting}

\begin{lstlisting}[language=C, frame=single, basicstyle=\ttfamily\tiny, escapechar=\@, caption=, title={Lemma22-5\_q2.txt}]
(1, 4, 6, 3)(2, 8, 5, 7) - 3/4*(1, 3, 6)(2, 8, 4) - 3/4*(2, 7, 3, 8, 6, 5, 4) - 
3/4*(1, 2, 7, 3)(4, 8, 5, 6) - 3/4*(1, 4)(2, 3)(5, 8)(6, 7) + (1, 6, 7, 5)(2, 3,
    8, 4) + (1, 7, 8, 5)(2, 6, 3, 4) + (1, 3, 5)(2, 7, 4)
\end{lstlisting}

\begin{lstlisting}[language=C, frame=single, basicstyle=\ttfamily\tiny, escapechar=\@, caption=, title={Lemma22-6\_q3.txt}]
(1, 6)(2, 4)(3, 7)(5, 8) - (1, 3, 6)(2, 8, 4) - (2, 7, 3, 8, 6, 5, 4) + (1, 8, 
    4)(2, 7, 3) - (1, 4)(2, 3)(5, 8)(6, 7) + (1, 5, 4)(2, 6, 3) + (1, 3, 5)(2, 
    7, 4)
\end{lstlisting}

\begin{lstlisting}[language=C, frame=single, basicstyle=\ttfamily\tiny, escapechar=\@, caption=, title={Lemma22-6\_q4.txt}]
1/4*(1, 5, 3, 4)(2, 7, 8, 6) + 1/4*(1, 3, 6)(2, 8, 4) + 1/4*(1, 2, 7, 3)(4, 8, 
    5, 6) + 1/4*(1, 3, 5)(2, 7, 4)
\end{lstlisting}

%% file: ZZpG_Output_3-2.tex

\begin{lstlisting}[language=C, frame=single, basicstyle=\ttfamily\tiny, escapechar=\@, caption=, title={Lemma33-2\_r0.txt}]
104/63*[a^3 a^2]
[a^5   0] - 1459/378*[a^6 a^4]
[a^2 a^3] - 311/54*[  1   0]
[a^3   1] - 1607/378*[  1   1]
[a^3   a] + 164/21*[  1   a]
[a^3 a^5] + 1151/378*[  1 a^2]
[a^3 a^4] + 809/189*[  0   1]
[  1   a] - 118/63*[a^3   1]
[a^5   a] - 25/126*[  1 a^3]
[a^3 a^2] + 1273/378*[  1 a^6]
[a^3 a^6] - 203/54*[a^3   a]
[a^2 a^5] + 32/189*[a^4   0]
[  1 a^3] - 1693/378*[a^4 a^4]
[  1   a] + 52/63*[a^4 a^5]
[  1   1] + 571/378*[a^4 a^6]
[  1 a^5] + 131/27*[  0 a^4]
[a^3   a] - 677/189*[  1 a^4]
[a^6   a] + 182/27*[a^4   1]
[  1   0] + 3331/378*[a^4 a^3]
[  1 a^4] + 16/21*[  1 a^5]
[  0   1] + 1133/378*[  a   0]
[a^4 a^6] + 619/378*[  a   a]
[a^4 a^3] + 62/21*[  a a^2]
[a^4   a] + 1559/378*[  a a^3]
[a^4   0] + 761/189*[  0   a]
[a^6 a^3] + 2/9*[a^4   a]
[  0 a^3] + 3797/378*[  a a^4]
[a^4 a^2] + 362/63*[  a   1]
[a^4 a^4] + 1003/378*[a^4 a^2]
[a^2   a] - 628/189*[a^5   0]
[  a a^2] - 1985/378*[a^5 a^5]
[  a a^4] - 401/63*[a^5 a^6]
[  a   0] - 89/18*[a^5   1]
[  a a^5] - 125/63*[  0 a^5]
[a^2 a^4] - 179/189*[  a a^5]
[a^6 a^4] - 1717/378*[a^5   a]
[  a   a] - 1487/378*[a^5 a^4]
[  a a^6] - 18/7*[  a a^6]
[  a   0] + 2701/378*[a^2   0]
[a^5 a^5] + 83/21*[a^2 a^2]
[a^5   0] + 2068/189*[a^2 a^3]
[a^5   a] + 2693/378*[a^2 a^4]
[a^5 a^4] + 1439/189*[  0 a^2]
[a^5   0] - 38/21*[a^5 a^2]
[a^5   0] + 29/3*[a^2 a^5]
[a^5 a^6] + 1219/189*[a^2   a]
[a^5   1] + 236/63*[a^5 a^3]
[a^6   a] - 565/126*[a^6   0]
[a^2   a] - 2530/189*[a^6 a^6]
[a^2 a^4] - 169/21*[a^6   1]
[a^2   1] - 274/27*[a^6   a]
[a^2 a^2] + 145/21*[  0 a^6]
[  a a^4] + 1258/189*[a^2 a^6]
[a^3 a^4] - 1979/378*[a^6 a^2]
[a^2 a^6] - 2764/189*[a^6 a^5]
[a^2   0] + 139/126*[a^2   1]
[a^6   1] + 914/189*[a^3   0]
[a^6 a^4] - 205/21*[a^3 a^3]
[a^6 a^3] - 29/7*[a^3 a^4]
[a^6 a^5] - 1307/378*[a^3 a^5]
[a^6 a^2] + 1093/189*[  0 a^3]
[a^4 a^3] - 733/189*[a^6 a^3]
[a^3 a^3] - 617/189*[a^3 a^6]
[a^6   a] - 515/126*[a^3 a^2]
[a^6   1] - 2/9*[a^6 a^4]
[  a a^5] + 1187/378*[  1   0]
[a^6   1] + 6281/378*[  1   1]
[a^6 a^2] - 1118/189*[  1   a]
[a^6   0] - 370/189*[  1 a^2]
[a^6 a^3] - 4343/378*[  0   1]
[  1 a^2] - 485/378*[  1 a^3]
[a^6 a^6] - 65/27*[  1 a^6]
[a^6 a^4] + 995/189*[a^3   a]
[a^6   0] + 95/63*[a^4   0]
[a^3 a^3] + 1612/189*[a^4 a^4]
[a^3   0] - 127/126*[a^4 a^5]
[a^3 a^6] - 569/189*[a^4 a^6]
[a^3 a^2] - 509/54*[  0 a^4]
[a^3   0] + 449/126*[  1 a^4]
[a^3   0] - 817/189*[a^4   1]
[a^3 a^4] - 1027/378*[a^4 a^3]
[a^3 a^5] + 425/126*[  1 a^5]
[a^4 a^6] + 59/378*[  a   0]
[  1 a^6] + 3061/189*[  a   a]
[  1 a^2] + 128/63*[  a a^2]
[  1 a^5] - 257/54*[  a a^3]
[  1   1] - 1915/378*[  0   a]
[a^6 a^2] + 23/18*[a^4   a]
[  a a^2] - 13/63*[  a a^4]
[  1 a^4] - 2465/378*[  a   1]
[  1   0] + 236/63*[a^4 a^2]
[a^4 a^5] - 3223/378*[a^5   0]
[a^4 a^2] + 2503/378*[a^5 a^5]
[a^4   a] + 67/27*[a^5 a^6]
[a^4 a^3] + 1781/378*[a^5   1]
[a^4   1] - 3391/378*[  0 a^5]
[a^2   a] + 8/9*[  a a^5]
[  a   a] - 115/27*[a^5   a]
[a^4 a^6] + 95/189*[a^5 a^4]
[a^4 a^5] + 2825/378*[  a a^6]
[a^6 a^3] + 1289/378*[a^2   0]
[  a a^5] + 1646/189*[a^2 a^2]
[  a a^6] + 34/27*[a^2 a^3]
[  a a^3] - 2309/378*[a^2 a^4]
[  a a^2] - 2875/378*[  0 a^2]
[a^5 a^6] + 359/189*[a^5 a^2]
[a^3 a^6] - 577/189*[a^2 a^5]
[  a   1] - 1139/189*[a^2   a]
[  a a^4] - 8/21*[a^5 a^3]
[  1 a^3] - 17/63*[a^6   0]
[a^5   a] - 1/126*[a^6 a^6]
[a^5 a^6] + 20/27*[a^6   1]
[a^5 a^5] + 47/189*[a^6   a]
[a^5 a^3] - 4055/378*[  0 a^6]
[  a a^6] - 151/378*[a^6 a^2]
[a^5   0] - 59/189*[a^6 a^5]
[a^5 a^2] + 250/63*[a^2   1]
[  0 a^5] + 227/126*[a^3   0]
[a^2 a^4] + 517/189*[a^3 a^3]
[a^2   a] - 250/63*[a^3 a^4]
[a^2 a^6] + 160/189*[a^3 a^5]
[a^2   0] - 2143/378*[  0 a^3]
[a^4   a] + 62/21*[a^6 a^3]
[  0   a] - 263/378*[a^3 a^6]
[a^2   1] + 383/126*[a^3 a^2]
[a^2 a^2] - 100/63*[a^6 a^4]
[  1 a^6] + 415/189*[  1   0]
[a^4   1] - 2131/189*[  1   1]
[a^4 a^5] + 643/189*[  1   a]
[a^4 a^4] - 185/126*[  1 a^2]
[a^4 a^2] + 109/21*[  0   1]
[  1 a^5] + 2/9*[a^3   1]
[a^3 a^5] - 191/189*[  1 a^3]
[a^4   0] + 562/189*[  1 a^6]
[a^4   a] - 370/189*[a^3   a]
[  0 a^4] - 346/189*[a^4   0]
[  a a^3] - 235/189*[a^4 a^4]
[  a   1] + 179/378*[a^4 a^5]
[  a a^5] + 548/189*[a^4 a^6]
[  a   0] + 779/189*[  0 a^4]
[a^3   1] - 34/63*[  1 a^4]
[  0   1] + 200/63*[a^4   1]
[  a a^6] - 898/189*[a^4 a^3]
[  a   a] - 29/6*[  1 a^5]
[a^6 a^5] + 277/63*[  a   0]
[a^5 a^6] - 1793/378*[  a   a]
[a^5   a] - 112/27*[  a a^2]
[a^5   0] + 134/189*[  a a^3]
[a^5 a^2] - 55/63*[  0   a]
[a^6   a] + 151/189*[a^4   a]
[a^3   a] + 1/378*[  a a^4]
[a^5 a^5] + 715/189*[  a   1]
[a^5 a^3] - 2/9*[a^4 a^2]
[a^5   0] + 401/378*[a^5   0]
[a^2 a^2] + 211/378*[a^5 a^5]
[a^2   0] + 614/189*[a^5 a^6]
[a^2 a^5] - 586/189*[a^5   1]
[a^2   a] + 1307/189*[  0 a^5]
[a^2   0] - 2/9*[  a a^5]
[a^2   0] - 571/189*[a^5   a]
[a^2 a^3] + 82/27*[a^5 a^4]
[a^2 a^4] - 38/21*[  a a^6]
[a^3 a^5] + 1465/378*[a^2   0]
[a^6 a^5] - 766/189*[a^2 a^2]
[a^6   a] - 899/378*[a^2 a^3]
[a^6 a^4] + 641/189*[a^2 a^4]
[a^6 a^6] + 223/189*[  0 a^2]
[a^5   a] + 8/21*[a^5 a^2]
[  1   a] - 401/378*[a^2 a^5]
[a^6 a^3] + 2003/378*[a^2   a]
[a^6   0] - 87/14*[a^5 a^3]
[a^3 a^4] + 289/126*[a^6   0]
[a^3   a] - 143/54*[a^6 a^6]
[a^3   1] - 1/18*[a^6   1]
[a^3 a^2] + 1693/378*[a^6   a]
[a^3 a^6] + 11/63*[  0 a^6]
[  a   1] + 2375/378*[a^2 a^6]
[  1   1] + 278/189*[a^6 a^2]
[a^3 a^5] + 394/189*[a^6 a^5]
[a^3 a^4] + 214/21*[a^2   1]
[a^5 a^2] - 2479/378*[a^3   0]
[  1 a^4] + 557/378*[a^3 a^3]
[  1 a^5] - 311/378*[a^3 a^4]
[  1 a^2] + 103/189*[a^3 a^5]
[  1   a] + 281/63*[  0 a^3]
[a^4 a^5] - 1079/126*[a^6 a^3]
[a^2 a^5] + 101/21*[a^3 a^6]
[  1 a^6] - 61/54*[a^3 a^2]
[  1 a^3] - 157/378*[  1   0]
[a^5   1] - 44/63*[  1   1]
[a^5 a^4] - 85/21*[  1   a]
[a^5 a^2] + 209/63*[  1 a^2]
[a^5   0] + 109/54*[  0   1]
[  1 a^4] + 10/27*[a^3   1]
[  0 a^4] + 1093/378*[  1 a^3]
[a^5 a^3] - 172/63*[  1 a^6]
[a^5 a^5] - 4/9*[a^3   a]
[a^3 a^2] + 25/42*[a^4   0]
[a^2 a^3] - 173/378*[a^4 a^4]
[a^2 a^5] + 23/42*[a^4 a^5]
[a^2   0] + 283/189*[a^4 a^6]
[a^2 a^6] + 19/42*[  0 a^4]
[a^3 a^5] - 2/9*[  1 a^4]
[  1 a^5] - 1003/189*[a^4   1]
[a^2 a^2] - 253/189*[a^4 a^3]
[a^2   1] + 236/63*[  1 a^5]
[a^2   0] - 59/14*[  a   0]
[a^6 a^6] - 2705/378*[  a   a]
[a^6   0] - 19/21*[  a a^2]
[a^6 a^2] + 1177/378*[  a a^3]
[a^6 a^5] + 13/126*[  0   a]
[a^6   0] - 106/63*[a^4   a]
[a^6   0] - 236/189*[  a a^4]
[a^6   1] + 734/189*[  a   1]
[a^6   a] - 169/378*[a^4 a^2]
[  1 a^2] + 619/126*[a^5   0]
[a^3 a^2] - 157/63*[a^5 a^5]
[a^3 a^5] - 122/189*[a^5 a^6]
[a^3   a] + 58/21*[a^5   1]
[a^3 a^3] + 33/14*[  0 a^5]
[a^2 a^5] + 1438/189*[  a a^5]
[a^4 a^5] + 1325/378*[a^5   a]
[a^3   1] - 1409/378*[a^5 a^4]
[a^3   0] - 397/126*[  a a^6]
[  1   a] - 268/63*[a^2   0]
[  1 a^5] - 652/189*[a^2 a^2]
[  1 a^4] - 425/126*[a^2 a^3]
[  1 a^6] + 205/126*[a^2 a^4]
[  1 a^3] + 967/378*[  0 a^2]
[a^5 a^4] - 45/14*[a^5 a^2]
[a^4 a^4] - 20/63*[a^2 a^5]
[  1 a^2] + 137/42*[a^2   a]
[  1   a] - 379/189*[a^5 a^3]
[a^2 a^6] + 85/189*[a^6   0]
[a^4   a] + 886/189*[a^6 a^6]
[a^4 a^2] + 407/378*[a^6   1]
[a^4 a^6] - 46/9*[a^6   a]
[a^4 a^5] + 185/54*[  0 a^6]
[  a a^2] - 2419/378*[a^2 a^6]
[a^6 a^2] - 569/189*[a^6 a^2]
[a^4 a^3] + 91/54*[a^6 a^5]
[a^4   1] - 1576/189*[a^2   1]
[a^3 a^6] - 71/18*[a^3   0]
[  a a^4] + 1303/378*[a^3 a^3]
[  a a^2] + 494/63*[a^3 a^4]
[  a   a] - 785/378*[a^3 a^5]
[  a a^6] - 259/54*[  0 a^3]
[a^4 a^2] - 2/9*[a^6 a^3]
[  1 a^2] - 23/378*[a^3 a^6]
[  a   0] - 92/63*[a^3 a^2]
[  a a^5] - 278/63*[a^6 a^4]
[  0   a]
\end{lstlisting}

\begin{lstlisting}[language=C, frame=single, basicstyle=\ttfamily\tiny, escapechar=\@, caption=, title={Lemma33-3\_r1.txt}]
[  1   0]
[  0   1] + [  1 a^2]
[  0   1] - [  1 a^3]
[  0   1] - [  0 a^4]
[a^3 a^3] + [a^4   a]
[  1 a^6] + [a^4 a^2]
[a^6 a^6] + [a^5   0]
[  0 a^2] + [a^5 a^4]
[  0 a^2] - [a^2 a^2]
[  0 a^5] + [a^2 a^4]
[  0 a^5] + [a^2 a^5]
[  0 a^5] - [a^6   0]
[  0   a] + [a^2   1]
[  a   a] - [a^3   0]
[  0 a^4] - [a^6 a^3]
[a^5 a^4] - [a^3 a^6]
[  0 a^4] + [a^3 a^2]
[  0 a^4] - [  0   1]
[  1   0] + [a^3   1]
[  1   0] + [a^3   a]
[  a a^3] + [a^4 a^6]
[a^4 a^4] + [  0 a^4]
[a^3 a^6] + [a^4 a^3]
[a^4   0] + [  a   a]
[  a a^5] + [  a a^3]
[  a a^4] + [  0   a]
[a^6 a^5] - [a^4 a^2]
[a^3   1] + [a^5 a^5]
[a^5 a^3] - [a^5   1]
[a^5 a^6] - [a^5   a]
[a^5 a^4] - [a^2 a^2]
[a^2 a^3] + [a^6   0]
[a^6   a] - [a^2 a^6]
[  0 a^5] + [a^6 a^2]
[a^6 a^4] - [a^6 a^5]
[a^6 a^6] - [a^3   0]
[a^3 a^4] - [a^3 a^5]
[a^3   1] + [a^3 a^6]
[a^3 a^3] + [  1 a^3]
[  a a^5] + [a^4   0]
[a^5 a^3] - [a^4 a^5]
[a^5 a^4] + [  1 a^4]
[a^2 a^2] - [a^4 a^3]
[a^5 a^6] - [  a a^2]
[a^2 a^4] - [  a a^3]
[a^2 a^3] - [a^5 a^5]
[a^6   1] - [a^5 a^6]
[a^6 a^6] + [a^5   1]
[a^6 a^4] + [  0 a^5]
[a^2   1] + [  a a^5]
[a^5   1] + [  a a^6]
[  0 a^6] + [a^2 a^2]
[a^3 a^2] - [a^2 a^3]
[a^3   1] + [a^2   a]
[a^3 a^3] + [  0 a^6]
[  a a^3] - [a^2 a^6]
[a^5 a^3] - [a^2   1]
[  1   0] + [a^3   0]
[a^4 a^4] + [a^3 a^3]
[a^4   0] - [  0 a^3]
[a^4   0] - [a^6 a^3]
[a^4   0] + [  1 a^3]
[a^2 a^4] + [a^4   0]
[a^6 a^3] - [a^4 a^5]
[a^6   a] + [  1 a^4]
[  a a^4] - [a^4 a^3]
[a^6 a^2] - [  a a^2]
[a^3 a^3] - [  a a^3]
[a^3   a] - [a^5 a^5]
[  1 a^6] - [a^5 a^6]
[  1 a^4] + [a^5   1]
[  1   0] + [  0 a^5]
[a^2 a^6] + [  a a^5]
[  0 a^6] + [  a a^6]
[a^5 a^4] + [a^2 a^2]
[a^4   1] - [a^2 a^3]
[a^4   0] + [a^2   a]
[a^4 a^2] + [  0 a^6]
[  a   0] - [a^2 a^6]
[  a   0] - [a^2   1]
[a^2 a^4] + [a^3   0]
[a^5 a^4] + [a^3 a^3]
[a^5   1] - [  0 a^3]
[a^4   1] - [a^6 a^3]
[a^6   1] - [  1   0]
[a^3   1] + [  1   1]
[a^3   a] - [  1   a]
[a^3 a^5] - [a^3   a]
[a^2 a^5] - [a^4 a^3]
[  1 a^4] - [  1 a^5]
[  0   1] + [  a   0]
[a^4 a^6] - [  a   a]
[a^4 a^3] - [  a a^4]
[a^4 a^2] - [a^5   0]
[  a a^2] - [a^5 a^6]
[  a   0] + [a^5   a]
[  a   a] - [a^2 a^4]
[a^5 a^4] + [a^5 a^2]
[a^5   0] + [a^5 a^3]
[a^6   a] - [a^6   a]
[a^2 a^2] - [  0 a^6]
[  a a^4] - [a^2 a^6]
[a^3 a^4] + [a^6 a^5]
[a^2   0] + [a^3 a^5]
[a^6 a^2] + [  0 a^3]
[a^4 a^3] - [a^6 a^3]
[a^3 a^3] - [a^3 a^2]
[a^6   1] - [  1   0]
[a^6   1] + [  1   1]
[a^6 a^2] - [  1   a]
[a^6   0] - [a^3   a]
[a^6   0] - [a^4 a^3]
[a^3 a^5] - [  1 a^5]
[a^4 a^6] + [  a   0]
[  1 a^6] - [  a   a]
[  1 a^2] - [  a a^4]
[  1 a^4] - [a^5   0]
[a^4 a^2] - [a^5 a^6]
[a^4 a^3] + [a^5   a]
[a^4 a^6] - [a^2 a^4]
[  a a^2] + [a^5 a^2]
[a^3 a^6] + [a^5 a^3]
[  1 a^3] - [a^6   a]
[a^5 a^3] - [  0 a^6]
[  a a^6] - [a^2 a^6]
[a^4 a^6] + [a^6 a^5]
[a^5 a^2] + [a^3 a^5]
[a^2   0] + [  0 a^3]
[a^4   a] - [a^6 a^3]
[  0   a] - [a^3 a^2]
[a^2 a^2] + [  1   0]
[a^4   1] + [  1 a^2]
[a^4 a^2] - [  1 a^3]
[a^4   0] - [  0 a^4]
[a^3   1] + [a^4   a]
[a^3   a] + [a^4 a^2]
[a^5   0] + [a^5   0]
[a^2 a^2] + [a^5 a^4]
[a^2 a^4] - [a^2 a^2]
[a^6   a] + [a^2 a^4]
[a^6 a^6] + [a^2 a^5]
[a^6 a^3] - [a^6   0]
[a^3   a] + [a^2   1]
[a^5 a^2] - [a^3   0]
[  1 a^4] - [a^6 a^3]
[a^2 a^5] - [a^3 a^6]
[  1 a^6] + [a^3 a^2]
[  1 a^3] - [  0   1]
[  1 a^4] + [a^3   1]
[  0 a^4] + [a^3   a]
[a^3 a^2] + [a^4 a^6]
[a^2 a^6] + [  0 a^4]
[a^3 a^5] + [a^4 a^3]
[a^2   1] + [  a   a]
[a^6   0] + [  a a^3]
[a^6 a^5] + [  0   a]
[a^6   0] - [a^4 a^2]
[  1 a^2] + [a^5 a^5]
[a^3 a^5] - [a^5   1]
[a^3 a^3] - [a^5   a]
[a^3   1] - [a^2 a^2]
[  1 a^4] + [a^6   0]
[a^4   a] - [a^2 a^6]
[a^6 a^2] + [a^6 a^2]
[a^4 a^3] - [a^6 a^5]
[a^4   1] - [a^3   0]
[  a a^4] - [a^3 a^5]
[  a a^6] + [a^3 a^6]
[  a   0]
\end{lstlisting}

\begin{lstlisting}[language=C, frame=single, basicstyle=\ttfamily\tiny, escapechar=\@, caption=, title={Lemma33-3\_rr1.txt}]
1/63*[a^3 a^6]
[  a   0]
\end{lstlisting}

%% file: ZZpG_Output_4-1.tex

\begin{lstlisting}[language=C, frame=single, basicstyle=\ttfamily\tiny, escapechar=\@, caption=, title={Lemma41-2\_r0.txt}]
12157288336960917778013/566319477556846273851*(1, 3, 2, 5, 6) - 
62847049404830244701429/4530555820454770190808*(1, 4, 5, 3)(2, 6) + 
73560691141071857899561/2265277910227385095404*(1, 6, 2, 4, 3) - 
1158737933233857340588/62924386395205141539*(1, 5, 4, 6)(2, 3) - 
4018630252209294546863/566319477556846273851*(1, 2, 3)(4, 5, 6) + 
103357175633014550372183/4530555820454770190808*(1, 3)(2, 6) - 
1576688218684949929411/377546318371230849234*(1, 4, 5, 6)(2, 3) - 
21373515117675617526433/2265277910227385095404*(1, 6)(2, 5, 4, 3) - 
491388485952495217919/1132638955113692547702*(1, 5, 3)(2, 4, 6) + 
3444648854282265351496/566319477556846273851*(3, 4)(5, 6) + 
4074033363557412118601/377546318371230849234*(1, 2)(3, 5) - 
2468656833362642633720/566319477556846273851*(1, 3, 4, 2)(5, 6) + 
1854654314597713826422/566319477556846273851*(1, 4)(3, 5) + 
13104637705219600757741/566319477556846273851*(1, 6, 4) - 
3523295748648910188688/566319477556846273851*(1, 5, 4, 2)(3, 6) + 
6471091454077057633045/377546318371230849234*(2, 3)(5, 6) + 
3477927265494562839485/251697545580820566156*(1, 2, 4, 5, 3) + 
25899293122256344538953/1132638955113692547702*(1, 3)(5, 6) + 
10048259769007176445951/1132638955113692547702*(1, 4, 5, 3, 2) + 
3854011266435214759375/377546318371230849234*(1, 6, 4, 3, 2) + 
71081998090336703542775/2265277910227385095404*(1, 5, 4, 6, 3) + 
2645233809448923093439/1132638955113692547702*(2, 4)(5, 6) + 
43752697697489395760219/2265277910227385095404*(1, 2, 5, 3, 4) - 
3015769579864282045067/2265277910227385095404*(1, 3, 2, 4)(5, 6) + 
6846714278667937004017/1132638955113692547702*(1, 4, 2, 5, 3) + 
40808246493227302145429/2265277910227385095404*(1, 6, 4, 2, 3) - 
1053285458497246898351/755092636742461698468*(1, 5, 4)(2, 6, 3) + 
24328330311747038736779/1132638955113692547702*(2, 5, 6, 4, 3) - 
61258141371076306500455/2265277910227385095404*(1, 2, 6, 5, 3) + 
9689340980299781954788/566319477556846273851*(1, 3)(2, 5, 6, 4) + 
31710814380001155277/4030743612504243942*(1, 4)(2, 6, 5, 3) + 
57226583969471563350337/2265277910227385095404*(1, 6, 3, 2, 4) + 
16489437491331773243071/755092636742461698468*(1, 5, 4, 2, 3) + 
17756416224356134385845/1132638955113692547702*(2, 6, 4) - 
9635516464900441420547/566319477556846273851*(1, 2, 3, 4)(5, 6) + 
6686318937631142430338/566319477556846273851*(1, 3, 2, 6, 4) + 
5586512071561497376583/566319477556846273851*(1, 4, 2, 3)(5, 6) - 
739772084310045457265/125848772790410283078*(1, 6, 3)(2, 5, 4) + 
62368097574938445619633/2265277910227385095404*(1, 5, 3, 2, 4) + 
5791191226899054227597/251697545580820566156*(3, 5, 6) - 
5966762490262052590615/2265277910227385095404*(1, 2)(3, 6, 4, 5) + 
35221693647818074549/4030743612504243942*(1, 3, 5, 6, 2) + 
23839509964422019569067/2265277910227385095404*(1, 4, 5, 3, 6) + 
26820716275153797796441/2265277910227385095404*(1, 6)(3, 4) + 
17524179103510094755667/1510185273484923396936*(1, 5, 4, 6, 2) + 
49419387495958114402001/2265277910227385095404*(2, 3, 5, 6, 4) + 
284148429233423034459/41949590930136761026*(1, 2, 4)(3, 6, 5) + 
604546197310084619017/2265277910227385095404*(1, 3, 5, 6, 4) + 
5825720207852441666152/566319477556846273851*(1, 4, 2)(3, 6, 5) + 
16123037134743099034169/2265277910227385095404*(1, 6, 3, 4, 2) + 
16433257877013430177487/2265277910227385095404*(1, 5, 4) + 
29379075593475522541423/4530555820454770190808*(2, 4, 3, 5, 6) - 
3773336767033636254833/1132638955113692547702*(1, 2, 5, 3, 6) - 
18849044250353985765727/2265277910227385095404*(1, 3, 5, 6)(2, 4) - 
2258064990038041815313/566319477556846273851*(1, 4)(2, 5, 3, 6) + 
19377487769590743768767/4530555820454770190808*(1, 6, 2, 3, 4) + 
34305915116767448117/251697545580820566156*(1, 5, 4, 2, 6) - 
1724589646482027199837/566319477556846273851*(2, 5, 6, 3, 4) - 
10925460327505687737005/566319477556846273851*(1, 2, 6, 4)(3, 5) + 
72679293408451878918107/4530555820454770190808*(1, 3, 4)(2, 5, 6) + 
82272538133438987985173/2265277910227385095404*(1, 4, 2, 6)(3, 5) - 
71307765790228849868101/2265277910227385095404*(1, 6)(2, 4) + 
21439707303871955507551/2265277910227385095404*(1, 5, 4)(2, 3, 6) + 
64394273374761248078509/2265277910227385095404*(2, 6)(3, 4) + 
7405899965899037268923/566319477556846273851*(1, 2, 3, 5, 6) - 
21490605218955987107951/1132638955113692547702*(1, 3, 4, 2, 6) - 
19191257992893111950159/755092636742461698468*(1, 4)(2, 3, 5, 6) + 
2273861038473152311613/62924386395205141539*(1, 6, 2, 5, 4) + 
33142093919577320975605/2265277910227385095404*(1, 5, 3, 6)(2, 4) - 
5731183994363445578129/566319477556846273851*(3, 6, 4) - 
28416648730748972480249/2265277910227385095404*(1, 2)(5, 6) - 
4256182735495884590857/377546318371230849234*(1, 3, 6, 4, 2) - 
35352463596272567054879/2265277910227385095404*(1, 4)(5, 6) - 
17573041311450537175825/2265277910227385095404*(1, 6, 3, 5, 4) + 
23882728131060210249677/2265277910227385095404*(1, 5, 3, 4, 2) - 
2417330504464352273773/1132638955113692547702*(2, 3, 6) + 
17355514859166430059853/2265277910227385095404*(1, 2, 4, 5, 6) - 
7987632287346682412173/1132638955113692547702*(1, 3, 6) + 
42450493788498453297643/4530555820454770190808*(1, 4, 5, 6, 2) + 
11840011089344224008439/755092636742461698468*(1, 6, 2)(3, 5, 4) + 
9614353609189901096483/377546318371230849234*(1, 5, 3, 4, 6) - 
5448018375797511816365/755092636742461698468*(2, 4)(3, 6) - 
1932518885762661449509/1132638955113692547702*(1, 2, 5, 6, 4) - 
4603970163731232187807/566319477556846273851*(1, 3, 6, 2, 4) - 
162647288144491587799/1132638955113692547702*(1, 4, 2, 5, 6) + 
1030943029081608244910/566319477556846273851*(1, 6)(2, 3, 5, 4) - 
464686810457896096916/62924386395205141539*(1, 5, 3, 4)(2, 6) - 
32647060645431356318999/4530555820454770190808*(2, 5, 4, 3, 6) - 
35695751042693747090759/1132638955113692547702*(1, 2, 6) - 
553240533397156259603/62924386395205141539*(1, 3, 6)(2, 5, 4) - 
4180427651668196891041/1132638955113692547702*(1, 4)(2, 6) + 
3532301604994269112687/4530555820454770190808*(1, 6, 2, 4)(3, 5) - 
51560812283928977566027/2265277910227385095404*(1, 5, 6)(2, 3, 4) + 
15212356970266993417325/2265277910227385095404*(2, 6, 3, 5, 4) + 
4911348051301874122631/377546318371230849234*(1, 2, 3, 6, 4) - 
11858306904655394383363/1510185273484923396936*(1, 3, 5, 4)(2, 6) - 
13582176346780223761609/2265277910227385095404*(1, 4, 2, 3, 6) - 
1243869711042727698757/377546318371230849234*(1, 6)(2, 5, 3, 4) - 
58419648149728016175775/2265277910227385095404*(1, 5, 6, 2, 4) + 
1666635792082168270625/188773159185615424617*(4, 6, 5) + 
2357237780690201929894/566319477556846273851*(1, 2)(3, 4) + 
11959328485042032680137/2265277910227385095404*(1, 3, 2)(4, 6, 5) - 
16952189889892281500947/566319477556846273851*(1, 4, 3) - 
6667508542399994503463/377546318371230849234*(1, 6, 4, 5, 3) + 
17764609755870216333593/2265277910227385095404*(1, 5, 6, 3, 2) + 
110930562215920083604/3800801862797625999*(2, 3, 4, 6, 5) - 
3185228674071843239335/566319477556846273851*(1, 2, 4, 3, 5) - 
8458288133123461520480/566319477556846273851*(1, 3, 4, 6, 5) - 
5825720207852441666152/566319477556846273851*(1, 4, 3, 5, 2) - 
15852741407218258081945/755092636742461698468*(1, 6, 4, 5, 2) + 
3164006737073383496429/2265277910227385095404*(1, 5)(3, 6) - 
25777036474253924761753/1132638955113692547702*(2, 4, 6, 5, 3) + 
10246733074559557713337/566319477556846273851*(1, 2, 5, 4, 3) - 
4407385714359900670031/2265277910227385095404*(1, 3)(2, 4, 6, 5) - 
458295930312878568767/62924386395205141539*(1, 4, 3, 2, 5) - 
6548010046734797651452/566319477556846273851*(1, 6, 4, 5)(2, 3) + 
19703256959039550175715/2265277910227385095404*(1, 5, 2, 6, 3) - 
33792488653411809605551/1132638955113692547702*(2, 5)(4, 6) - 
3860909102779596931279/377546318371230849234*(1, 2, 6, 5)(3, 4) + 
17035977554225871298883/2265277910227385095404*(1, 3, 2, 5)(4, 6) - 
7098411824248757570173/1132638955113692547702*(1, 4, 3)(2, 6, 5) + 
973549522878147987634/188773159185615424617*(1, 6, 3)(2, 4, 5) - 
748483331462092793065/377546318371230849234*(1, 5)(2, 3) - 
2542329825101223015335/1132638955113692547702*(2, 6, 4, 5, 3) + 
2618602904680584511699/2265277910227385095404*(1, 2, 3)(4, 6, 5) - 
2307123845203785591235/251697545580820566156*(1, 3)(2, 6, 4, 5) + 
3824711403423803668151/1132638955113692547702*(1, 4, 6, 5)(2, 3) - 
3283149612811303498316/566319477556846273851*(1, 6, 3, 2, 5) - 
345660148166138276747/62924386395205141539*(1, 5, 2, 4, 3) - 
9529918100606148397981/566319477556846273851*(3, 4, 6) + 
574662676828975485271/41949590930136761026*(1, 2)(3, 5, 6, 4) - 
3425647294173905434435/755092636742461698468*(1, 3, 4, 6, 2) - 
5766881998287639591283/1132638955113692547702*(1, 4, 3, 5, 6) + 
618853852835308424171/41949590930136761026*(1, 6)(4, 5) - 
40357566194702509036535/4530555820454770190808*(1, 5, 3, 6, 2) + 
6227556009774948975809/1132638955113692547702*(2, 3)(4, 6) + 
3668063793298720992539/2265277910227385095404*(1, 2, 4, 3)(5, 6) + 
7255187661976561044203/755092636742461698468*(1, 3)(4, 6) + 
2180286350688884726827/566319477556846273851*(1, 4, 3, 2)(5, 6) - 4/9*(1, 5, 3) + 
16592847996977052570335/4530555820454770190808*(2, 4, 6) - 5/9*(1, 2, 5, 6)(3, 4) -
7816696611307061726101/251697545580820566156*(1, 3, 2, 4, 6) - 
57378025547377715526293/2265277910227385095404*(1, 4, 3)(2, 5, 6) - 
97936113804177218508863/4530555820454770190808*(1, 6, 2, 3)(4, 5) + 
5130014500530642606413/377546318371230849234*(1, 5, 3, 2, 6) - 
5908174157764232547755/188773159185615424617*(2, 5, 4, 6, 3) - 
845479037470271421716/566319477556846273851*(1, 2, 6, 4, 3) - 
6347389670847932212921/2265277910227385095404*(1, 3)(2, 5, 4, 6) + 
4677431115977492707807/2265277910227385095404*(1, 4, 3, 2, 6) - 
2758737348588474512444/188773159185615424617*(1, 6)(2, 4, 5, 3) - 
45441905324650871879009/4530555820454770190808*(1, 5, 6, 2, 3) + 
28412925633946021820843/2265277910227385095404*(2, 6)(4, 5) + 
5768884069725696572591/566319477556846273851*(1, 2, 3, 4, 6) - 
225877576201555494305/188773159185615424617*(1, 3, 2, 6)(4, 5) + 
22741845695787631849871/1132638955113692547702*(1, 4, 6, 2, 3) - 
8264795433640820959529/755092636742461698468*(1, 6, 2, 5, 3) + 
17827442252624569194521/1132638955113692547702*(1, 5, 6)(2, 4, 3) - 
22421533478115494872517/755092636742461698468*(3, 5)(4, 6) + 
5165731761459967348835/566319477556846273851*(1, 2)(3, 6, 5, 4) + 
4787183464429709636687/1132638955113692547702*(1, 3, 5, 2)(4, 6) + 
3357634204259246475392/188773159185615424617*(1, 4, 3, 6, 5) + 
2458819408605850430719/251697545580820566156*(1, 6, 3, 4, 5) - 
4814091079721274594842/566319477556846273851*(1, 5, 2) - 
15573727153223581257353/755092636742461698468*(2, 3, 5, 4, 6) + 
3066642169048064756585/2265277910227385095404*(1, 2, 4, 3, 6) - 
39183518275217962323445/2265277910227385095404*(1, 3, 5, 4, 6) + 
91757764431681057833689/4530555820454770190808*(1, 4, 3, 6, 2) - 
3807642510337152209152/188773159185615424617*(1, 6, 2)(3, 4, 5) - 
7864831189761803984819/2265277910227385095404*(1, 5, 6) + 
16265850772635771062483/2265277910227385095404*(2, 4, 6, 3, 5) + 
71819521901159397302831/2265277910227385095404*(1, 2, 5)(3, 6, 4) + 
6793539111311148444257/566319477556846273851*(1, 3, 5)(2, 4, 6) - 
1349848669289417674835/41949590930136761026*(1, 4, 3, 6)(2, 5) + 
11956159508334502611859/2265277910227385095404*(1, 6)(2, 3, 4, 5) - 
951080584555438904914/62924386395205141539*(1, 5)(2, 6) - 
8916414240994597932935/4530555820454770190808*(2, 5, 3, 4, 6) - 
23522165745031188917351/1132638955113692547702*(1, 2, 6)(3, 5, 4) + 
3984584251649756098711/251697545580820566156*(1, 3, 4, 6)(2, 5) + 
8543318191699217561737/4530555820454770190808*(1, 4, 3, 5)(2, 6) + 
8715502872264164602205/4530555820454770190808*(1, 6, 2, 4, 5) - 
9788228827086417626569/566319477556846273851*(1, 5, 2, 3, 6) + 
8727041608811562298/422311318088625111*(2, 6, 3, 4, 5) - 
11416968669098691842591/755092636742461698468*(1, 2, 3, 5)(4, 6) + 
4052491205696936068756/566319477556846273851*(1, 3, 4, 5)(2, 6) + 
1052217834836004194947/566319477556846273851*(1, 4, 6)(2, 3, 5) - 
10156206478653508766371/2265277910227385095404*(1, 6)(2, 5) + 
1158717390714157720655/1132638955113692547702*(1, 5)(2, 4, 3, 6) + 1/9*(3, 6)(4, 5)
+ 30305235588664308074515/1132638955113692547702*(1, 2)(4, 6) - 
4297786483208117378683/566319477556846273851*(1, 3, 6, 2)(4, 5) - 
13020227044381094112301/2265277910227385095404*(1, 4, 6) + 
62142107726468555519683/2265277910227385095404*(1, 6)(3, 5) - 
15043336214899452629195/755092636742461698468*(1, 5, 6, 2)(3, 4) + 
12173958637957732849670/566319477556846273851*(2, 3, 6, 4, 5) + 
54910379232388798846315/2265277910227385095404*(1, 2, 4, 6, 5) + 
10014306752408837032511/1132638955113692547702*(1, 3, 6, 4, 5) + 
494701328261055627539/2265277910227385095404*(1, 4, 6, 5, 2) + 
5571176716399933093852/566319477556846273851*(1, 6, 3, 5, 2) - 
1589144804485493122237/566319477556846273851*(1, 5)(3, 4) - 
10803363272190568687151/1132638955113692547702*(2, 4, 5, 3, 6) + 
847162517306501860185/41949590930136761026*(1, 2, 5, 4, 6) + 
1837246367210332277321/755092636742461698468*(1, 3, 6)(2, 4, 5) + 
58589098099561785584125/2265277910227385095404*(1, 4, 6, 2, 5) - 
6881557473478609842851/1132638955113692547702*(1, 6, 2, 3, 5) + 
38035767143814632446921/2265277910227385095404*(1, 5, 2, 6)(3, 4) + 
2591772427998609988891/2265277910227385095404*(2, 5)(3, 6) - 
48638708116443865835611/2265277910227385095404*(1, 2, 6, 4, 5) + 
54839735788748691370567/4530555820454770190808*(1, 3, 6, 2, 5) + 
27510737369655292434625/755092636742461698468*(1, 4, 5, 2, 6) - 
3343852806885046616609/188773159185615424617*(1, 6)(2, 4, 3, 5) - 
69145602518156081329937/4530555820454770190808*(1, 5)(2, 3, 4, 6) + 
5523660597686065414313/1132638955113692547702*(2, 6)(3, 5) + 
48064935005950997474377/2265277910227385095404*(1, 2, 3, 6)(4, 5) + 
4471158995895707149826/188773159185615424617*(1, 3, 5, 2, 6) + 
18697407515159975950979/755092636742461698468*(1, 4, 5)(2, 3, 6) + 
34191475587691985669033/2265277910227385095404*(1, 6, 2, 5)(3, 4) + 
10037065223482130655001/251697545580820566156*(1, 5, 2, 4, 6)
\end{lstlisting}

\begin{lstlisting}[language=C, frame=single, basicstyle=\ttfamily\tiny, escapechar=\@, caption=, title={Lemma41-3\_r1[1].txt}]
Id(G) + 1/2*(1, 6, 5, 3, 2) - 2*(1, 6, 5, 2, 3) + 5/2*(1, 6, 5)(2, 4, 3) + 2*(2,
    6, 5) - 2*(1, 3, 2, 6, 5) + 5/2*(1, 6, 4, 3)(2, 5) + (1, 2)(3, 6) + 1/2*(1, 
    6, 5, 3, 4) - (1, 6, 5, 2)(3, 4) - (1, 4, 2, 5)(3, 6) + (1, 6, 5)(2, 3, 4) +
(1, 5, 2, 6, 4) + 1/2*(1, 6, 5, 2, 4) + (1, 5)(2, 3, 6, 4) - 2*(2, 6, 5, 3, 4) +
(1, 3, 4)(2, 6, 5) + 1/2*(1, 6, 4, 2, 5) - (1, 5, 2, 4)(3, 6) - (3, 6, 5) + (1, 
    2)(4, 5) - (1, 6, 4, 3, 5) + 1/2*(1, 6, 4, 2)(3, 5) - 9/2*(2, 4, 3, 6, 5) + 
1/2*(1, 3, 6, 5)(2, 4) - (1, 4)(2, 5) + (1, 6, 4)(2, 3, 5) + 5/2*(2, 5, 3, 6, 4)
+ 1/2*(1, 3, 6, 4)(2, 5) + (1, 4, 2, 6, 5) + (2, 6, 4, 3, 5) + 2*(1, 2, 3, 6, 5)
+ (1, 3, 5)(2, 6, 4) + (1, 4)(2, 3, 6, 5) - (1, 5)(2, 4) - (4, 5, 6) - (1, 2)(3,
    4, 6, 5) + 1/2*(1, 3, 2)(4, 5, 6) - 1/2*(1, 4, 6, 2)(3, 5) - 1/2*(1, 6, 2) +
9/2*(1, 5, 4, 3, 6) + 1/2*(1, 2, 5, 3)(4, 6) - 1/2*(1, 3)(2, 4, 5, 6) + (1, 4, 
    6)(2, 5, 3) - (1, 6)(2, 3) - (1, 5, 4, 3)(2, 6) - 1/2*(1, 3, 2, 5, 6) + 
2*(1, 4, 5, 3)(2, 6) + (1, 6, 2, 4, 3) - (1, 5, 4, 6)(2, 3) - (1, 3)(2, 6) + 
2*(1, 4, 5, 6)(2, 3) + (1, 6)(2, 5, 4, 3) + (1, 5, 3)(2, 4, 6) - (3, 4)(5, 6) + 
1/2*(1, 3, 4, 2)(5, 6) + 1/2*(1, 3)(5, 6) + 5/2*(2, 4)(5, 6) + 1/2*(2, 5, 6, 4, 
    3) - 2*(1, 2, 3, 4)(5, 6) - (1, 2)(3, 6, 4, 5) - 5/2*(1, 4, 5, 3, 6) - 
5/2*(1, 6)(3, 4) + 9/2*(1, 5, 4, 6, 2) - 2*(2, 3, 5, 6, 4) + 5/2*(1, 2, 4)(3, 6,
    5) + 1/2*(1, 3, 5, 6, 4) - 1/2*(1, 3, 5, 6)(2, 4) + (1, 4)(2, 5, 3, 6) + 
2*(1, 6, 2, 3, 4) - 2*(1, 5, 4, 2, 6) - 2*(1, 2, 6, 4)(3, 5) - 1/2*(1, 3, 4)(2, 
    5, 6) - 2*(1, 5, 4)(2, 3, 6) + 2*(1, 3, 4, 2, 6) + (1, 5, 3, 6)(2, 4) + (3, 
    6, 4) - (1, 2)(5, 6) - (2, 3, 6) - 1/2*(1, 3, 6) - 5/2*(1, 4, 5, 6, 2) - 
5/2*(1, 6, 2)(3, 5, 4) - 1/2*(1, 5, 3, 4, 6) + 5/2*(1, 2, 5, 6, 4) + (1, 4, 2, 
    5, 6) + 2*(1, 6)(2, 3, 5, 4) - (1, 2, 6) - 2*(1, 4)(2, 6) - 1/2*(1, 6, 2, 
    4)(3, 5) + 2*(1, 3, 5, 4)(2, 6) - 2*(1, 4, 2, 3, 6) - 1/2*(1, 6)(2, 5, 3, 4)
+ (1, 5, 6, 2, 4) + (4, 6, 5) + 1/2*(1, 6, 4, 5, 3) + (2, 3, 4, 6, 5) - (1, 3, 
    4, 6, 5) + 1/2*(2, 4, 6, 5, 3) + 5/2*(1, 3)(2, 4, 6, 5) - 2*(1, 6, 4, 5)(2, 
    3) - 9/2*(2, 5)(4, 6) + (1, 2, 6, 5)(3, 4) + 5/2*(1, 3, 2, 5)(4, 6) - 2*(1, 
    3)(2, 6, 4, 5) - (1, 2)(3, 5, 6, 4) - 5/2*(1, 3, 4, 6, 2) - 1/2*(1, 4, 3, 5,
    6) - 1/2*(1, 6)(4, 5) - 5/2*(1, 5, 3, 6, 2) + 1/2*(1, 2, 4, 3)(5, 6) + (1, 
    3, 2, 4, 6) + (1, 4, 3)(2, 5, 6) - (1, 6, 2, 3)(4, 5) + 2*(1, 5, 3, 2, 6) + 
(1, 3)(2, 5, 4, 6) - (1, 4, 3, 2, 6) - 1/2*(1, 6)(2, 4, 5, 3) + 2*(1, 5, 6, 2, 
    3) - (2, 6)(4, 5) - (1, 3, 2, 6)(4, 5) - (1, 4, 6, 2, 3) - 1/2*(1, 6, 2, 5, 
    3) + (1, 5, 6)(2, 4, 3) - (3, 5)(4, 6) + (1, 2)(3, 6, 5, 4) - (1, 3, 5, 
    2)(4, 6) - 5/2*(1, 3, 5, 4, 6) + 9/2*(1, 4, 3, 6, 2) - 5/2*(1, 5, 6) - 
9/2*(1, 2, 5)(3, 6, 4) - (1, 4, 3, 6)(2, 5) + (1, 2, 3, 5)(4, 6) - (1, 5)(2, 4, 
    3, 6) + (3, 6)(4, 5) + (1, 2)(4, 6) - 1/2*(1, 3, 6, 2)(4, 5) + 9/2*(1, 4, 6)
- 1/2*(1, 5, 6, 2)(3, 4) + 2*(2, 3, 6, 4, 5) - 9/2*(1, 2, 4, 6, 5) - (1, 4, 6, 
    2, 5) + 2*(1, 2, 6, 4, 5) - (1, 2, 3, 6)(4, 5) - (1, 5, 2, 4, 6)
\end{lstlisting}

\begin{lstlisting}[language=C, frame=single, basicstyle=\ttfamily\tiny, escapechar=\@, caption=, title={Lemma41-3\_r2[1].txt}]
Id(G) - (1, 3)(2, 4) + (1, 4)(2, 3) + 1/2*(1, 3)(4, 5) + (1, 5, 6, 4, 3) - 
1/2*(2, 4, 5) - (1, 4, 6, 3)(2, 5) - 1/2*(1, 5)(2, 6, 4, 3) + 1/2*(1, 4, 5)(2, 
    6, 3) + 2*(1, 5, 2, 3)(4, 6) + 1/2*(1, 2, 3, 4, 5) - 2*(1, 5)(2, 4, 6, 3) - 
3/2*(3, 5, 4) + 1/2*(1, 3, 5, 4, 2) + 1/2*(1, 5, 6, 4, 2) - 1/2*(2, 3, 5) + (1, 
    4, 5, 2)(3, 6) - 1/2*(1, 5)(4, 6) + 1/2*(1, 2, 5, 4)(3, 6) + 3/2*(1, 3, 5, 
    2, 4) + (1, 4, 2, 5)(3, 6) + 1/2*(1, 2, 6, 3, 5) - 3/2*(1, 4, 2, 3, 5) - 
1/2*(1, 6, 4, 2, 5) - 1/2*(3, 6, 5) + 3/2*(1, 3, 6, 5, 2) + 1/2*(1, 5, 2)(3, 4, 
    6) - 3/2*(2, 3, 6, 5, 4) - 1/2*(1, 5, 6, 3, 4) + 1/2*(1, 3, 6, 5)(2, 4) - 
1/2*(2, 5, 3, 6, 4) - (1, 3, 6, 4)(2, 5) - 1/2*(2, 6, 4, 3, 5) - 1/2*(1, 4)(2, 
    3, 6, 5) + 1/2*(1, 6, 3, 4)(2, 5) - 1/2*(1, 6, 3) + (1, 4, 6, 2)(3, 5) - (1,
    6, 2) - 1/2*(1, 2, 5, 3)(4, 6) - 1/2*(1, 2, 6)(3, 4, 5) - (1, 4, 5, 3)(2, 6)
- 1/2*(1, 5, 4, 6)(2, 3) - 1/2*(1, 6)(2, 5, 4, 3) + 1/2*(1, 5, 3)(2, 4, 6) + 
3/2*(1, 2)(3, 5) - (1, 3, 4, 2)(5, 6) - 1/2*(1, 4)(3, 5) + (2, 3)(5, 6) - 
1/2*(1, 4, 5, 3, 2) + 1/2*(1, 6, 4, 3, 2) + (1, 6, 4, 2, 3) - (1, 6, 3, 2, 4) + 
(1, 6)(3, 4) - 2*(1, 5, 4, 6, 2) + 3/2*(1, 2, 4)(3, 6, 5) - (1, 2, 5, 3, 6) - 
1/2*(1, 6, 2, 3, 4) - 1/2*(1, 5, 4, 2, 6) - (1, 2, 6, 4)(3, 5) - 1/2*(1, 4, 2, 
    6)(3, 5) + 1/2*(1, 6)(2, 4) + (1, 4)(5, 6) + 2*(1, 5, 3, 4, 6) + 1/2*(1, 5, 
    3, 4)(2, 6) + (2, 5, 4, 3, 6) + (1, 5, 6)(2, 3, 4) + (2, 6, 3, 5, 4) + 
1/2*(1, 3, 5, 4)(2, 6) - (1, 5, 6, 2, 4) - (1, 2)(3, 4) + 2*(1, 6, 4, 5, 3) - 
    (1, 5, 6, 3, 2) + 1/2*(1, 2, 4, 3, 5) + 1/2*(1, 3, 4, 6, 5) - 1/2*(2, 4, 6, 
    5, 3) - 1/2*(1, 6, 4, 5)(2, 3) + 1/2*(1, 5, 2, 6, 3) + (1, 3, 2, 5)(4, 6) + 
1/2*(1, 6, 3)(2, 4, 5) + 1/2*(1, 2, 3)(4, 6, 5) - 1/2*(1, 3)(2, 6, 4, 5) - 
1/2*(1, 6)(4, 5) - (1, 2, 4, 3)(5, 6) - 2*(1, 6, 3, 2)(4, 5) + 1/2*(2, 5, 4, 6, 
    3) - 1/2*(1, 5, 6, 2, 3) + 1/2*(2, 6)(4, 5) + (1, 3, 2, 6)(4, 5) + 1/2*(1, 
    6, 2, 5, 3) + 1/2*(1, 5, 6)(2, 4, 3) + 1/2*(1, 2)(3, 6, 5, 4) - 3/2*(1, 4, 
    3, 6, 5) - (1, 3, 5, 4, 6) + 1/2*(1, 6, 2)(3, 4, 5) + (2, 4, 6, 3, 5) - 
1/2*(1, 4, 3, 6)(2, 5) + 2*(1, 6)(2, 3, 4, 5) - (2, 5, 3, 4, 6) - 1/2*(1, 3, 4, 
    6)(2, 5) + (1, 4, 3, 5)(2, 6) - 2*(1, 6, 2, 4, 5) - (2, 6, 3, 4, 5) - (1, 2,
    3, 5)(4, 6) - (1, 3, 6, 4, 5) - 1/2*(1, 4, 6, 5, 2) + (2, 4, 5, 3, 6) + (1, 
    2, 5, 4, 6) + 1/2*(1, 4, 6, 2, 5) + (1, 2, 6, 4, 5) + 1/2*(1, 3, 6, 2, 5) - 
(1, 2, 3, 6)(4, 5) - (1, 3, 5, 2, 6)
\end{lstlisting}

\begin{lstlisting}[language=C, frame=single, basicstyle=\ttfamily\tiny, escapechar=\@, caption=, title={Lemma41-3\_r3[1].txt}]
Id(G) - (1, 2)(3, 5, 4, 6) + (1, 5, 2)(3, 6, 4) + (2, 3)(4, 5) - (1, 2, 4, 6, 3)
+ 1/2*(1, 5, 6, 4, 3) - 1/2*(1, 2, 5)(3, 4, 6) + (1, 5)(2, 6, 4, 3) + (1, 3)(2, 
    5) - (1, 4, 5)(2, 6, 3) + (2, 6, 5) + (1, 4, 5, 2, 3) + 1/2*(1, 6, 4, 3)(2, 
    5) - 1/2*(1, 2)(3, 6) + (1, 3, 5, 4, 2) - (1, 4)(3, 6) - (1, 6, 5, 2)(3, 4) 
+ 1/2*(1, 5, 2, 6, 4) - (1, 6, 5, 2, 4) + (1, 4, 2, 3, 5) + (1, 3, 6, 5, 2) - 
1/2*(1, 5, 2)(3, 4, 6) + 1/2*(1, 2, 4) + (1, 4, 2) + (1, 6, 4, 2)(3, 5) + 
1/2*(1, 4)(2, 5) - (1, 6, 3, 5)(2, 4) + (1, 5, 2, 3, 4) - (1, 4)(2, 3, 6, 5) - 
    (1, 6, 3, 4)(2, 5) + (1, 3, 2)(4, 5, 6) - (2, 4, 5, 6, 3) - (1, 6, 2, 4, 3) 
- (1, 5, 4, 6)(2, 3) + (1, 3)(2, 6) + 1/2*(1, 6, 4) + (1, 4, 5, 3, 2) + (1, 6, 
    4, 2, 3) + (1, 5, 3, 2, 4) - (1, 6)(3, 4) - (1, 5, 4, 6, 2) + (2, 3, 5, 6, 
    4) + (1, 2, 4)(3, 6, 5) + 1/2*(1, 5, 4) - 1/2*(1, 2, 5, 3, 6) - (1, 3, 5, 
    6)(2, 4) - (1, 6, 2, 3, 4) + (1, 2, 6, 4)(3, 5) - (1, 4, 2, 6)(3, 5) - (1, 
    2, 3, 5, 6) - (1, 3, 4, 2, 6) - (1, 2)(5, 6) + (2, 3, 6) - 1/2*(1, 3, 6) - 
1/2*(1, 5, 3, 4, 6) - (2, 4)(3, 6) + 1/2*(1, 2, 5, 6, 4) + (1, 5, 3, 4)(2, 6) - 
(1, 2, 6) - (1, 4)(2, 6) - (2, 6, 3, 5, 4) + (1, 2, 3, 6, 4) + (1, 2)(3, 4) + 
1/2*(1, 4, 3) - (1, 5, 6, 3, 2) + (2, 3, 4, 6, 5) + (1, 2, 4, 3, 5) + (1, 6, 4, 
    5, 2) - 1/2*(1, 5)(3, 6) + 1/2*(1, 2, 5, 4, 3) + (1, 3)(2, 4, 6, 5) + (1, 6,
    4, 5)(2, 3) - (1, 5, 2, 6, 3) + (2, 5)(4, 6) + (1, 2, 6, 5)(3, 4) - (1, 4, 
    3)(2, 6, 5) - (1, 2, 3)(4, 6, 5) + 1/2*(1, 5, 2, 4, 3) - 1/2*(3, 4, 6) + (1,
    3, 4, 6, 2) + (1, 3)(4, 6) + (2, 4, 6) + 1/2*(1, 2, 6, 4, 3) - (1, 6)(2, 4, 
    5, 3) - (1, 5, 6, 2, 3) - 1/2*(1, 2, 3, 4, 6) - (1, 3, 2, 6)(4, 5) - (1, 4, 
    3, 6, 2) - (1, 4, 3, 6)(2, 5) - 1/2*(1, 3, 4, 6)(2, 5) - 1/2*(1, 5, 2, 3, 6)
- (2, 6, 3, 4, 5) + (1, 5)(2, 4, 3, 6) - (1, 3, 6)(2, 4, 5) - 1/2*(2, 5)(3, 6) -
(1, 5, 2, 4, 6)
\end{lstlisting}

\begin{lstlisting}[language=C, frame=single, basicstyle=\ttfamily\tiny, escapechar=\@, caption=, title={Lemma41-3\_r4[1].txt}]
Id(G) - (2, 3)(4, 5) + (1, 2, 4, 6, 3) - (1, 6, 5, 3, 2) + 2*(1, 6, 5, 2, 3) - 
    (1, 6, 5)(2, 4, 3) - (1, 5, 2, 3)(4, 6) - (1, 3, 2, 6, 5) + (1, 5)(2, 4, 6, 
    3) - (1, 3, 5, 4, 2) + (1, 4)(3, 6) + (2, 4)(3, 5) - (1, 6, 5)(2, 3, 4) - 
    (2, 5)(3, 4) - (2, 6, 5, 3, 4) + (1, 3, 4)(2, 6, 5) + (1, 6, 4, 2, 5) + (1, 
    5, 2, 4)(3, 6) + (1, 6, 4, 3, 5) + (1, 5, 2)(3, 4, 6) + 2*(2, 3, 6, 5, 4) + 
(1, 3, 6, 5, 4) + (1, 5, 6, 3, 4) - (1, 5)(2, 6, 3, 4) - (1, 3, 6, 4)(2, 5) + 
    (2, 6, 4, 3, 5) + (4, 5, 6) + (1, 4, 6, 5, 3) - (1, 4, 6, 2)(3, 5) - (2, 4, 
    5, 6, 3) - (1, 2, 5, 3)(4, 6) + (2, 5, 6) + (1, 2, 6)(3, 4, 5) - 2*(1, 6, 2,
    4, 3) - (2, 6, 3) - (1, 6)(2, 5, 4, 3) - (3, 4)(5, 6) - (1, 3, 4, 2)(5, 6) -
(2, 3)(5, 6) + (1, 4, 5, 3, 2) + (1, 5, 4, 6, 3) + (1, 3, 2, 4)(5, 6) + (1, 5, 
    4)(2, 6, 3) + 2*(1, 6, 3)(2, 5, 4) + (2, 3, 5, 6, 4) + (1, 2, 4)(3, 6, 5) - 
(1, 6, 3, 4, 2) + (2, 4, 3, 5, 6) - (1, 4)(2, 5, 3, 6) - (1, 5, 4, 2, 6) - (2, 
    5, 6, 3, 4) - (1, 2, 6, 4)(3, 5) + (1, 6)(2, 4) - (1, 5, 4)(2, 3, 6) - (1, 
    4)(2, 3, 5, 6) - 2*(3, 6, 4) + (1, 4, 5, 6, 2) + (1, 6, 2)(3, 5, 4) - (1, 2,
    5, 6, 4) - (1, 4)(2, 6) - (1, 2, 3, 6, 4) + (1, 2)(3, 4) + (1, 3, 2)(4, 6, 
    5) - 2*(1, 6, 4, 5, 3) - (1, 5, 6, 3, 2) + (1, 3, 4, 6, 5) + (1, 6, 4, 5, 2)
+ (1, 3)(2, 4, 6, 5) - (1, 6, 3, 2, 5) + (1, 2)(3, 5, 6, 4) - (1, 3)(2, 5, 4, 6)
+ (1, 4, 3, 2, 6) + (2, 6)(4, 5) - (1, 4, 6, 2, 3) - (1, 6, 3, 4, 5) - (1, 5, 2)
- (1, 6)(2, 3, 4, 5) + (1, 4, 3, 5)(2, 6) + (1, 6, 2, 4, 5) - (1, 2, 3, 5)(4, 6)
- (1, 2)(4, 6) + (1, 6)(3, 5) - 2*(2, 4, 5, 3, 6) - (1, 4, 6, 2, 5) + (1, 6, 2, 
    3, 5) + 2*(2, 5)(3, 6) + (1, 2, 6, 4, 5) - (1, 3, 5, 2, 6)
\end{lstlisting}

\begin{lstlisting}[language=C, frame=single, basicstyle=\ttfamily\tiny, escapechar=\@, caption=, title={Lemma41-3\_r5[5].txt}]
(1, 5, 3, 2)(4, 6) + 4/5*(1, 6, 5, 3)(2, 4) + 3/5*(1, 3)(2, 6, 5, 4) - 16/5*(1, 
    4)(2, 3) - 13/5*(1, 5, 6, 3)(2, 4) + 4/5*(1, 2, 4, 6, 3) - 2/5*(1, 6, 5, 3, 
    2) - 3/5*(1, 5, 6, 4, 3) + 12/5*(2, 4, 5) + 1/5*(1, 2, 5)(3, 4, 6) + 
16/5*(1, 3, 2, 4, 5) + 12/5*(1, 4, 6, 3)(2, 5) - 4/5*(1, 5)(2, 6, 4, 3) - 
16/5*(2, 5, 3) - 12/5*(1, 3)(2, 5) + 3/5*(2, 6, 5) + 3/5*(1, 3, 2, 6, 5) - 
12/5*(1, 4, 5, 2, 3) - 8/5*(1, 5)(2, 4, 6, 3) - 2/5*(1, 2)(3, 6) + 4/5*(1, 4)(3,
    6) - 12/5*(1, 5, 6, 4, 2) - 9/5*(2, 4)(3, 5) + 12/5*(1, 3, 5, 2, 4) - (1, 4,
    2, 5)(3, 6) + 1/5*(1, 6, 5)(2, 3, 4) - 12/5*(2, 5)(3, 4) + 9/5*(1, 3, 4, 2, 
    5) + 3/5*(1, 4)(2, 6, 3, 5) - 4/5*(1, 6, 4, 2, 5) + 3/5*(1, 5, 2, 4)(3, 6) +
4/5*(3, 6, 5) + 3/5*(1, 3, 6, 5, 2) - 13/5*(1, 5, 2)(3, 4, 6) - (2, 3, 6, 5, 4) 
+ 3/5*(1, 3, 6, 5, 4) - 9/5*(1, 4, 2) - 1/5*(1, 6, 4, 2)(3, 5) - 8/5*(1, 3, 6, 
    5)(2, 4) + 3/5*(1, 5)(2, 6, 3, 4) + 8/5*(2, 5, 3, 6, 4) - 12/5*(1, 5, 2, 3, 
    4) + 4/5*(1, 2)(3, 4, 6, 5) + 16/5*(1, 5, 4, 3, 2) - 12/5*(2, 3, 4, 5, 6) - 
1/5*(1, 2, 4, 6)(3, 5) - 3/5*(1, 3, 4, 5, 6) - 9/5*(1, 4, 6, 2)(3, 5) - 4/5*(1, 
    5, 4, 3, 6) + 8/5*(1, 4, 6)(2, 5, 3) + 4/5*(2, 6, 3) - 3/5*(1, 4, 5, 6)(2, 
    3) + 2/5*(1, 6)(2, 5, 4, 3) - 13/5*(1, 3, 4, 2)(5, 6) - 8/5*(1, 5, 4, 2)(3, 
    6) + 3/5*(1, 5, 4, 6, 3) + 3/5*(1, 3, 2, 4)(5, 6) + 3/5*(1, 5, 4)(2, 6, 3) -
3/5*(2, 5, 6, 4, 3) + 4/5*(1, 4)(2, 6, 5, 3) - 2/5*(1, 6, 3, 2, 4) + 9/5*(1, 5, 
    4, 2, 3) - 3/5*(2, 6, 4) + 12/5*(1, 4, 2, 3)(5, 6) + 1/5*(1, 6, 3)(2, 5, 4) 
+ 3/5*(3, 5, 6) + 12/5*(1, 3, 5, 6, 2) - 3/5*(1, 4, 5, 3, 6) + 13/5*(2, 3, 5, 6,
    4) - 2/5*(1, 2, 4)(3, 6, 5) + 4/5*(1, 6, 3, 4, 2) + 12/5*(2, 4, 3, 5, 6) - 
12/5*(1, 3, 5, 6)(2, 4) - 16/5*(1, 4)(2, 5, 3, 6) + 12/5*(2, 5, 6, 3, 4) - 
12/5*(1, 3, 4)(2, 5, 6) - 3/5*(1, 4, 2, 6)(3, 5) - 1/5*(1, 6)(2, 4) + (1, 5, 3, 
    6)(2, 4) - 3/5*(3, 6, 4) + (1, 3, 6, 4, 2) + 3/5*(1, 4)(5, 6) - 16/5*(2, 3, 
    6) + 12/5*(1, 4, 5, 6, 2) + 13/5*(1, 4, 2, 5, 6) - 4/5*(1, 6)(2, 3, 5, 4) + 
8/5*(1, 4, 2, 3, 6) + 12/5*(1, 5, 6, 2, 4) - 8/5*(1, 3, 2)(4, 6, 5) + 3/5*(1, 5,
    6, 3, 2) + 12/5*(2, 3, 4, 6, 5) + 3/5*(1, 3, 4, 6, 5) + 12/5*(1, 4, 3, 5, 2)
+ 3/5*(1, 5)(3, 6) - 13/5*(1, 3)(2, 4, 6, 5) + 2/5*(1, 6, 4, 5)(2, 3) + 13/5*(2,
    5)(4, 6) - 3/5*(2, 6, 4, 5, 3) + 1/5*(1, 2, 3)(4, 6, 5) - 3/5*(1, 3)(2, 6, 
    4, 5) - (1, 4, 6, 5)(2, 3) + 8/5*(2, 3)(4, 6) - 9/5*(2, 4, 6) + (1, 3, 2, 4,
    6) - 12/5*(1, 4, 3)(2, 5, 6) - (2, 5, 4, 6, 3) + 9/5*(1, 3)(2, 5, 4, 6) - 
3/5*(1, 4, 3, 2, 6) - 12/5*(1, 5, 6, 2, 3) - 4/5*(1, 3, 2, 6)(4, 5) - 3/5*(3, 
    5)(4, 6) - 12/5*(1, 3, 5, 2)(4, 6) + 12/5*(1, 5, 2) + 2/5*(1, 2, 5)(3, 6, 4)
- 3/5*(1, 4, 3, 6)(2, 5) - 4/5*(1, 2, 3, 5)(4, 6) + 13/5*(1, 4, 6)(2, 3, 5) + 
16/5*(1, 5)(2, 4, 3, 6) - 1/5*(1, 2)(4, 6) + 16/5*(1, 3, 6, 2)(4, 5) - 3/5*(1, 
    4, 6) - 3/5*(2, 3, 6, 4, 5) - 4/5*(1, 3, 6, 4, 5) - 4/5*(1, 2, 5, 4, 6) - 
3/5*(1, 5, 2, 6)(3, 4) + 9/5*(1, 5)(2, 3, 4, 6) + 2/5*(1, 2, 3, 6)(4, 5) - 
12/5*(1, 5, 2, 4, 6)
\end{lstlisting}

%% file: ZZpG_Output_4-2.tex

\lstinputlisting
[language=C, frame=single, basicstyle=\ttfamily\tiny, escapechar=\@, caption=, title={Lemma42-2\_r0.txt}]
{Output/4-2/Lemma42-2_r0.txt}

\lstinputlisting
[language=C, frame=single, basicstyle=\ttfamily\tiny, escapechar=\@, caption=, title={Lemma42-3\_r1.txt}]
{Output/4-2/Lemma42-3_r1.txt}

\begin{lstlisting}[language=C, frame=single, basicstyle=\ttfamily\tiny, escapechar=\@, caption=, title={Lemma42-3\_rr1.txt}]
-1/40*(1, 3, 6, 2, 5) - 1/40*(1, 6)(2, 4, 3, 5) - 1/40*(2, 6)(3, 5) - 1/40*(1, 
    6, 2, 5)(3, 4)
\end{lstlisting}

\lstinputlisting
[language=C, frame=single, basicstyle=\ttfamily\tiny, escapechar=\@, caption=, title={Lemma42-3\_r2.txt}]
{Output/4-2/Lemma42-3_r2.txt}

\lstinputlisting
[language=C, frame=single, basicstyle=\ttfamily\tiny, escapechar=\@, caption=, title={Lemma42-3\_r3.txt}]
{Output/4-2/Lemma42-3_r3.txt}

\lstinputlisting
[language=C, frame=single, basicstyle=\ttfamily\tiny, escapechar=\@, caption=, title={Lemma42-3\_r4.txt}]
{Output/4-2/Lemma42-3_r4.txt}

\lstinputlisting
[language=C, frame=single, basicstyle=\ttfamily\tiny, escapechar=\@, caption=, title={Lemma42-3\_r5.txt}]
{Output/4-2/Lemma42-3_r5.txt}

\lstinputlisting
[language=C, frame=single, basicstyle=\ttfamily\tiny, escapechar=\@, caption=, title={Lemma42-3\_r6.txt}]
{Output/4-2/Lemma42-3_r6.txt}

\lstinputlisting
[language=C, frame=single, basicstyle=\ttfamily\tiny, escapechar=\@, caption=, title={Lemma42-3\_r7.txt}]
{Output/4-2/Lemma42-3_r7.txt}

\lstinputlisting
[language=C, frame=single, basicstyle=\ttfamily\tiny, escapechar=\@, caption=, title={Lemma42-3\_r8.txt}]
{Output/4-2/Lemma42-3_r8.txt}

\lstinputlisting
[language=C, frame=single, basicstyle=\ttfamily\tiny, escapechar=\@, caption=, title={Lemma42-3\_r9.txt}]
{Output/4-2/Lemma42-3_r9.txt}